\documentclass[amssymb,amsfonts,reqno]{amsart}

\usepackage{amssymb,latexsym}
\usepackage{mathrsfs}
\usepackage{xypic}
\usepackage{verbatim}
\usepackage[colorlinks,
linkcolor=blue,
anchorcolor=blue,
citecolor=blue]{hyperref}

\numberwithin{equation}{section}

\renewcommand{\L}{{\mathcal L}}

\usepackage[numbers,sort&compress]{natbib}
\allowdisplaybreaks \allowdisplaybreaks[4]
\include{psfig}

\begin{document}

\title[Besov spaces on Grushin spaces ]{Geometric topics   related to  Besov type spaces on the Grushin setting}

\author[N. Zhao]{Nan Zhao}
\address{School of Mathematics and Physics, University of Science and Technology Beijing, Beijing 100083, China}
\email{zhaonanmath@126.com}

\author[Z. Wang]{Zhiyong Wang}
\address{School of Mathematics and Physics, University of Science and Technology Beijing, Beijing 100083, China}
\email{zywang199703@163.com}

\author[P. Li]{Pengtao Li}
\address{School of Mathematics and Statistics, Qingdao University, Qingdao, 266071, China}
\email{ptli@qdu.edu.cn}\thanks{P.T. Li was supported by the National
Natural Science Foundation of China  (No.\,12071272) and the
Shandong Natural Science Foundation of China (No.\,ZR2020MA004). }

\author[Y. Liu]{Yu Liu}
\address{School of Mathematics and Physics, University of Science and Technology Beijing, Beijing 100083, China}
\email{liuyu75@pku.org.cn}
\thanks{Y. Liu was supported by  the Beijing Natural Science
Foundation of China (No.\,1232023),  the National Natural Science
Foundation of China (No.\,12271042) and    the National Science and
Technology Major Project of China (No. J2019-I-0019-0018, No.
J2019-I-0001-0001).}
\thanks {The corresponding author
is Yu Liu, E-mail: liuyu75@pku.org.cn}


\maketitle

\begin{abstract}

The Grushin spaces, as one of the most
important models in the Carnot-Carath\'eodory space,    are a class
of locally compact and geodesic metric spaces which admit a
dilation. Function spaces on Grushin spaces and some related
geometric problems are always the research hotspots in this field.
Firstly, we investigate two classes of Besov type spaces based on
the Grushin semigroup and the fractional Grushin semigroup,
respectively, and prove some important properties of these two Besov
type spaces. Moreover, we also reveal the relationship between them.
Secondly, we establish the isoperimetric inequality for the
fractional perimeter, which is defined by the Grushin-Laplace
operator on Grushin spaces. Finally, we combine the semigroup theory
with a nonlocal calculus for the Grushin-Laplace operator to obtain
the Sobolev type inequality. As a corollary, we also obtain the
embedding theorem for Besov type spaces.
\end{abstract}

\hspace{0.2cm}{\Small {\bf Keywords:}{ Besov space, Grushin space,
fractional perimeter, isoperimetric inequality, Sobolev type
inequality.}}

\hspace{0.2cm}{\Small {\bf 2010 Mathematics Subject Classification:}
{30H25, 53C17,
28A75, 35R11}}

\tableofcontents
 \pagenumbering{arabic}
\numberwithin{equation}{section}
\newtheorem{theorem}{Theorem}[section]
\newtheorem{lemma}[theorem]{Lemma}
\newtheorem{definition}[theorem]{Definition}
\newtheorem{corollary}[theorem]{Corollary}
\newtheorem{proposition}[theorem]{Proposition}
\newtheorem{claim}[theorem]{Claim}
\newtheorem{remark}[theorem]{Remark}
\newtheorem{assumption}[theorem]{Assumption}
\newtheorem{example}[theorem]{Example}
\allowdisplaybreaks

\section{Introduction}
Besov spaces play an important role in the theory of partial differential equations,
 and can also be regarded as a generalization of Sobolev spaces. In the literature, many scholars
 have shown great interest in Besov type spaces under different frameworks,
 see \cite{ABC2020,CG2022,GT2020,GL2015,P2010} and the references
 therein. Recently, more and more scholars pay attention to some geometric topics related to Besov type spaces,
 such as fractional perimeters, isoperimetric inequalities,  and so on.
As a continuation of the previous research, we focus on Besov type
spaces and some related geometric problems on the Grushin setting.

As we know, isoperimetric inequalities occupy a central position in
many areas of mathematics such as geometry, linear and nonlinear
PDEs, and probability theory. In the Euclidean setting, the classical
isoperimetric inequality states that for all measurable sets
$E\subset\mathbb{R}^{n},$ one has
\begin{equation}\label{CII}
P(E)\geq n\omega_{n}^{1/n}|E|^{(n-1)/n},
\end{equation}
where $\omega_{n}$ is defined as the $n$-dimensional Lebesgue measure
of the unit ball in $\mathbb{R}^{n}$.
Here $P(E)$ denotes the distributional
perimeter of $E$ which coincides with the $(n-1)$-dimensional measure of
$\partial E$ when $E$ has a smooth boundary. It should be mentioned that
De Giorgi \cite{De1954,De1958} first provided a complete proof of (\ref{CII}) by exploiting the concept of perimeter.
 Specifically, he proved that equality holds in (\ref{CII})  if and only if $E$ is a ball.
Later, these results have sparked intense research interest and have
been used and extended in many different directions. For instance,
we refer to (\cite{A1998,CDG1994,CDPT2007,F2003,GN1996,MM}) for the
  Carnot-Carath\'eodory spaces, (\cite{bor,car}) for the
  Gauss spaces and (\cite{sch}) for  the
hyperbolic spaces.

As an alternative to classical perimeters, fractional perimeters
started to gain attraction in study of nonlocal minimal surfaces,
phase transitions, fractal sets and many other problems over the
previous decade (see for example \cite{maz,val,V1991}). Specially,
Caffarelli, Roquejoffre and Savin \cite{CRS2010} considered the
following nonlocal perimeter in the study of problems related to the
nonlocal minimal surface. For any $s\in(0,1/2),$ a measurable set
$E\subset\mathbb{R}^{n}$ is said to have finite $s$-perimeter if
\begin{align}\label{Ps}
P_{s}(E):=\int_{\mathbb{R}^{n}}\int_{\mathbb{R}^{n}}\frac{|\mathbf{1}_{E}(x)-\mathbf{1}_{E}(y)|^{2}}{|x-y|^{n+2s}}
dxdy=2\int_{\mathbb{R}^{n}\setminus E}\int_{E}\frac{dxdy}{|x-y|^{n+2s}}<\infty.
\end{align}
It is worth noting that (\ref{Ps}) is equivalent to
$\mathbf{1}_{E}\in W^{s,2}(\mathbb{R}^{n}),$ where
$W^{s,p}(\mathbb{R}^{n})$ is defined as the Banach space of
functions $u\in L^{p}(\mathbb{R}^{n})$ with finite
Aronszajn-Gagliardo-Slobedetzky semi-norm, that is,
\begin{align}\label{AGS}
[u]_{s,p}:=\bigg(\int_{\mathbb{R}^{n}}\int_{\mathbb{R}^{n}}\frac{|u(x)-u(y)|^{p}}{|x-y|^{n+ps}}dxdy\bigg)^{1/p}<\infty
\end{align}
for $p\in[1,\infty)$ and $s>0.$ This implies that
\begin{align}\label{AGSP}
P_{s}(E)=[\mathbf{1}_{E}]^{2}_{s,2}=[\mathbf{1}_{E}]_{2s,1}.
\end{align}
 Using a symmetrization result established by Almgren and Lieb
\cite{AL1989}, we can see that
the nonlocal version of the classical isoperimetric inequality (\ref{CII}) is in fact true.
More precisely,
given $s\in(0,1/2),$ there exists a positive constant $C,$ depending on $n$ and $s,$
such that for any measurable set $E\subset\mathbb{R}^{n}$ with $|E|<\infty,$
\begin{align}\label{nonlocal}
P_{s}(E)\geq C|E|^{(n-2s)/n}.
\end{align}
Then Frank and Seiringer \cite{FS2008} proved that equality holds in
(\ref{nonlocal}) precisely for balls (up to sets of measure zero). A
stability version has been recently established by Fusco, Millot,
and Morini \cite{FMM2011}. Inspired by Ledoux's powerful and
flexible semigroup method for isoperimetric inequalities in the
local case in \cite{L1994}, many scholars have made profound researches
on (\ref{nonlocal}) from different perspectives.
For instance, Garofalo and Tralli \cite{GT2020} proved the  nonlocal
isoperimetric inequality such as (\ref{nonlocal}) adapted to a class
of Kolmogorov-Fokker-Planck operators which are of interest in
 analysis, physics and the
applied sciences. These operators are degenerate and do not possess
a variational structure. Recently, Alonso-Ruiz et al. \cite{ABC2020}
introduced the heat semigroup-based Besov classes in the general
framework of Dirichlet spaces and established isoperimetric
inequalities such as (\ref{nonlocal}).
 We also refer the reader to \cite{ABCRST2020,ABCRST2021,MPPP2007,P2004} for more
insights on such a heat-kernel approach to perimeters and isoperimetric properties.

 Grushin spaces are a important model in
Carnot-Carath\'eodory spaces and it appears in the theory of
hypoelliptic operators (see \cite{G1970,G1971,W}).  Let
$\mathbb{R}^{n}=\mathbb{R}^{m}\times\mathbb{R}^{k},$ where
$m,k\geq1$ are integers and $n=m+k.$ Let $\alpha\geq0$. A Grushin
space, denoted as $\mathbb{G}_{\alpha}^{n}$, is the Euclidean space
$\mathbb{R}^{n}$ endowed with the Carnot-Carath\'{e}odory distance
$d_{\alpha}$ associated to the family of vector fields
$$X_{\alpha}=\{X_{1},X_{2},\ldots,X_{m},X_{m+1},X_{m+2},\ldots,X_{m+k}\}
$$ with
\begin{equation}\label{vector}
\left\{ \begin{aligned}
         &X_{i}=\partial_{x_{i}},&\ \ i=1,2,\ldots,m; \\
         &X_{m+j}=|x|^{\alpha}\partial_{y_{j}},&\ \ j=1,2,\ldots,k,
\end{aligned} \right.
\end{equation}
where $|x|$ is the standard norm of $x\in\mathbb{R}^{m}.$
Especially, if
$m=k=1,\mathbb{G}_{\alpha}^{2}:=(\mathbb{R}^{2},d_{\alpha})$ is
called the Grushin plane. On the Grushin plane, Monti and Morbidelli
\cite{MM} obtained  the sharp constant and extremal sets for the
isoperimetric inequality such as (\ref{CII}). Subsequently,
Franceschi and Monti \cite{FM2016} studied the isoperimetric problem
on Grushin spaces (regarded as the high-dimensional case of the
Grushin plane). Via a symmetry assumption that depends on the
dimension, they proved the existence, additional symmetry, and
regularity of an isoperimetric set on Grushin spaces. For a more
detailed introduction to Grushin spaces, please refer to Section
\ref{sec-2}.

In what follows, we recall the following differential
operators on Grushin spaces,
\begin{equation}\label{K}
  \mathcal{L}:=-\Delta_{x}-|x|^{2\alpha}\Delta_{y}
\end{equation}
where $\Delta_{x}$ and $\Delta_{y}$ denote Laplace operators with the
variables $x\in\mathbb{R}^{m}$ and $y\in\mathbb{R}^{k},$
respectively. The differential operator $\mathcal{L}$ is known in
the literature as the Grushin-Laplace operator (see \cite{G1970}),
and it is hypoelliptic for $\alpha\in\mathbb{N}.$ Franchi and
Lanconelli \cite{FL} proved the H\"older regularity of the weak
solutions of $\mathcal{L}u=0$ by using Moser's technique.

The first aim of this paper is to investigate Besov type spaces in Grushin settings. In Section \ref{sec-3}, we introduce two classes of Besov type spaces
$B_{p,q}^{\mathcal{L},\beta}(\mathbb{G}^{n}_{\alpha})$ and
$B_{p,q}^{\L^s,\beta}(\mathbb{G}^{n}_{\alpha})$ for
$(p,q,\beta,s)\in[1,\infty)\times[1,\infty]\times(0,\infty)\times(0,1),$ which are
 generated
by the semigroup $\{e^{-t\L}\}_{t>0}$ and the fractional semigroup
$\{e^{-t\L^s}\}_{t>0}$, respectively. We first prove several basic properties of $B_{p,q}^{\mathcal{L},\beta}(\mathbb{G}^{n}_{\alpha})$ and
$B_{p,q}^{\L^s,\beta}(\mathbb{G}^{n}_{\alpha})$. Then we prove that
$B_{p,q}^{\mathcal{L},\beta}(\mathbb{G}^{n}_{\alpha})$ and
$B_{p,q}^{\L^s,\beta}(\mathbb{G}^{n}_{\alpha})$ coincide with the Besov type spaces $B^{\beta}_{p,q}(\mathbb G^{n}_{\alpha})$ defined via the difference. Theorems \ref{com-1}\ \&\ \ref{com-2} indicate that for $u\in L^{p}(\mathbb G^{n}_{\alpha})$, the following three statements are equivalent:
$$\left\{\begin{aligned}
&\bigg(\int_{0}^{\infty}\Big(\int_{\mathbb{G}^{n}_{\alpha}}\int_{B(g,t)}\frac{|u(g)-u(g')|^{p}}{t^{2\beta p}|B(g,t)|}dg'dg\Big)^{q/p}\frac{dt}{t}\bigg)^{1/q}<\infty;\\
&\bigg(\int_{0}^{\infty}
\Big(\int_{\mathbb{G}^{n}_{\alpha}}e^{-t\mathcal{L}}(|u-u(g)|^{p})(g)dg\Big)^{q/p}\frac{dt}{t^{{\beta q}+1}}\bigg)^{1/q}<\infty;\\
&\bigg(\int_{0}^{\infty}
\Big(\int_{\mathbb{G}^{n}_{\alpha}}e^{-t\mathcal{L}^{s}}(|u-u(g)|^{p})(g)dg\Big)^{q/p}\frac{dt}{t^{{\beta q}/{s}+1}}
\bigg)^{1/q}<\infty.
\end{aligned}\right.$$

The second aim of this paper is to establish isoperimetric
inequalities for fractional perimeters $P^{\mathcal{L}}_{s},
P^{\mathcal{L},*}_{s}$ and $P^{\mathcal{L}}_{s,\infty}$ on the
Grushin setting, respectively. For the definitions of these three
perimeters, please refer to Definitions \ref{perimeter},
\ref{perimeter1} \& \ref{FP3}. Different from \cite{FM2016,MM}, we
define the fractional perimeter mainly by means of the interplay
between the semigroup $\{e^{-t\L}\}_{t>0}$ and the Grushin-Laplace
operator $\L$, which provides a different perspective to   nonlocal
interactions and the minimal surfaces. Notice that any non-empty
bounded open set $E$ has infinite $s$-perimeter as soon as $s\geq
1/2.$ This fact explains the restriction on the range of $s$ (see
\cite[Remark 4.2]{GT2020}, or the explicit
              constant in \cite[Proposition 1.1]{G2020}).

\begin{theorem}\label{main1}
Let $\mathcal{L}$ be as in (\ref{K}) and $s\in(0,1/2).$ For any
$E\subset \mathbb{G}_{\alpha}^{n}$ with finite $s$-perimeter, there exists a positive constant
$C(Q,s),$ depending on $Q$ and $s,$ such that
\begin{align}\label{Per}
\mathscr{P}^{\mathcal{L}}_{s}(E)\geq C(Q,s)|E|^{(Q-2s)/Q},
\end{align}
where $Q$ is the homogeneous dimension of
$\mathbb{G}_{\alpha}^{n}$ and $\mathscr{P}^{\mathcal{L}}_{s}$ is $P^{\mathcal{L}}_{s}, P^{\mathcal{L},*}_{s}$
or $P^{\mathcal{L}}_{s,\infty}.$
\end{theorem}

For the classical Sobolev spaces, we know that $$W^{1,1}(\mathbb{R}^{n})\hookrightarrow L^{n/(n-1),\infty}(\mathbb{R}^{n}).$$
The weak Sobolev embedding above actually implies that the isoperimetric inequality (\ref{CII}) holds.
Combining this with the coarea formula, we conclude that the strong embedding
$$W^{1,1}(\mathbb{R}^{n})\hookrightarrow L^{n/(n-1)}(\mathbb{R}^{n}).$$
This establishes the fact that when $p=1$, the weak Sobolev embedding is equivalent to the strong Sobolev embedding,
and they are both equivalent to the isoperimetric inequality. For more details on this aspect, we can refer to \cite{M2010}.

It should be noted that  the isoperimetric inequality (\ref{Per}) in
Theorem \ref{main1}  can be applied to a class of Besov type spaces
$\mathfrak{B}_{p,q}^{\mathcal{L},\beta}(\mathbb{G}^{n}_{\alpha})$
related to the Grushin-Laplace operator $\mathcal{L}$, which are the
closed subspaces of
$B_{p,q}^{\mathcal{L},\beta}(\mathbb{G}^{n}_{\alpha})$ (see
Definition \ref{Besov2} below). To be precise, we have the following endpoint results
for Besov type spaces $\mathfrak{B}_{p,q}^{\mathcal{L},2s}(\mathbb{G}_{\alpha}^{n})$ with $p=q=1.$
\begin{theorem}\label{main2}
Let $\mathcal{L}$ be as in (\ref{K}) and $s\in(0,1/2).$ Then
\begin{align}\label{BEG}
\mathfrak{B}_{1,1}^{\mathcal{L},2s}(\mathbb{G}_{\alpha}^{n})\hookrightarrow L^{Q/(Q-2s)}(\mathbb{G}_{\alpha}^{n}).
\end{align}
To be precise, for every $u\in \mathfrak{B}_{1,1}^{\mathcal{L},2s}(\mathbb{G}_{\alpha}^{n}),$ we have
\begin{align}\label{BEG1}
\|u\|_{L^{Q/(Q-2s)}(\mathbb{G}_{\alpha}^{n})}\leq C(Q,s)^{-1}N_{1,1}^{\mathcal{L},2s}(u),
\end{align}
where $C(Q,s)$ is the constant appearing in Theorem \ref{main1}.
Moreover, in Theorem \ref{main1}, if we take $\mathscr{P}^{\mathcal{L}}_{s}=P^{\mathcal{L},*}_{s},$ then (\ref{Per}) and (\ref{BEG1}) are equivalent.
\end{theorem}
We are in the position to pass to the Sobolev type inequality which
is a widely studied topic in analysis and PDEs. The classical
Sobolev inequality states that for any $p\in[1,n),$ there exists a
constant $C$ depending on $n$ and $p$ such that for any function $u$
in the Schwartz class $\mathscr{S}(\mathbb{R}^{n}),$ one has
\begin{align}\label{Sobolev1}
\|u\|_{L^{q}(\mathbb{R}^{n})}\leq C\|\nabla u\|_{L^{p}(\mathbb{R}^{n})},
\end{align}
where $q=np/(n-p).$ This is a fundamental result in analysis and it
has been widely studied in a variety of contexts (see e.g. the
classical references \cite{M2010} and \cite{S2002}).

The inequality (\ref{Sobolev1}) has also been appropriately extended
to the Carnot-Carath\'eodory spaces in
\cite{CDG1994,CDPT2007,GN1996}. Recently, Garofalo and Tralli
\cite{GT2021} have been influenced by the ideas of Stein in
\cite{S1970} and Varopoulos in \cite{V1985} in the setting of
positive symmetric semigroups to establish Sobolev type inequalities
similar to (\ref{Sobolev1}) for a class of hypoelliptic operators.
Alonso-Ruiz et al. \cite{ABC2020} also proved some Sobolev type
inequalities in the general framework of Dirichlet spaces. It is
worth noting that the results mentioned above are mainly based on
the properties of operator semigroups. For more applications of the
heat kernel methods in Sobolev inequalities, we also refer to
\cite{ABCRST2020,ABCRST2021,L1994} and the references therein.

 Motivated by \cite{ABC2020,ABCRST2020,ABCRST2021,GT2021}, the third aim of
this paper is to introduce Sobolev type spaces
$\mathcal{W}^{2s,p}(\mathbb{G}_{\alpha}^{n})$ (see
Definition \ref{STS} below) related to the fractional powers $\L^{s}$
and establish Sobolev type inequalities on Grushin spaces by
combining semigroup theory with nonlocal calculus of Grushin-Laplace
operators.

\begin{theorem}\label{Sobolev}
Let $(s,p)\in(0,1)\times[1,Q/2s)$ and $1/p-1/q=2s/Q.$
Then the following statements hold:
\begin{itemize}
  \item [(i)] If $p=1,$ then $\mathcal{W}^{2s,1}(\mathbb{G}_{\alpha}^{n})\hookrightarrow L^{Q/(Q-2s),\infty}(\mathbb{G}^{n}_{\alpha}).$ More precisely,
  there exists a positive constant $C(Q,s),$ depending on $Q$ and $s,$ such that for any $u\in \mathscr{S}(\mathbb{G}^{n}_{\alpha})$ we have
  $$\sup_{\lambda>0}\lambda|\{g\in\mathbb{G}^{n}_{\alpha}:|u(g)|>\lambda\}|^{(Q-2s)/Q}
  \leq C(Q,s)\|\mathcal{L}^{s}u\|_{L^{1}(\mathbb{G}^{n}_{\alpha})}.$$
  \item [(ii)] If $p\in(1,Q/2s),$ then $\mathcal{W}^{2s,p}(\mathbb{G}_{\alpha}^{n})\hookrightarrow L^{pQ/(Q-2s p)}(\mathbb{G}^{n}_{\alpha}).$ More precisely,
   there exists a positive constant $C(Q,p,s),$ depending on $Q,p$ and $s,$ such that for any $u\in \mathscr{S}(\mathbb{G}^{n}_{\alpha})$ we have
  $$\|u\|_{L^{q}(\mathbb{G}^{n}_{\alpha})}\leq C(Q,p,s)\|\mathcal{L}^{s}u\|_{L^{p}(\mathbb{G}^{n}_{\alpha})}.$$
\end{itemize}
\end{theorem}
As a corollary to Theorem \ref{Sobolev}, we have the following embedding relationship for Besov type spaces.
\begin{corollary}\label{propo2.6}
Let $s\in(0,1).$ The following statements hold:
\begin{itemize}
  \item [(i)] If $p=q\in(1,Q/2s)$ and $\beta>2s$ then $$B_{p,p}^{\mathcal{L},\beta}(\mathbb{G}^{n}_{\alpha})\hookrightarrow L^{pQ/(Q-2s p)}(\mathbb{G}^{n}_{\alpha}).$$
                  In particular, when $p=q=1$ and $\beta\geq2s,$ we have
                  $$B_{1,1}^{\mathcal{L},\beta}(\mathbb{G}_{\alpha}^{n})\hookrightarrow L^{Q/(Q-2s),\infty}(\mathbb{G}^{n}_{\alpha}).$$
  \item [(ii)] If $p\in(1,Q/2s),q=\infty$ and $\beta>2s$ then
  $$B_{p,\infty}^{\mathcal{L},\beta}(\mathbb{G}^{n}_{\alpha})\hookrightarrow L^{pQ/(Q-2s p)}(\mathbb{G}^{n}_{\alpha}).$$
   If $p=1,q=\infty$ and $\beta>2s$ then
   $$B_{1,\infty}^{\mathcal{L},\beta}(\mathbb{G}_{\alpha}^{n})\hookrightarrow L^{Q/(Q-2s),\infty}(\mathbb{G}^{n}_{\alpha}).$$
\end{itemize}
\end{corollary}
\subsection{Structure of the paper}
This article is organized as follows.

$\bullet$ Sections \ref{sec-2.1} \& \ref{sec-2.2} deal with some basic concepts of Grushin spaces and some
known facts and results of the semigroup
$\{e^{-t\mathcal{L}}\}_{t>0}$ generated by the Grushin-Laplace
operator $\mathcal{L}.$ Here, we also obtain some ultracontractivity
properties of $\{e^{-t\mathcal{L}}\}_{t>0}$,  which are used to
establish Ledoux type estimates and some limiting behaviour of Besov
semi-norms, respectively (see Propositions \ref{lem4} \&\
\ref{ulc}).

Moreover, in Section \ref{sec-2.3}, via Balakrishnan's formula
(\ref{-G}), we can precisely identify the fractional Grushin-Laplace
operator $\mathcal{L}^{s}$ by means of
$\{e^{-t\mathcal{L}}\}_{t>0}$. This allows us to introduce Sobolev
type spaces $\mathcal{W}^{2s,p}(\mathbb{G}_{\alpha}^{n})$ defined by
$\mathcal{L}^{s} $ (see Definition \ref{STS}). Then, we prove that $\mathscr{S}(\mathbb{G}_{\alpha}^{n})$ is dense in $\mathcal{W}^{2s,p}(\mathbb{G}_{\alpha}^{n})$ (see Proposition \ref{W-de}).

$\bullet$
In Section \ref{sec-3.1}, we first establish some basic properties of $B_{p,q}^{\mathcal{L},\beta}(\mathbb{G}^{n}_{\alpha})$,
such as completeness, nontriviality and the min-max property (see Propositions  \ref{RB}, \ref{desi}   and Lemma \ref{max}).
Additionally, we provide some inclusion relationships between
Sobolev type spaces $\mathcal{W}^{2s,p}(\mathbb{G}_{\alpha}^{n})$
and Besov type spaces
$B_{p,q}^{\mathcal{L},\beta}(\mathbb{G}^{n}_{\alpha})$ under certain
indicator constraints  (see Corollary  \ref{BW}).

Then, in Section \ref{sec-3.2}, by the subordinative formula
(\ref{k-1}), we also introduce another class of Besov type spaces
$B_{p,q}^{\mathcal{L}^{s},\beta}(\mathbb{G}^{n}_{\alpha})$ which are
related to $\{e^{-tL^{s}}\}_{t>0}$ with $s\in(0,1)$. By comparing
the inclusion relation with the Besov space
$B_{p,q}^{\beta}(\mathbb{G}^{n}_{\alpha})$ defined via the
difference, we can see that the difference between
$B_{p,q}^{\mathcal{L},\beta}(\mathbb{G}^{n}_{\alpha})$ and
$B_{p,q}^{\mathcal{L}^{s},\beta}(\mathbb{G}^{n}_{\alpha})$ (see
Theorems \ref{com-1} \& \ref{com-2}).

Finally, we analyze some limiting behaviour of Besov semi-norms in
Section \ref{sec-3.3}. Our results, Theorem \ref{MS1} and
Proposition \ref{BBM} generalize a previous one of
Maz'ya-Shaposhnikova and Bourgain-Brezis-Mironescu for the classical
fractional Sobolev spaces on the Grushin setting.

$\bullet$ In Section \ref{sec-4}, we first introduce  the notion of
the fractional bounded variation function and the fractional
perimeter on Grushin spaces, which depend on the fractional power
$\mathcal{L}^{s}$. See Definitions \ref{FBV} \&\ \ref{perimeter},
respectively. Then inspired by that originally given by Caffarelli,
Roquejoffre and Savin in \cite{CRS2010}, we also define a second
notion of the fractional perimeter $P_{s}^{\mathcal{L},*},$ which
corresponded to the semi-norm of Besov type spaces (see Definition
\ref{perimeter1}). With such notions in hands, we provide a
connection between these two fractional perimeters in Proposition
\ref{com}.

In Section \ref{sec-4.1}, we prove Theorem \ref{main1} by using a
key Ledoux type estimate for Grushin semigroups (see Proposition
\ref{Led}). As an application, with the help of Theorem \ref{main1}
and the coarea formula (\ref{coarea}),
we prove Theorem \ref{main2} in Section \ref{sec-4.2}.
 As a corollary of
Theorem \ref{main1}, the isoperimetric inequality similar to
(\ref{Per}) holds for $P_{s}^{\mathcal{L},*},$ and the isoperimetric
inequality for $P_{s}^{\mathcal{L},*}$ is equivalent to Theorem
\ref{main2}.

$\bullet$ In Section \ref{sec-4.2}, we obtain the Sobolev embedding
on Grushin spaces (see Theorem \ref{Sobolev}). Our strategy follows
the classical approach to the subject. We first establish the key
Hardy-Littlewood-Sobolev type result (see Proposition \ref{HLS}).
With such tools in hands, we are easily able to prove Theorem
\ref{Sobolev}.

%

\subsection{Notation}
Throughout this paper,  all notations will be listed  as needed.

\begin{itemize}
  \item If $E$ is a measurable subset of $\mathbb{G}_{\alpha}^{n},$ then $\mathbf{1}_{E}$ denotes the characteristic function of $E.$
  \item We denote by $C^{1}(\mathbb{G}_{\alpha}^{n})$ the space of all functions that are continuous together
    with their partial derivatives up to order 1,
    by $C_{c}^{\infty}(\mathbb{G}_{\alpha}^{n})$ the space of all smooth functions on $\mathbb{G}_{\alpha}^{n}$ with compact
    support, and by $\mathscr{S}(\mathbb{G}_{\alpha}^{n})$ the Schwartz class on $\mathbb{G}_{\alpha}^{n}.$
  \item The symbol $\thicksim$ between two positive expressions $u,v$ means that their ratio
  $\frac{u}{v}$ is bounded from above and below by positive constants.
  \item The symbol $\lesssim$ (respectively $\gtrsim$)  between two nonnegative expressions $u,v$
  means that there exists a constant $C>0$ such that $u\leq Cv$ (respectively $u\geq Cv$).
\end{itemize}

\section{Preliminaries}\label{sec-2}
\subsection{Grushin spaces and their metrics}\label{sec-2.1}
Let $\mathbb{R}^{n}=\mathbb{R}^{m}\times\mathbb{R}^{k},$
where $m,k\geq1,n=m+k$ are integers. For a given number
$\alpha\geq0,$ define a family of vector fields on $\mathbb{R}^{n}$ by
(\ref{vector}).
These vector fields induce the following distance between two points
$g,g'$ in $\mathbb{R}^{n}:$
\begin{align*}d_{\alpha}(g,g')&:=\inf\{\delta\mid \exists\ \mathrm{Lipschitz}\ \mathrm{continuous}\
 \mathrm{curve}\ \gamma:[0,\delta]\rightarrow \mathbb{R}^{n}\mid \gamma(0)=g,\gamma(\delta)=g'
\ \mathrm{and} \ \\ &\gamma'(t)=\sum^n_{i=1}a_i(t)X_i(\gamma(t))\
 \mathrm{with}\
 \sum^n_{i=1}|a_i(t)|^2\leq 1, \mathrm{for} \ \mathrm{all}\ t\in[0,\delta] \}.
\end{align*}
From \cite{F2016}, $d_{\alpha}(g,g')$ is well defined and coincides with the Carnot-Carath\'{e}odory
distance-namely-one has:
\begin{equation*}
d_{\alpha}(g,g')=\inf_{\gamma=(x,y)\in\Gamma_{g,g'}}\int_{0}^{1}\bigg(\sum_{i=1}^{m}|
\dot{x_{i}}(t)|^{2}+|x(t)|^{-2\alpha}\sum_{j=1}^{k}|\dot{y_{j}}(t)|^{2}\bigg)^{1/2}dt,
\end{equation*}
where $\Gamma_{g,g'}$ is the set of all Lipschitz continuous curves
$$\gamma:[0,1]\rightarrow\mathbb{R}^{n}\ \mbox{with}\ \gamma(0)=g\ \mbox{and}\ \gamma(1)=g',$$
  $\dot{x}_i(t):=dx_i/dt $ and $\dot{y}_j(t):=dy_j/dt.$ The
resulting space
$\mathbb{G}_{\alpha}^{n}:=(\mathbb{R}^{n},d_{\alpha})$, which is
called a Grushin space, is the completion of the Riemannian metric
space $\{(x,y)\in\mathbb{R}^{n}:x\neq 0\}$ equipped with the
Riemannian metric $dw^{2}=dx^{2}+|x|^{-2\alpha}dy^{2}.$

The Grushin space $\mathbb{G}_{\alpha}^{n}$  can be endowed with a
family of non-isotropic dilations parametrized by $\lambda>0$
\begin{align*}\label{anisotropic}
\delta^{\alpha}_{\lambda}g=(\lambda x,\lambda^{\alpha+1}y),~~g=(x,y)\in \mathbb{G}_{\alpha}^{n}
\end{align*}
such that
\begin{align*}
d_{\alpha}(\delta^{\alpha}_{\lambda}g,\delta^{\alpha}_{\lambda}g')=\lambda d_{\alpha}(g,g')
\end{align*}
for $g=(x,y),g'=(x',y')\in \mathbb{G}_{\alpha}^{n}.$

 The measure on
$\mathbb{G}_{\alpha}^{n}$ is the usual Lebesgue measure $dg=dxdy.$
For any measurable set
$E\subseteq\mathbb{G}_{\alpha}^{n},|\delta^{\alpha}_{\lambda}(E)|=\lambda^{Q}|E|
$ (see~\cite{MM}), where $$Q=m+(\alpha+1)k=n+\alpha k$$ is called the homogeneous dimension of
$\mathbb{G}_{\alpha}^{n}.$

Let $B(g,r)=\{g'\in\mathbb{G}_{\alpha}^{n}:d_{\alpha}(g,g')<r\}$
denote the ball with center $g$ and radius $r>0$ in the metric
$d_{\alpha}(g,g')$ and let $|B(g,r)|$ be its Lebesgue measure.
Then
\begin{equation}\label{B}
|B(g,r)|\thicksim r^{n}(r+|x|)^{k\alpha}.
\end{equation}
According to \cite{FL} there are two positive constants
$C_{1}<C_{2}$ such that
\begin{equation}\label{Q}
Q_{\alpha}(g,C_{1}r)\subseteq B(g,r)\subseteq Q_{\alpha}(g,C_{2}r)
\end{equation}
for all $g=(x,y)\in\mathbb{G}_{\alpha}^{n},$ where the boxes
$$Q_{\alpha}(g,r)=\prod_{i=1}^{m}[x_{i}-r,x_{i}+r]\times\prod_{j=1}^{k}[y_{j}-r(|x|+r)^{\alpha},y_{j}+r(|x|+r)^{\alpha}].$$
It follows from (\ref{Q}) that there exists a positive constant
$C$ such that
$$|B(g,2r)|\leq C|B(g,r)|$$
for all $(g,r)\in\mathbb{G}_{\alpha}^{n}\times(0,\infty).$
Moreover,
\begin{equation}\label{Re}
\Big(\frac{R}{r}\Big)^{n}\lesssim \frac{|B(g,R)|}{|B(g,r)|}\lesssim\Big(\frac{R}{r}\Big)^{Q}
\ \mbox{for}\ R>r>0.
\end{equation}
 The divergence
of a vector-valued function
$$\phi=\{\phi_{1,1},\ldots,\phi_{1,m},\phi_{2,1},\ldots,\phi_{2,k}\}\in
C^{1}(\mathbb{G}_{\alpha}^{n};\mathbb{R}^{n})$$ is
$$\textrm{div}_{\alpha}(\phi)=\sum_{i=1}^{m}X_{i}\phi_{1i}+\sum_{j=1}^{k}X_{m+j}\phi_{2j}.$$
Let $$\nabla_{\alpha}=(X_{1},X_{2},\ldots
X_{m},X_{m+1},X_{m+2},\ldots,X_{m+k})$$ be the Grushin gradient operator.
The Grushin-Laplace operator can also be expressed  as
$\mathcal{L}=-\textrm{div}_{\alpha}(\nabla_{\alpha}).$

\subsection{The Grushin semigroup $\{e^{-t\mathcal{L}}\}_{t>0}$}\label{sec-2.2}

In this section, we collect various properties of the semigroup
associated with the Grushin-Laplace operator $\mathcal{L}$ which will be used throughout the rest of
this paper.

The semigroup $e^{-t\mathcal{L}}$ is defined as
\begin{equation}\label{G}
e^{-t\mathcal{L}}u(g)=\int_{\mathbb{G}^{n}_{\alpha}}K_{t}(g,g')u(g')dg',\quad
u\in\mathscr{S}(\mathbb{G}^{n}_{\alpha}),
\end{equation} where $K_t(\cdot,\cdot)$ is the heat kernel of this semigroup.
The semigroup is sub-Markovian, and
defines a family of bounded operators
$e^{-t\mathcal{L}}:L^{2}(\mathbb{G}^{n}_{\alpha})\rightarrow
L^{2}(\mathbb{G}^{n}_{\alpha})$ satisfying  the following
properties:
\begin{proposition}\label{G-Pro}
The following statements hold:
\begin{itemize}
  \item [(i)]   For $t,t'\geq0,\ e^{-t\mathcal{L}}e^{-t'\mathcal{L}}=e^{-(t+t')\mathcal{L}};$
  \item [(ii)] For $u\geq0$ and $t>0,$ $e^{-t\mathcal{L}}u\geq0;$
  \item [(iii)] For $t>0$, $e^{-t\mathcal{L}}$ is a self-adjoint operator on $L^{2}(\mathbb{G}^{n}_{\alpha});$
  \item [(iv)] For $u\in L^{p}(\mathbb{G}^{n}_{\alpha})$ with $p\in[1,\infty],$ $e^{-t\mathcal{L}}:L^{p}(\mathbb{G}^{n}_{\alpha})\rightarrow L^{p}(\mathbb{G}^{n}_{\alpha})$ with $$\|e^{-t\mathcal{L}}\|_{L^{p}(\mathbb{G}^{n}_{p})\rightarrow L^{p}(\mathbb{G}^{n}_{\alpha})}\le1;$$
  \item [(v)] For $p\in[1,\infty),$ the semigroup $e^{-t\mathcal{L}}$ is a strongly continuous semigroup on $L^{p}(\mathbb{G}^{n}_{\alpha}).$
\end{itemize}
\end{proposition}
\begin{proof}
Statements (i), (iii) follows from \cite[Section 6]{RS1} and
\cite{RS2}. In addition, the semigroup $e^{-t\mathcal{L}}$ is a
contraction on $L^{\infty}(\mathbb{G}^{n}_{\alpha})$ for any $t>0$
(see \cite[Section 6]{RS1}). Notice that $\mathcal{L}$ is a
non-negative self-adjoint operator on
$L^{2}(\mathbb{G}_{\alpha}^{n}).$ Thus, we use \cite[Theorem
1.3.3]{D1990} to deduce that (iv). (ii) can be directly obtained
from \cite[p365]{J}. Combining (ii) and (iv), we conclude from
\cite[Theorem 1.4.1]{D1990} that (v).
\end{proof}

Moreover, the semigroup $e^{-t\mathcal{L}}$ is stochastically
complete, which means that for almost all
$g\in\mathbb{G}^{n}_{\alpha}$ and $t>0,$
\begin{align}\label{SC}
e^{-t\mathcal{L}}1(g)=\int_{\mathbb{G}^{n}_{\alpha}}K_{t}(g,g')dg'=1.
\end{align}
For a proof of (\ref{SC}) we can refer to  \cite[Theorem 6.1]{RS1}.

The infinitesimal generator $\mathcal{L}$ of the semigroup $\{e^{-t\mathcal{L}}\}_{t>0}$
is defined by
\begin{align}\label{inf-gen}
\mathcal{L}u:=\lim_{t\rightarrow0}\frac{e^{-t\mathcal{L}}u-u}{t},
\end{align}
where the limit is in $L^{p}(\mathbb{G}_{\alpha}^{n}).$
The domain $D(\mathcal{L})$ of the generator $\mathcal{L}$
is the space of functions $u\in L^{p}(\mathbb{G}_{\alpha}^{n})$ for which
the limit in (\ref{inf-gen}) exists. By \cite[Theorem 1.4]{EN}, $\mathcal{L}$ is closed and $D(\mathcal{L})$
is dense in $L^{p}(\mathbb{G}_{\alpha}^{n}).$
Moreover, for any
$t>0$ and $u\in D(\mathcal{L}),$ then $e^{-t\mathcal{L}}u\in D(\mathcal{L})$ and
\begin{equation}\label{D}
\frac{d}{dt}e^{-t\mathcal{L}}u=e^{-t\mathcal{L}}\mathcal{L}u=\mathcal{L}e^{-t\mathcal{L}}u
\end{equation}
(see \cite[Lemma 1.3 (ii)]{EN}).
\begin{remark}\label{rem3}
Combining (\ref{inf-gen}) with (\ref{D}), we conclude
$\mathscr{S}(\mathbb{G}^{n}_{\alpha})\subseteq D(\mathcal{L}).$

In fact, for any $u\in\mathscr{S}(\mathbb{G}^{n}_{\alpha}),$ we first notice that
\begin{align}\label{E}
e^{-t\mathcal{L}}u-u=\int_{0}^{t}\frac{d}{dt'}e^{-t'\mathcal{L}}udt'
=\int_{0}^{t}\mathcal{L}e^{-t'\mathcal{L}}udt'=\int_{0}^{t}e^{-t'\mathcal{L}}\mathcal{L}udt',
\end{align}
which implies that
$$\frac{e^{-t\mathcal{L}}u-u}{t}-\mathcal{L}u=\frac{1}{t}\int_{0}^{t}\Big(e^{-t'\mathcal{L}}\mathcal{L}u
-\mathcal{L}u\Big)dt'.$$
Therefore, we obtain from (\ref{E}) and Proposition \ref{G-Pro} (iv) that
$$\Big\|\frac{e^{-t\mathcal{L}}u-u}{t}-\mathcal{L}u\Big\|_{L^{p}(\mathbb{G}_{\alpha}^{n})}
\leq\frac{1}{t}\int_{0}^{t}\Big\|e^{-t'\mathcal{L}}\mathcal{L}u
-\mathcal{L}u\Big\|_{L^{p}(\mathbb{G}_{\alpha}^{n})}dt'\leq t\|\mathcal{L}(\mathcal{L}u)\|_{L^{p}(\mathbb{G}_{\alpha}^{n})},$$
where we have used the fact $\mathcal{L}u\in\mathscr{S}(\mathbb{G}^{n}_{\alpha}).$
This shows that $\mathscr{S}(\mathbb{G}^{n}_{\alpha})\subseteq D(\mathcal{L})$ by letting $t\rightarrow0.$
 \qed
\end{remark}

The following upper bounds for the semigroup kernel $K_{t}(\cdot,\cdot)$ can be found in \cite{RS1}.
\begin{lemma}$\mathrm{(}$\cite[Theorem 6.4]{RS1}$\mathrm{)}$\label{upper}
There exists a positive constant $C$ such that the kernel
$K_{t}(\cdot,\cdot)$ of the Grushin semigroup satisfies
$$0\leq K_{t}(g,g')\leq C(|B(g,\sqrt{t})||B(g',\sqrt{t})|)^{-1/2} $$
for all $t>0$ and almost all $g,g'\in\mathbb{G}^{n}_{\alpha}.$
\end{lemma}

Moreover, we know from \cite[Theorem 1.2]{RS2} that the semigroup kernel $K_{t}(\cdot,\cdot)$ also satisfies the
following estimates for Gaussian bounds: there exist two positive
constants $C_1, C_2$ such that for all
$g,g'\in\mathbb{G}^{n}_{\alpha},t>0,$
\begin{equation}\label{Kt}
|B(g,\sqrt{t})|^{-1}e^{-{C_1d_{\alpha}(g,g')^{2}}/{t}}\lesssim
K_{t}(g,g')\lesssim
|B(g,\sqrt{t})|^{-1}e^{-{C_2d_{\alpha}(g,g')^{2}}/{t}}.
\end{equation}

The following lemma implies the
$L^{p}(\mathbb{G}^{n}_{\alpha})\rightarrow
L^{\infty}(\mathbb{G}^{n}_{\alpha})$ ultracontractivity of the
semigroup $e^{-t\mathcal{L}}$ generated by the Grushin-Laplace
operator $\mathcal{L}.$
\begin{proposition}\label{lem4}
Let $p\in[1,\infty)$ and $u\in L^{p}(\mathbb{G}^{n}_{\alpha}).$
Then there exists a positive constant $C(p,Q),$ depending on $p$ and $Q,$ such that
for any $g\in \mathbb{G}^{n}_{\alpha}$
and $t>0,$
$$|e^{-t\mathcal{L}}u(g)|\leq C(p,Q) t^{-Q/2p}\|u\|_{L^{p}(\mathbb{G}^{n}_{\alpha})}.$$
\end{proposition}
\begin{proof}
Applying the H\"{o}lder inequality to (\ref{G}), we see
$$|e^{-t\mathcal{L}}u(g)|\leq\|u\|_{L^{p}(\mathbb{G}^{n}_{\alpha})}\Big(\int_{\mathbb{G}^{n}_{\alpha}}K_{t}(g,g')^{p'}dg'\Big)^{{1}/{p'}},$$ where $1/p+1/p'=1.$
By using the Gaussian upper bound
(\ref{Kt}) and (\ref{Re}), we obtain
\begin{align*}
\Big(\int_{\mathbb{G}^{n}_{\alpha}}K_{t}(g,g')^{p'}dg'\Big)^{{1}/{p'}}&\lesssim\Big(\int_{\mathbb{G}^{n}_{\alpha}}
|B(g,\sqrt{t})|^{-p'}e^{-{p'C_{2}d_{\alpha}(g,g')^{2}}/{t}}dg'\Big)^{{1}/{p'}}\\
&\lesssim\Big(\int_{d_{\alpha}(g,g')< \sqrt{t}}
|B(g,\sqrt{t})|^{-p'}e^{-{p'C_{2}d_{\alpha}(g,g')^{2}}/{t}}dg'\\
&\quad+\sum^{\infty}_{j=0}\int_{2^j\sqrt{t}\leq d_{\alpha}(g,g')<
2^{j+1}\sqrt{t}}
|B(g,\sqrt{t})|^{-p'}e^{-{p'C_{2}d_{\alpha}(g,g')^{2}}/{t}}dg'\Big)^{{1}/{p'}}\\
&\lesssim\Big(|B(g,\sqrt{t})|^{-p'+1}+\sum^{\infty}_{j=0}e^{-p'C_{2}4^j}|B(g,\sqrt{t})|^{-p'} |B(g,2^j\sqrt{t})|\Big)^{{1}/{p'}}\\
&\lesssim\Big(|B(g,\sqrt{t})|^{-p'+1}+\sum^{\infty}_{j=0}e^{-p'C_{2}4^j}2^{jQ}|B(g,\sqrt{t})|^{-p'+1}  \Big)^{{1}/{p'}}\\
&\leq C(p,Q)
 |B(g,\sqrt{t})|^{-1+{1}/{p'}}\\&\lesssim C(p,Q) t^{-{Q}/{(2p)}}.
\end{align*}
This completes the proof of Proposition \ref{lem4}.
\end{proof}

Using Proposition \ref{lem4}, we also establish the
following  $L^{p}(\mathbb{G}^{n}_{\alpha})\rightarrow L^{q}(\mathbb{G}^{n}_{\alpha})$ ultracontractivity of the Grushin semigroup.
\begin{proposition}\label{ulc}
Let $1\leq p\leq q<\infty$ and $u\in L^{p}(\mathbb{G}^{n}_{\alpha}).$ Then there exists a positive constant $C(p,q,Q),$ depending
on $p,q$ and $Q,$ such that for any $t>0,$
$$\|e^{-t\mathcal{L}}u\|_{L^{q}(\mathbb{G}^{n}_{\alpha})}\leq C(p,q,Q) t^{Q(p-q)/(2pq)}\|u\|_{L^{p}(\mathbb{G}^{n}_{\alpha})}.$$
\end{proposition}
\begin{proof}Let $u\in L^{1}(\mathbb{G}^{n}_{\alpha})$.
For any given $r\geq1,$  via Minkowski's integral inequality we have
\begin{align*}
\bigg(\int_{\mathbb{G}^{n}_{\alpha}}|e^{-t\mathcal{L}}u(g)|^{r}dg\bigg)^{1/r}
&=\bigg(\int_{\mathbb{G}^{n}_{\alpha}}\bigg|\int_{\mathbb{G}^{n}_{\alpha}}K_{t}(g,g')u(g')dg'
\bigg|^{r}dg\bigg)^{1/r}\\
&\leq\int_{\mathbb{G}^{n}_{\alpha}}|u(g')|\bigg(\int_{\mathbb{G}^{n}_{\alpha}}K_{t}(g,g')^{r}dg\bigg)^{1/r}dg'.
\end{align*}
Similarly to the proof of Proposition \ref{lem4}, we deduce from the symmetry of $K_{t}(\cdot,\cdot)$ that
\begin{align*}
\Big(\int_{\mathbb{G}^{n}_{\alpha}}K_{t}(g,g')^{r}dg\Big)^{1/r}\leq C(r,Q) t^{Q(1-r)/2r}.
\end{align*}
Therefore,
$$\bigg(\int_{\mathbb{G}^{n}_{\alpha}}|e^{-t\mathcal{L}}u(g)|^{r}dg\bigg)^{1/r}\leq C(r,Q)
t^{Q(1-r)/2r}\|u\|_{L^{1}(\mathbb{G}^{n}_{\alpha})}.$$ That is
$$e^{-t\mathcal{L}}:L^{1}(\mathbb{G}^{n}_{\alpha})\rightarrow L^{r}(\mathbb{G}^{n}_{\alpha}).$$
On the other hand, we know from Proposition \ref{lem4} that
$$e^{-t\mathcal{L}}:L^{r'}(\mathbb{G}^{n}_{\alpha})\rightarrow L^{\infty}(\mathbb{G}^{n}_{\alpha})$$
with $$\|e^{-t\mathcal{L}}\|_{L^{r'}(\mathbb{G}^{n}_{\alpha})\rightarrow L^{\infty}(\mathbb{G}^{n}_{\alpha})}
\leq C(r',Q)t^{-Q/2r'}.$$
We now choose $r\geq1$ such that $r'\geq p,$ then there exists $\theta\in[0,1]$ such that
$$\frac{1}{p}=\frac{1-\theta}{1}+\frac{\theta}{r'}=1-\frac{\theta}{r}.$$
By the Riesz-Thorin interpolation theorem, we have
$$e^{-t\mathcal{L}}:L^{p}(\mathbb{G}^{n}_{\alpha})\rightarrow L^{q}(\mathbb{G}^{n}_{\alpha})$$
with
$$\frac{1}{q}=\frac{1-\theta}{r}+\frac{\theta}{\infty}=\frac{1}{p}+\frac{1}{r}-1.$$
Moreover, we obtain
$$\|e^{-t\mathcal{L}}\|_{L^{p}(\mathbb{G}^{n}_{\alpha})\rightarrow L^{q}(\mathbb{G}^{n}_{\alpha})}
\leq
\Big(\frac{C(r,Q)}{t^{Q(r-1)/2r}}\Big)^{1-\theta}\Big(\frac{C(r',Q)}{t^{Q/2r'}}\Big)^{\theta}
=C(p,q,Q)t^{Q(p-q)/(2pq)},$$ where we have used the fact that
$(1-\theta)/r=1/q,1-\theta/r=1/p$ and
$$C(p,q,Q)=C(r,Q)^{1-\theta}C(r',Q)^{\theta},$$ which implies the
desired conclusion.
\end{proof}

\subsection{Fractional powers of the Grushin-Laplace operator}\label{sec-2.3}
For a given closed operator $A$ on a Banach space $X,$ under the
assumption that $\|\lambda R(\lambda,A)\|\leq M$ for $\lambda>0$
(there exists a operator  $A$ which satisfies such
hypothesis but does not generate a semigroup),
Balakrishnan \cite{B} constructed the fractional powers  of $A$ by
the following formula
\begin{equation}\label{A}
A^{s}x=-\dfrac{\sin(\pi
s)}{\pi}\int_{0}^{\infty}\lambda^{s-1}R(\lambda,A)Axd\lambda,\quad
\mathfrak{R}es\in(0,1), x\in X,
\end{equation}
where $R(\lambda,A)=(\lambda I-A)^{-1}$ and $\mathfrak{R}es$ is the
real part of $s.$

When $A$ does generate a strongly continuous semigroup
$\{T(t)\}_{t>0}$ on $X,$ then it is well-known that (\ref{A}) can
be also expressed as follows:
\begin{equation}\label{Aa}
A^{s}x=-\dfrac{s}{\Gamma{(1-s)}}\int_{0}^{\infty}t^{-s-1}(T(t)x-x)dt,\quad
\mathfrak{R}e s\in (0,1), x\in X.
\end{equation}

With $\mathcal{L}$ as in (\ref{K}), we use (\ref{Aa}) and the
semigroup $\{e^{-t\mathcal{L}}\}_{t>0}$ to define the fractional
power  defined on functions
$u\in\mathscr{S}(\mathbb{G}_{\alpha}^{n})$ by the following
pointwise formula:
\begin{definition}\label{NO}
Let $s\in(0,1).$ For any $u\in\mathscr{S}(\mathbb{G}_{\alpha}^{n}),$ we have
\begin{equation}\label{-G}
\mathcal{L}^{s}u(g)=-\dfrac{s}{\Gamma(1-s)}\int_{0}^{\infty}t^{-s-1}\big(e^{-t\mathcal{L}}u(g)-u(g)\big)dt.
\end{equation}
\end{definition}
\begin{remark}\label{rem}
Observe that the right-hand side of (\ref{-G}) is a convergent integral in $L^{p}(\mathbb{G}_{\alpha}^{n})$ for any $p\in[1,\infty].$

In fact, by (\ref{E}), we first notice that
\begin{align}\label{Strong}
\|e^{-t\mathcal{L}}u-u\|_{L^{p}(\mathbb{G}^{n}_{\alpha})}\leq \int_{0}^{t}\|e^{-t'\mathcal{L}}\mathcal{L}u\|_{L^{p}(\mathbb{G}^{n}_{\alpha})}dt'
\leq t\|\mathcal{L}u\|_{L^{p}(\mathbb{G}^{n}_{\alpha})}.
\end{align}
Thus, it follows from the right-hand side of (\ref{-G}) that
\begin{align*}
\Big(\int_{\mathbb{G}^{n}_{\alpha}}\Big|\int_{0}^{\infty}\big(e^{-t\mathcal{L}}u(g)-u(g)\big)\dfrac{dt}{t^{1+s}}\Big|^{p}dg\Big)^{1/p}
&\leq\int_{0}^{\infty}\|e^{-t\mathcal{L}}u-u\|_{L^{p}(\mathbb{G}^{n}_{\alpha})}\dfrac{dt}{t^{1+s}}
=I+II,
\end{align*}
where
\begin{align*}
  \left\{\begin{aligned}
  I&:=\int_{0}^{1}t^{-s-1}\|e^{-t\mathcal{L}}u-u\|_{L^{p}(\mathbb{G}^{n}_{\alpha})}dt;\\
  II&:=\int_{1}^{\infty}t^{-s-1}\|e^{-t\mathcal{L}}u-u\|_{L^{p}(\mathbb{G}^{n}_{\alpha})}dt.
  \end{aligned}\right.
\end{align*}
For the first integral $I$, by using (\ref{Strong}), we obtain
\begin{align}\label{First}
I\leq \|\mathcal{L}u\|_{L^{p}(\mathbb{G}^{n}_{\alpha})}\int_{0}^{1}t^{-s}dt<\infty.
\end{align}
For the second integral $II$, we deduce from Proposition \ref{G-Pro} (iv) that
have
\begin{align}\label{Second}
II\lesssim \|u\|_{L^{p}(\mathbb{G}^{n}_{\alpha})} \int_{1}^{\infty}t^{-s-1}dt<\infty.
\end{align}
Combining (\ref{First}) with (\ref{Second}) yields the desired
result.\qed
\end{remark}

In what follows, we recall Sobolev type spaces in terms of the
Grushin-Laplace operator $\mathcal{L}$ on Grushin spaces.
\begin{definition}\label{STS}
Let $(p,s)\in[1,\infty)\times(0,1).$ The Sobolev type space $\mathcal{W}^{2s,p}(\mathbb{G}^{n}_{\alpha})$ related to $\mathcal{L}^{s}$ is defined by
$$\mathcal{W}^{2s,p}(\mathbb{G}^{n}_{\alpha}):=\{u\in L^{p}(\mathbb{G}^{n}_{\alpha}):\mathcal{L}^{s}u\in L^{p}(\mathbb{G}^{n}_{\alpha})\},$$
endowed with the norm
\begin{align}\label{G-N}
\|u\|_{\mathcal{W}^{2s,p}(\mathbb{G}^{n}_{\alpha})}:=\|u\|_{L^{p}(\mathbb{G}^{n}_{\alpha})}+\|\mathcal{L}^{s}u\|_{L^{p}(\mathbb{G}^{n}_{\alpha})}.
\end{align}
If we also denote the domain of $\mathcal{L}^{s}$ in
$L^{p}(\mathbb{G}^{n}_{\alpha})$ as
$\mathcal{W}^{2s,p}(\mathbb{G}^{n}_{\alpha}),$ then the operator
$\mathcal{L}^{s}$ can be extended to a closed operator on
$\mathcal{W}^{2s,p}(\mathbb{G}^{n}_{\alpha})$ (see \cite[Lemma
2.1]{B}). Therefore, $\mathcal{W}^{2s,p}(\mathbb{G}^{n}_{\alpha})$,
endowed with the norm (\ref{G-N}),  becomes a Banach space.
\end{definition}
\begin{remark}\label{rem1}
Two comments on Definition \ref{STS} are in order:
\begin{itemize}
  \item [(i)] From Remark \ref{rem}, it is not difficult to verify
$\mathscr{S}(\mathbb{G}^{n}_{\alpha})\subseteq
\mathcal{W}^{2s,p}(\mathbb{G}^{n}_{\alpha}).$
  \item [(ii)] For $p\in(1,\infty),$ if $\alpha=0,s=1/2,$  then the space
$\mathcal{W}^{1,p}(\mathbb{G}_{\alpha}^{n})$ coincides with the
classical Sobolev space $$W^{1,p}(\mathbb{R}^{n})=\{u\in
L^{p}(\mathbb{R}^{n}):\nabla u\in L^{p}(\mathbb{R}^{n})\},$$ which
is endowed with the usual norm
$$\|u\|_{W^{1,p}(\mathbb{R}^{n})}:=\|u\|_{L^{p}(\mathbb{R}^{n})}+\|\nabla
u\|_{L^{p}(\mathbb{R}^{n})}.$$
\end{itemize}
\end{remark}

Below we provide a conclusion on density, which plays a crucial role in the proof of the main theorem.
\begin{proposition}\label{W-de}
Let $(p,s)\in[1,\infty)\times(0,1).$ Then $\mathscr{S}(\mathbb{G}^{n}_{\alpha})$ is dense in $\mathcal{W}^{2s,p}(\mathbb{G}^{n}_{\alpha}).$
\end{proposition}
\begin{proof}
We borrow  the ideas from \cite[Proposition 2.13]{GT2020} to
complete the proof of Proposition \ref{W-de}. For $\varepsilon>0,$
we first construct the following operator
$$\mathcal{L}_{\varepsilon}=\mathcal{L}+\varepsilon I: D(\mathcal{L})\rightarrow L^{p}(\mathbb{G}^{n}_{\alpha}),$$
where $D(\mathcal{L})$ is the domain of $\mathcal{L}$ in $L^{p}(\mathbb{G}^{n}_{\alpha})$ with the norm
$$\|u\|_{D(\mathcal{L})}:=\|u\|_{L^{p}(\mathbb{G}^{n}_{\alpha})}
+\|\mathcal{L}u\|_{L^{p}(\mathbb{G}^{n}_{\alpha})}$$
(see Section \ref{sec-2.2}). According to Proposition \ref{G-Pro} (iv) and \cite[Theorem 1.10]{EN},
$\mathcal{L}_{\varepsilon}$ is invertible, and the inverse is given by the formula
$$R(\varepsilon,\mathcal{L})u=\int_{0}^{\infty}e^{-\varepsilon t}e^{-t\mathcal{L}}udt.$$
By this fact, we can define the rescaled semigroup $e^{-t\mathcal{L_{\varepsilon}}}:=e^{-\varepsilon t}e^{-t\mathcal{L}}.$
In addition, this semigroup has $\mathcal{L}_{\varepsilon}$ as its infinitesimal generator.
Similarly to (\ref{-G}), the fractional powers of $\mathcal{L}_{\varepsilon}$ is defined by
\begin{align}\label{-G1}
\mathcal{L}_{\varepsilon}^{s}u(g)=-\dfrac{s}{\Gamma(1-s)}\int_{0}^{\infty}t^{-s-1}\big(e^{-t\mathcal{L}_{\varepsilon}}u(g)-u(g)\big)dt.
\end{align}
As in Definition \ref{STS}, the domain of the generator $\mathcal{L}_{\varepsilon}^{s}$
is the space
$$\mathcal{W}_{\varepsilon}^{2s,p}(\mathbb{G}^{n}_{\alpha}):=\{u\in L^{p}(\mathbb{G}^{n}_{\alpha}):\mathcal{L}_{\varepsilon}^{s}u\in L^{p}(\mathbb{G}^{n}_{\alpha})\},$$
endowed with the norm
\begin{align*}
\|u\|_{\mathcal{W}_{\varepsilon}^{2s,p}(\mathbb{G}^{n}_{\alpha})}:=\|u\|_{L^{p}(\mathbb{G}^{n}_{\alpha})}+\|\mathcal{L}_{\varepsilon}^{s}u\|_{L^{p}(\mathbb{G}^{n}_{\alpha})}.
\end{align*}
Note that $$\mathcal{W}_{\varepsilon}^{2s,p}(\mathbb{G}^{n}_{\alpha})=\mathcal{W}^{2s,p}(\mathbb{G}^{n}_{\alpha}),$$
and there exists positive constant $C$ such that for any $u\in \mathcal{W}^{2s,p}(\mathbb{G}^{n}_{\alpha}),$
$$\|\mathcal{L}_{\varepsilon}^{s}u-\mathcal{L}^{s}u\|_{L^{p}(\mathbb{G}^{n}_{\alpha})}\leq C\varepsilon^{s}\|u\|_{L^{p}(\mathbb{G}^{n}_{\alpha})}$$
(see \cite[Lemma 4.11]{L2018}). In order to prove that $\mathscr{S}(\mathbb{G}^{n}_{\alpha})$ is dense in $\mathcal{W}^{2s,p}(\mathbb{G}^{n}_{\alpha}),$ we only need to prove that $\mathscr{S}(\mathbb{G}^{n}_{\alpha})$ is dense in $\mathcal{W}_{\varepsilon}^{2s,p}(\mathbb{G}^{n}_{\alpha}).$
Once the latter holds, then for any $u\in \mathcal{W}_{\varepsilon}^{2s,p}(\mathbb{G}^{n}_{\alpha})=\mathcal{W}^{2s,p}(\mathbb{G}^{n}_{\alpha}),$ there exists $u_{k}\in \mathscr{S}(\mathbb{G}^{n}_{\alpha})$ such that
$$\|u_{k}-u\|_{L^{p}(\mathbb{G}^{n}_{\alpha})},\ \  \|\mathcal{L}_{\varepsilon}^{s}(u_{k}-u)\|_{L^{p}(\mathbb{G}^{n}_{\alpha})}
\rightarrow0\ \ \mbox{as}\ k\rightarrow\infty.$$
Further, we deduce that
\begin{align*}
\|\mathcal{L}^{s}(u_{k}-u)\|_{L^{p}(\mathbb{G}^{n}_{\alpha})}
&=\|\mathcal{L}^{s}(u_{k}-u)-\mathcal{L}_{\varepsilon}^{s}(u_{k}-u)+\mathcal{L}_{\varepsilon}^{s}(u_{k}-u)\|_{L^{p}(\mathbb{G}^{n}_{\alpha})}\\
&\leq C\varepsilon^{s}\|u_{k}-u\|_{L^{p}(\mathbb{G}^{n}_{\alpha})}+
\|\mathcal{L}_{\varepsilon}^{s}(u_{k}-u)\|_{L^{p}(\mathbb{G}^{n}_{\alpha})}\rightarrow 0
\end{align*}
as $k\rightarrow\infty.$ This implies that the density of
$\mathscr{S}(\mathbb{G}^{n}_{\alpha})$ in
$\mathcal{W}^{2s,p}(\mathbb{G}^{n}_{\alpha})$.

Therefore, it suffices to prove that $\mathscr{S}(\mathbb{G}^{n}_{\alpha})$ is dense in $\mathcal{W}_{\varepsilon}^{2s,p}(\mathbb{G}^{n}_{\alpha}),$
that is,
\begin{align}\label{DD-De}
\overline{\mathscr{S}(\mathbb{G}^{n}_{\alpha})}
^{\|\cdot\|_{\mathcal{W}_{\varepsilon}^{2s,p}(\mathbb{G}_{\alpha}^{n})}}=\mathcal{W}_{\varepsilon}^{2s,p}(\mathbb{G}_{\alpha}^{n}).
\end{align}
(Here, we have used the fact
$\mathscr{S}(\mathbb{G}^{n}_{\alpha})\subseteq
\mathcal{W}_{\varepsilon}^{2s,p}(\mathbb{G}^{n}_{\alpha}),$ see also
Remark \ref{rem1} (i).)  In addition, for any $u\in D(\mathcal{L}),$
it follows from (\ref{-G1}), (\ref{First}) and $(\ref{Second})$ that
$$\|\mathcal{L}_{\varepsilon}^{s}u\|_{L^{p}(\mathbb{G}^{n}_{\alpha})}\lesssim \|u\|_{L^{p}(\mathbb{G}^{n}_{\alpha})}
+\|\mathcal{L}_{\varepsilon}u\|_{L^{p}(\mathbb{G}^{n}_{\alpha})},$$
which implies that $D(\mathcal{L})\subseteq
\mathcal{W}_{\varepsilon}^{2s,p}(\mathbb{G}_{\alpha}^{n}).$ Observe
that the invertibility of $\mathcal{L}_{\varepsilon}^{s} $ (see
Lemma \ref{FF}). Using the same argument as in \cite[Proposition
2.13 Step II]{GT2020} tells us that $D(\mathcal{L})$ is dense in
$\mathcal{W}_{\varepsilon}^{2s,p}(\mathbb{G}_{\alpha}^{n}).$ Thus,
we have
\begin{align}\label{D-de}
\overline{D(\mathcal{L})}
^{\|\cdot\|_{\mathcal{W}_{\varepsilon}^{2s,p}(\mathbb{G}_{\alpha}^{n})}}=\mathcal{W}_{\varepsilon}^{2s,p}(\mathbb{G}_{\alpha}^{n}).
\end{align}
Since $\mathscr{S}(\mathbb{G}^{n}_{\alpha})$ is dense in
$L^{p}(\mathbb{G}^{n}_{\alpha})$ and $D(\mathcal{L})\subset
L^{p}(\mathbb{G}^{n}_{\alpha}),$ we use Remark   \ref{rem3} and
(\ref{D-de}) to  derive
$$\overline{\mathscr{S}(\mathbb{G}^{n}_{\alpha})}
^{\|\cdot\|_{\mathcal{W}_{\varepsilon}^{2s,p}(\mathbb{G}_{\alpha}^{n})}}=\overline{D(\mathcal{L})}
^{\|\cdot\|_{\mathcal{W}_{\varepsilon}^{2s,p}(\mathbb{G}_{\alpha}^{n})}}=\mathcal{W}_{\varepsilon}^{2s,p}(\mathbb{G}_{\alpha}^{n}).$$
This proves the desired conclusion (\ref{DD-De}).
\end{proof}

\begin{lemma}\label{co}
Let $(p,s)\in [1,\infty)\times(0,1).$ For any $u\in \mathscr{S}(\mathbb{G}^{n}_{\alpha})$ and for all $t>0,$ we have
$e^{-t\mathcal{L}}u\in \mathcal{W}^{2s,p}(\mathbb{G}^{n}_{\alpha}).$ Moreover,
\begin{equation*}
\mathcal{L}^{s}e^{-t\mathcal{L}}u=e^{-t\mathcal{L}}\mathcal{L}^{s}u\ \ a.e.\ in\ \mathbb{G}^{n}_{\alpha}.
\end{equation*}
\end{lemma}
\begin{proof}
For any $u\in \mathscr{S}(\mathbb{G}^{n}_{\alpha}),$ we deduce from (\ref{-G}), (\ref{SC}) and  Proposition \ref{G-Pro} (iv) that
\begin{align*}
\|\mathcal{L}^{s}e^{-t\mathcal{L}}u\|_{L^{p}(\mathbb{G}^{n}_{\alpha})}&\leq\frac{s}{\Gamma(1-s)}\int_{0}^{\infty}
\frac{\|e^{-t'\mathcal{L}}e^{-t\mathcal{L}}u-e^{-t\mathcal{L}}u\|_{L^{p}(\mathbb{G}^{n}_{\alpha})}}{t'^{1+s}}dt'\\
&\leq\frac{s}{\Gamma(1-s)}\int_{0}^{\infty}
\frac{\|e^{-t'\mathcal{L}}u-u\|_{L^{p}(\mathbb{G}^{n}_{\alpha})}}{t'^{1+s}}dt',
\end{align*}
which, together with (\ref{First}) and (\ref{Second}), yields
$$\|\mathcal{L}^{s}e^{-t\mathcal{L}}u\|_{L^{p}(\mathbb{G}^{n}_{\alpha})}
\lesssim \|\mathcal{L}u\|_{L^{p}(\mathbb{G}^{n}_{\alpha})}+\|u\|_{L^{p}(\mathbb{G}^{n}_{\alpha})}<\infty.$$
This implies $e^{-t\mathcal{L}}u\in
\mathcal{W}^{2s,p}(\mathbb{G}^{n}_{\alpha}).$
Thanks to the Fubini theorem we
conclude  that
\begin{align*}
\mathcal{L}^{s}e^{-t\mathcal{L}}u(g)&=-\frac{s}{\Gamma{(1-s)}}\int_{0}^{\infty}\frac{e^{-t\mathcal{L}}(e^{-t'\mathcal{L}}u-u)(g)}{t'^{1+s}}dt'\\
&=-\frac{s}{\Gamma{(1-s)}}\int_{0}^{\infty}\int_{\mathbb{G}^{n}_{\alpha}}K_{t}(g,g')\frac{e^{-t'\mathcal{L}}u(g')-u(g')}{t'^{1+s}}dg'dt'\\
&=-\frac{s}{\Gamma{(1-s)}}\int_{\mathbb{G}^{n}_{\alpha}}K_{t}(g,g')\Big(\int_{0}^{\infty}\frac{e^{-t'\mathcal{L}}u(g')-u(g')}{t'^{1+s}}dt'\Big)dg'\\
&=e^{-t\mathcal{L}}\mathcal{L}^{s}u(g)
\end{align*} for almost every $g\in \mathbb{G}^{n}_{\alpha}$ and for all $t>0.$
\end{proof}

In order to investigate the Riesz potential operators
$\mathscr{I}_{\tilde{\alpha}}$, we next recall the Riesz kernel of
order $\tilde{\alpha}$ (see \cite{ZL2023}).
\begin{definition}\label{def1}
Let $\tilde{\alpha}\in(0,Q).$ The Riesz kernel of order
$\tilde{\alpha}$ is defined as
\begin{equation}\label{Ia}
I_{\tilde{\alpha}}(g,g'):=\dfrac{1}{\Gamma({\tilde{\alpha}}/{2})}\int_{0}^{\infty}t^{{\tilde{\alpha}}/{2}-1}K_{t}(g,g')dt
\quad \forall~ g,g'\in\mathbb{G}^{n}_{\alpha}.
\end{equation}
\end{definition}

It should be noted that for every
$g\in\mathbb{G}_{\alpha}^{n}\backslash\{0\}$, the integral in
Definition  \ref{def1} converges absolutely, which can be seen from
(\ref{Kt}).

The next lemma, which is exactly \cite[Lemma 2.5]{ZL2023}, tells us
that the Riesz potential operator $\mathscr{I}_{\tilde{\alpha}}$ is
the inverse of $\mathcal{L}^{s},$ where the
Riesz potential operator $\mathscr{I}_{\tilde{\alpha}}$ is defined
as
\begin{equation}\label{FFF}
\mathscr{I}_{\tilde{\alpha}}u(g):=\int_{\mathbb{G}_{\alpha}^{n}}u(g')
I_{\tilde{\alpha}}(g,g')dg'
=\dfrac{1}{\Gamma({\tilde{\alpha}}/{2})}\int_{\mathbb{G}_{\alpha}^{n}}u(g')
\int_{0}^{\infty}t^{{\tilde{\alpha}}/{2}-1}K_{t}(g,g')dtdg'
\end{equation}
for any $u\in \mathscr{S}(\mathbb{G}^{n}_{\alpha}).$ Moreover, using (\ref{G}) and exchanging the
 order of integration in  (\ref{FFF}) deduce  that
 $\mathscr{I}_{\tilde{\alpha}}$ can be also written as
\begin{equation}\label{R}
\mathscr{I}_{\tilde{\alpha}}u(g)=\dfrac{1}{\Gamma({\tilde{\alpha}}/{2})}\int_{0}^{\infty}t^{{\tilde{\alpha}}/{2}-1}e^{-t\mathcal{L}}u(g)dt
\quad \forall~ u\in\mathscr{S}(\mathbb{G}^{n}_{\alpha}).
\end{equation}

\begin{lemma}$\mathrm{(}$\cite[Lemma 2.5]{ZL2023}$\mathrm{)}$ \label{FF}
Let $s\in(0,1).$ Then for any $u\in\mathscr{S}(\mathbb{G}^{n}_{\alpha})$ we have
\begin{equation}\label{ff}
u=\mathscr{I}_{2s}(\mathcal{L}^{s}u)=\mathcal{L}^{s}(\mathscr{I}_{2s}u).
\end{equation}
\end{lemma}

\section{Besov type spaces on Grushin settings}\label{sec-3}

\subsection{Besov type spaces associated with $\{e^{-t\mathcal{L}}\}_{t>0}$}\label{sec-3.1}

 In this subsection, we mainly consider Besov type spaces based on Grushin semigroups,
 and then provide some properties of Besov type spaces.
Previously, Besov type spaces on the general metric
spaces were studied in
\cite{ABC2020,ABCRST2020,ABCRST2021,CG2022,GL2015,P2010}. It should be
noted that the authors imposed some   Ahlfor's regularity
  on the ball in the metric space in \cite{CG2022,GL2015,P2010}, while the Grushin space doesn't
  satisfy it, which can be seen from (\ref{B}); in addition,
  \cite{ABC2020,ABCRST2020,ABCRST2021} focus on a class of Besov
  type space in the endpoint case on the Dirichlet space. Therefore,
  our results can't be covered by their results on the metric space
  or Dirichlet space.

\begin{definition}\label{Besov1}
The Besov type space
$B_{p,q}^{\mathcal{L},\beta}(\mathbb{G}^{n}_{\alpha})$ is defined
according to the values of the parameters $(p,q,\beta)$ as follows:
\begin{itemize}
  \item [(i)] If $(p,q,\beta)\in[1,\infty)\times[1,\infty)\times(0,\infty),$ then $B_{p,q}^{\mathcal{L},\beta}(\mathbb{G}^{n}_{\alpha})$
              is the class of all $L^{p}(\mathbb{G}^{n}_{\alpha})$-functions $u$ such that
              \begin{align}\label{Besov}
                      N_{p,q}^{\mathcal{L},\beta}(u):=\bigg(\int_{0}^{\infty}
                       \Big(\int_{\mathbb{G}^{n}_{\alpha}}e^{-t\mathcal{L}}(|u-u(g)|^{p})(g)dg\Big)^{q/p}\frac{dt}{t^{{\beta q}/{2}+1}}
                    \bigg)^{1/q}<\infty.
              \end{align}
  \item [(ii)] If $(p,q,\beta)\in[1,\infty)\times\{\infty\}\times(0,\infty),$ then $B_{p,q}^{\mathcal{L},\beta}(\mathbb{G}^{n}_{\alpha})$
              is the class of all $L^{p}(\mathbb{G}^{n}_{\alpha})$-functions $u$ such that
              \begin{align}\label{NP}
                 N_{p,\infty}^{\mathcal{L},\beta}(u):=\sup_{t>0}t^{-\beta/2}\Big(\int_{\mathbb{G}_{\alpha}^{n}}e^{-t\mathcal{L}}(|u-u(g)|^{p})(g)dg\Big)^{1/p}<\infty.
              \end{align}
\end{itemize}
When $p=q,$ we write
$B_{p,q}^{\mathcal{L},\beta}(\mathbb{G}^{n}_{\alpha})$ as
$B_{p,p}^{\mathcal{L},\beta}(\mathbb{G}^{n}_{\alpha}).$
The space $B_{p,q}^{\mathcal{L},\beta}(\mathbb{G}^{n}_{\alpha})$ is endowed
with the following norm
\begin{align}\label{BSN}
\|u\|_{B_{p,q}^{\mathcal{L},\beta}(\mathbb{G}^{n}_{\alpha})}:=
\|u\|_{L^{p}(\mathbb{G}^{n}_{\alpha})}+N_{p,q}^{\mathcal{L},\beta}(u).
\end{align}
\end{definition}

\begin{remark}\label{rem2}
Two comments on Definition \ref{Besov1} are in order:
\begin{itemize}
  \item [(i)] Observe that if $(p,q,\beta)\in[1,\infty)\times[1,\infty)\times(0,\infty),$ then
$$(\ref{Besov})\Longleftrightarrow \widetilde{N}_{p,q}^{\mathcal{L},\beta}(u):=\bigg(\int_{0}^{1}
\Big(\int_{\mathbb{G}^{n}_{\alpha}}e^{-t\mathcal{L}}(|u-u(g)|^{p})(g)dg\Big)^{q/p}\frac{dt}{t^{{\beta q}/{2}+1}}
\bigg)^{1/q}<\infty.$$
\item [(ii)] When $(p,q,\beta)\in[1,\infty)\times\{\infty\}\times(0,\infty),$ then
  $$(\ref{NP})\Longleftrightarrow \limsup_{t\rightarrow 0}t^{-\beta/2}\Big(\int_{\mathbb{G}_{\alpha}^{n}}e^{-t\mathcal{L}}(|u-u(g)|^{p})(g)dg\Big)^{1/p}<\infty.$$
\end{itemize}
\begin{proof}
(i): If $N_{p,q}^{\mathcal{L},\beta}(u)<\infty,$ it is easy to see
that
$$\bigg(\int_{0}^{1}
\Big(\int_{\mathbb{G}^{n}_{\alpha}}e^{-t\mathcal{L}}(|u-u(g)|^{p})(g)dg\Big)^{q/p}\frac{dt}{t^{{\beta q}/{2}+1}}
\bigg)^{1/q}<\infty$$
Conversely, we assume that
$$\bigg(\int_{0}^{1}
\Big(\int_{\mathbb{G}^{n}_{\alpha}}e^{-t\mathcal{L}}(|u-u(g)|^{p})(g)dg\Big)^{q/p}\frac{dt}{t^{{\beta q}/{2}+1}}
\bigg)^{1/q}<\infty.$$
It follows
from (\ref{SC}) and Proposition \ref{G-Pro} (iv) that
\begin{align*}
&\bigg(\int_{0}^{\infty}
\Big(\int_{\mathbb{G}^{n}_{\alpha}}e^{-t\mathcal{L}}(|u-u(g)|^{p})(g)dg\Big)^{q/p}\frac{dt}{t^{{\beta q}/{2}+1}}
\bigg)^{1/q}\\
&\leq\bigg(\int_{0}^{1}
\Big(\int_{\mathbb{G}^{n}_{\alpha}}e^{-t\mathcal{L}}(|u-u(g)|^{p})(g)dg\Big)^{q/p}\frac{dt}{t^{{\beta q}/{2}+1}}
\bigg)^{1/q}\\
&\quad+
\bigg(\int_{1}^{\infty}
\Big(\int_{\mathbb{G}^{n}_{\alpha}}e^{-t\mathcal{L}}(|u-u(g)|^{p})(g)dg\Big)^{q/p}\frac{dt}{t^{{\beta q}/{2}+1}}
\bigg)^{1/q}\\
&\lesssim \widetilde{N}_{p,q}^{\mathcal{L},\beta}(u)+
\bigg(\int_{1}^{\infty}
\Big(\int_{\mathbb{G}^{n}_{\alpha}}\big(e^{-t\mathcal{L}}(|u|^{p})(g)+|u(g)|^{p}
e^{-t\mathcal{L}}1(g)\big)dg\Big)^{q/p}\frac{dt}{t^{{\beta q}/{2}+1}}
\bigg)^{1/q}\\
&\lesssim \widetilde{N}_{p,q}^{\mathcal{L},\beta}(u)+\|u\|_{L^{p}(\mathbb{G}^{n}_{\alpha})}.
\end{align*}
This proves that
$$N_{p,q}^{\mathcal{L},\beta}(u)\lesssim \widetilde{N}_{p,q}^{\mathcal{L},\beta}(u)+\|u\|_{L^{p}(\mathbb{G}^{n}_{\alpha})}<\infty.$$

(ii): We first assume that $N_{p,\infty}^{\mathcal{L},\beta}(u)<\infty,$ then
  $$\limsup_{t\rightarrow 0}t^{-\beta/2}\Big(\int_{\mathbb{G}_{\alpha}^{n}}e^{-t\mathcal{L}}(|u-u(g)|^{p})(g)dg\Big)^{1/p}
  <N_{p,\infty}^{\mathcal{L},\beta}(u)<\infty.$$
  Conversely, if
  $$\limsup_{t\rightarrow 0}t^{-\beta/2}\Big(\int_{\mathbb{G}_{\alpha}^{n}}e^{-t\mathcal{L}}(|u-u(g)|^{p})(g)dg\Big)^{1/p}<\infty,$$
  then there is some $\varepsilon>0$ such that,
  \begin{align}\label{T1}
  \sup_{t\in(0,\varepsilon]}t^{-\beta/2}\Big(\int_{\mathbb{G}_{\alpha}^{n}}e^{-t\mathcal{L}}(|u-u(g)|^{p})(g)dg\Big)^{1/p}<\infty.
  \end{align}
  On the other hand, for any $t>\varepsilon,$ we deduce from (\ref{SC}) and Proposition \ref{G-Pro} (iv) that
  $$t^{-\beta/2}\Big(\int_{\mathbb{G}_{\alpha}^{n}}e^{-t\mathcal{L}}(|u-u(g)|^{p})(g)dg\Big)^{1/p}\lesssim \varepsilon^{-\beta/2}\|u\|_{L^{p}(\mathbb{G}_{\alpha}^{n})},$$
  which, together with (\ref{T1}), implies
  \begin{align*}
  &\sup_{t>0}t^{-\beta/2}\Big(\int_{\mathbb{G}_{\alpha}^{n}}e^{-t\mathcal{L}}(|u-u(g)|^{p})(g)dg\Big)^{1/p}\\
  &\lesssim \sup_{t\in(0,\varepsilon]}t^{-\beta/2}\Big(\int_{\mathbb{G}_{\alpha}^{n}}e^{-t\mathcal{L}}(|u-u(g)|^{p})(g)dg\Big)^{1/p}
    +\varepsilon^{-\beta/2}\|u\|_{L^{p}(\mathbb{G}_{\alpha}^{n})}<\infty.
  \end{align*}
This completes the proof.
\end{proof}
\end{remark}

In what follows,
we define another type of Besov spaces, which is very convenient when dealing with certain problems.
\begin{definition}\label{Besov2}
For any $(p,q,\beta)\in[1,\infty)\times[1,\infty]\times(0,\infty),$
the space $\mathfrak{B}_{p,q}^{\mathcal{L},\beta}(\mathbb{G}^{n}_{\alpha})$ is defined as
the closure of the space $C_{c}^{\infty}(\mathbb{G}^{n}_{\alpha})$ in $B_{p,q}^{\mathcal{L},\beta}(\mathbb{G}^{n}_{\alpha})$
relative to the norm (\ref{BSN}).
That is, $u\in \mathfrak{B}_{p,q}^{\mathcal{L},\beta}(\mathbb{G}^{n}_{\alpha})$ if and only if there exist functions
$u_{k}\in C_{c}^{\infty}(\mathbb{G}^{n}_{\alpha})$ such that
$$\|u_{k}-u\|_{B_{p,q}^{\mathcal{L},\beta}(\mathbb{G}^{n}_{\alpha})}\rightarrow0
\ \mbox{as}\ k\rightarrow\infty.$$
If $p=q,$ we shall denote
$\mathfrak{B}_{p,q}^{\mathcal{L},\beta}(\mathbb{G}^{n}_{\alpha})$
as $\mathfrak{B}_{p,p}^{\mathcal{L},\beta}(\mathbb{G}^{n}_{\alpha}).$

\end{definition}

The following proposition presents the relationship between $\mathfrak{B}_{p,q}^{\mathcal{L},\beta}(\mathbb{G}^{n}_{\alpha})$ and $B_{p,q}^{\mathcal{L},\beta}(\mathbb{G}^{n}_{\alpha})$.

\begin{proposition}\label{desi}
Let $(p,q,\beta)\in[1,\infty)\times[1,\infty]\times(0,1).$ Then
$$\mathfrak{B}_{p,q}^{\mathcal{L},\beta}(\mathbb{G}^{n}_{\alpha})\subseteq
B_{p,q}^{\mathcal{L},\beta}(\mathbb{G}^{n}_{\alpha}).$$
\end{proposition}
\begin{proof}
We only consider the case $(p,q,\beta)\in[1,\infty)\times[1,\infty)\times(0,1)$, since the proof of $(p,q,\beta)\in[1,\infty)\times\{\infty\}
\times(0,1)$ is similar.
By Definition \ref{Besov2}, it suffices to prove that for any $u\in C_{c}^{\infty}(\mathbb{G}^{n}_{\alpha}),$
\begin{align}\label{NN}
N_{p,q}^{\mathcal{L},\beta}(u)^{q}=\int_{0}^{\infty}
\Big(\int_{\mathbb{G}^{n}_{\alpha}}e^{-t\mathcal{L}}(|u-u(g)|^{p})(g)dg\Big)^{q/p}\frac{dt}{t^{{\beta q}/{2}+1}}
<\infty.
\end{align}
Assume that $\textrm{supp}\ u\subset B(g_{0},r)$ for some $g_{0}\in \mathbb{G}^{n}_{\alpha}$
and $r\in(0,\infty).$ If $|u(g)-u(g')|\neq 0,$ then either $g\in B(g_{0},r)$ or $g'\in B(g_{0},r).$
Via (\ref{Besov}), we can write
\begin{align*}
N_{p,q}^{\mathcal{L},\beta}(u)^{q}&=\int_{0}^{\infty}
\Big(\int_{\mathbb{G}^{n}_{\alpha}}e^{-t\mathcal{L}}(|u-u(g)|^{p})(g)dg\Big)^{q/p}\frac{dt}{t^{{\beta q}/{2}+1}}=I_{1}+I_{2},
\end{align*}
where
\begin{align*}
  \left\{\begin{aligned}
  I_{1}&:=\int_{0}^{1}
\Big(\int_{\mathbb{G}^{n}_{\alpha}}\int_{\mathbb{G}^{n}_{\alpha}}K_{t}(g,g')|u(g')-u(g)|^{p}dg'dg\Big)^{q/p}\frac{dt}{t^{{\beta q}/{2}+1}};\\
  I_{2}&:=\int_{1}^{\infty}
\Big(\int_{\mathbb{G}^{n}_{\alpha}}\int_{\mathbb{G}^{n}_{\alpha}}K_{t}(g,g')|u(g')-u(g)|^{p}dg'dg\Big)^{q/p}\frac{dt}{t^{{\beta q}/{2}+1}}.
  \end{aligned}\right.
\end{align*}
For $I_{2},$ we deduce from the Gaussian upper bound (\ref{Kt}) that
\begin{align*}
I_{2}&\lesssim\int_{1}^{\infty}
\Big(\int_{\mathbb{G}^{n}_{\alpha}}\int_{\mathbb{G}^{n}_{\alpha}}e^{-{C_2d_{\alpha}(g,g')^{2}}/{t}}
|u(g')-u(g)|^{p}\frac{dg'dg}{|B(g,\sqrt{t})|}\Big)^{q/p}\frac{dt}{t^{{\beta q}/{2}+1}}\\
&\lesssim I_{2,1}+I_{2,2},
\end{align*}
where
\begin{align*}
  \left\{\begin{aligned}
  I_{2,1}&:=\int_{1}^{\infty}
\Big(\int_{B(g_{0},r)}\int_{d_{\alpha}(g,g')<\sqrt{t}}e^{-{C_2d_{\alpha}(g,g')^{2}}/{t}}
\|u\|_{L^{\infty}(\mathbb{G}^{n}_{\alpha})}^{p}\frac{dg'dg}{|B(g,\sqrt{t})|}\Big)^{q/p}\frac{dt}{t^{{\beta q}/{2}+1}};\\
  I_{2,2}&:=\int_{1}^{\infty}
\Big(\int_{B(g_{0},r)}\int_{d_{\alpha}(g,g')\geq\sqrt{t}}e^{-{C_2d_{\alpha}(g,g')^{2}}/{t}}
\|u\|_{L^{\infty}(\mathbb{G}^{n}_{\alpha})}^{p}\frac{dg'dg}{|B(g,\sqrt{t})|}\Big)^{q/p}\frac{dt}{t^{{\beta q}/{2}+1}}.
  \end{aligned}\right.
\end{align*}
Clearly, we have
$$I_{2,1}\lesssim |B(g_{0},r)|^{q/p}\int_{1}^{\infty}t^{-{\beta q}/{2}-1}dt<\infty.$$
On the other hand, it follows from (\ref{Re}) that
\begin{align*}
I_{2,2}&\lesssim \int_{1}^{\infty}t^{-{\beta q}/{2}-1}\Big(\int_{B(g_{0},r)}\sum_{i=0}^{\infty}
\int_{d_{\alpha}(g,g')\thicksim2^{i}\sqrt{t}}e^{-2^{2i}C_{2}}\frac{dg'dg}{|B(g,\sqrt{t})|}\Big)^{q/p}dt\\
&\lesssim |B(g_{0},r)|^{q/p}\sum_{i=0}^{\infty}
2^{iQq/p}e^{-2^{2i}C_{2}q/p}\int_{1}^{\infty}t^{-{\beta q}/{2}-1}dt<\infty,
\end{align*}
which, together with the previous estimate $I_{2,1},$ implies
$$I_{2}\lesssim I_{2,1}+I_{2,2}<\infty.$$
Now we estimate the term $I_{1}.$ By using the Gaussian upper bound
(\ref{Kt}), we obtain
\begin{align*}
I_{1}&\lesssim\int_{0}^{1}\frac{1}{t^{{\beta q}/{2}+1}}
\Big(\int_{\mathbb{G}^{n}_{\alpha}}\int_{\mathbb{G}^{n}_{\alpha}}e^{-{C_2d_{\alpha}(g,g')^{2}}/{t}}
|u(g')-u(g)|^{p}\frac{dg'dg}{|B(g,\sqrt{t})|}\Big)^{q/p}dt\\
&\lesssim\int_{0}^{1}\frac{1}{t^{{\beta q}/{2}+1}}
\Big(\int_{B(g_{0},r)}\int_{\mathbb{G}^{n}_{\alpha}}e^{-{C_2d_{\alpha}(g,g')^{2}}/{t}}
\frac{d_{\alpha}(g,g')^{p}}{|B(g,\sqrt{t})|}dg'dg\Big)^{q/p}dt\\
&\lesssim I_{1,1}+I_{1,2},
\end{align*}
where
\begin{align*}
  \left\{\begin{aligned}
  I_{1,1}&:=\int_{0}^{1}
           \Big(\int_{B(g_{0},r)}\int_{d_{\alpha}(g,g')<\sqrt{t}}e^{-{C_2d_{\alpha}(g,g')^{2}}/{t}}
            \frac{d_{\alpha}(g,g')^{p}}{|B(g,\sqrt{t})|}dg'dg\Big)^{q/p}\frac{dt}{t^{{\beta q}/{2}+1}};\\
  I_{1,2}&:=\int_{0}^{1}
             \Big(\int_{B(g_{0},r)}\int_{d_{\alpha}(g,g')\geq\sqrt{t}}e^{-{C_2d_{\alpha}(g,g')^{2}}/{t}}
             \frac{d_{\alpha}(g,g')^{p}}{|B(g,\sqrt{t})|}dg'dg\Big)^{q/p}\frac{dt}{t^{{\beta q}/{2}+1}}.
  \end{aligned}\right.
\end{align*}
For the term $I_{1,1},$ we have
\begin{align*}
I_{1,1}&\lesssim\int_{0}^{1}t^{-{\beta q}/{2}-1}\Big(\int_{B(g_{0},r)}t^{{p}/{2}}dg\Big)^{q/p}dt\\
&\lesssim |B(g_{0},r)|^{q/p}\int_{0}^{1}t^{-{\beta q}/{2}+{q}/{2}-1}dt<\infty,
\end{align*}
where  we have used the fact that $\beta\in(0,1) $ in the last
inequality.

Additionally, we can use (\ref{Re}) again to obtain
\begin{align*}
I_{1,2}&\lesssim \int_{0}^{1}t^{-{\beta q}/{2}-1}\Big(\int_{B(g_{0},r)}\sum_{i=0}^{\infty}
\int_{d_{\alpha}(g,g')\thicksim2^{i}\sqrt{t}}e^{-2^{2i}C_{2}}\frac{d_{\alpha}(g,g')^{p}}{|B(g,\sqrt{t})|}dg'dg\Big)^{q/p}dt\\
&\lesssim |B(g_{0},r)|^{q/p}\sum_{i=0}^{\infty}
2^{iq(p+Q)/p}e^{-2^{2i}C_{2}q/p}\int_{0}^{1}t^{-{\beta q}/{2}+{q}/{2}-1}dt<\infty,
\end{align*}
which, together with the estimate $I_{1,1},$ yields
$$I_{1}\lesssim I_{1,1}+I_{1,2}<\infty.$$
Combining the estimate $I_{1}$ with $I_{2}$ we reach the desired conclusion (\ref{NN}).
\end{proof}
\begin{remark}
We presently do not know whether
$$\mathfrak{B}_{p,q}^{\mathcal{L},\beta}(\mathbb{G}^{n}_{\alpha})=
B_{p,q}^{\mathcal{L},\beta}(\mathbb{G}^{n}_{\alpha})$$
 holds due to
the fact that the metric on the Grushin space does not  satisfy the
translation invariance. However, if $\alpha=0$ and
$\mathcal{L}=-\Delta,$ then
$$\mathfrak{B}_{p,q}^{-\Delta,\beta}(\mathbb{R}^{n})=
B_{p,q}^{-\Delta,\beta}(\mathbb{R}^{n}).$$ Furthermore,  if
$\alpha=0$ and we replace $\mathcal{L}$ by  a class of
Kolmogorov-Fokker-Planck operators on $\mathbb{R}^{n}$, then
$$\mathfrak{B}_{p,q}^{\mathcal{L},\beta}(\mathbb{R}^{n})=
B_{p,q}^{\mathcal{L},\beta}(\mathbb{R}^{n})$$ (see \cite[Proposition
3.2]{BGT2022}).
\end{remark}

\begin{proposition}\label{RB}
Let $(p,q,\beta)\in[1,\infty)\times[1,\infty]\times(0,\infty).$ Then $B_{p,q}^{\mathcal{L},\beta}(\mathbb{G}^{n}_{\alpha})$ is a Banach space.
\end{proposition}
\begin{proof}
We mainly consider the case $(p,q,\beta)\in[1,\infty)\times[1,\infty)\times(0,\infty)$, since the proof of $(p,q,\beta)\in[1,\infty)\times\{\infty\}\times(0,\infty)$ is similar.
It is easy to check that $\|\cdot\|_{B_{p,q}^{\mathcal{L},\beta}(\mathbb{G}^{n}_{\alpha})}$ is a norm.
Thus, we only need to prove the completeness of
$B_{p,q}^{\mathcal{L},\beta}(\mathbb{G}^{n}_{\alpha})$. Let
$\{u_{k}\}_{k\in\mathbb{N}}$ be a Cauchy sequence in
$B_{p,q}^{\mathcal{L},\beta}(\mathbb{G}^{n}_{\alpha}).$ Via the
definition of the norm of Besov spaces and the completeness of
$L^p(\mathbb{G}^{n}_{\alpha})$, we may assume that there
exists a function $u\in L^p(\mathbb{G}^{n}_{\alpha})$ such that
$$\|u_{k}-u\|_{L^p(\mathbb{G}^{n}_{\alpha})}\rightarrow0 \ \ \textrm{as}\ \ k\rightarrow\infty.$$
It follows from Minkowski's inequality,
  (\ref{SC}) and Proposition \ref{G-Pro} (iv) that
\begin{align*}
&\bigg|\bigg(\int_{\mathbb{G}^{n}_{\alpha}}e^{-t\mathcal{L}}(|u_{k}-u_{k}(g')|^{p})(g')dg'\bigg)^{1/p}-\bigg(\int_{\mathbb{G}^{n}_{\alpha}}
e^{-t\mathcal{L}}(|u-u(g')|^{p})(g')dg'\bigg)^{1/p}\bigg|\\
&\leq\bigg(\int_{\mathbb{G}^{n}_{\alpha}}e^{-t\mathcal{L}}(|(u_{k}-u)-(u_{k}(g')-u(g'))|^{p})(g')dg'\bigg)^{1/p}\\
&\leq\bigg(\int_{\mathbb{G}^{n}_{\alpha}}\int_{\mathbb{G}^{n}_{\alpha}}K_{t}(g,g')(|u_{k}(g)-u(g)|^{p})
dgdg'\bigg)^{1/p}\\
&\ \ \ +
\bigg(\int_{\mathbb{G}^{n}_{\alpha}}\int_{\mathbb{G}^{n}_{\alpha}}K_{t}(g,g')(|u_{k}(g')-u(g')|^{p})dgdg'\bigg)^{1/p}\\
&\lesssim\|u-u_{k}\|_{L^{p}(\mathbb{G}^{n}_{\alpha})}.
\end{align*}
Therefore,
$$\lim_{k\rightarrow\infty}\bigg(\int_{\mathbb{G}^{n}_{\alpha}}e^{-t\mathcal{L}}(|u_{k}-u_{k}(g')|^{p})(g')dg'\bigg)^{1/p}
=\bigg(\int_{\mathbb{G}^{n}_{\alpha}}e^{-t\mathcal{L}}(|u-u(g')|^{p})(g')dg'\bigg)^{1/p}.$$
Furthermore, via Fatou's lemma, we deduce that
\begin{align*}
&\bigg(\int_{0}^{\infty}\dfrac{1}{t^{\beta q/2}}\Big(\int_{\mathbb{G}^{n}_{\alpha}}e^{-t\mathcal{L}}(|u-u(g')|^{p})(g')dg'\Big)^{q/p}\dfrac{dt}{t}\bigg)^{1/q}\\
&\le \liminf_{k\rightarrow\infty}\bigg(\int_{0}^{\infty}\dfrac{1}{t^{{\beta q}/{2}}}\Big(\int_{\mathbb{G}^{n}_{\alpha}}e^{-t\mathcal{L}}(|u_{k}-u_{k}(g')|^{p})(g')dg'\Big)^{q/p}\dfrac{dt}{t}\bigg)^{1/q}\\
&\leq\liminf_{k\rightarrow\infty}N_{p,q}^{\mathcal{L},\beta}(u_{k})<\infty.
\end{align*}
Therefore, $u\in B_{p,q}^{\mathcal{L},\beta}(\mathbb{G}^{n}_{\alpha})$ and
$$N_{p,q}^{\mathcal{L},\beta}(u)\leq\liminf_{k\rightarrow\infty}N_{p,q}^{\mathcal{L},\beta}(u_{k}).$$
Similarly, for each fixed positive integer $m$,
$$N_{p,q}^{\mathcal{L},\beta}(u-u_{m})\leq\lim_{k\rightarrow\infty}N_{p,q}^{\mathcal{L},\beta}(u_{k}-u_{m}).$$
Thus, by taking the limit $m\rightarrow\infty$, together with the
fact that $\{u_{k}\}_{k\in\mathbb{N}}$ is a Cauchy sequence with respect to the
semi-norm $N_{p,q}^{\mathcal{L},\beta}(\cdot)$, we prove the
completeness of
$B_{p,q}^{\mathcal{L},\beta}(\mathbb{G}^{n}_{\alpha})$.
\end{proof}

The following results provide the min-max property of
$B_{p,q}^{\mathcal{L},\beta}(\mathbb{G}^{n}_{\alpha}).$
\begin{lemma}\label{max}
Let $u_{1},u_{2}\in
B_{p,q}^{\mathcal{L},\beta}(\mathbb{G}^{n}_{\alpha})$ and the functions
$$G:=\max\{u_{1},u_{2}\},\ H:=\min\{u_{1},u_{2}\}.$$ Then
the following statements hold:
\begin{itemize}
  \item[(i)] If $p=q\in[1,\infty)$ and $\beta>0,$ then $G, H$
  belong to
$B_{p,q}^{\mathcal{L},\beta}(\mathbb{G}^{n}_{\alpha})$ with
$$\|G\|^{q}_{B_{p,q}^{\mathcal{L},\beta}(\mathbb{G}^{n}_{\alpha})}+\|H\|^{q}_{B_{p,q}^{\mathcal{L},\beta}
(\mathbb{G}^{n}_{\alpha})}\leq
\|u_{1}\|^{q}_{B_{p,q}^{\mathcal{L},\beta}(\mathbb{G}^{n}_{\alpha})}
+\|u_{2}\|^{q}_{B_{p,q}^{\mathcal{L},\beta}(\mathbb{G}^{n}_{\alpha})}.$$
  \item[(ii)] If $p\in[1,\infty),q=\infty$ and $\beta>0,$ then
  the above results are also valid for $B_{p,\infty}^{\mathcal{L},\beta}(\mathbb{G}^{n}_{\alpha}).$
\end{itemize}

\end{lemma}
\begin{proof}
(i) We first prove the case $p=q\in[1,\infty)$ and $\beta>0.$ Let
\begin{align*}
  \left\{\begin{aligned}
  D_{1}&:=\Big\{g\in\mathbb{G}^{n}_{\alpha}:\ u_{1}(g)\geq u_{2}(g)\Big\};\\
  D_{2}&:=\Big\{g\in\mathbb{G}^{n}_{\alpha}:\ u_{1}(g)< u_{2}(g)\Big\}.
  \end{aligned}\right.
\end{align*}
 Then
\begin{align*}
\|G\|^{p}_{L^{p}(\mathbb{G}^{n}_{\alpha})}+\|H\|^{p}_{L^{p}(\mathbb{G}^{n}_{\alpha})}
&\leq\bigg(\int_{D_{1}}|u_{1}(g)|^{p}dg+\int_{D_{2}}|u_{2}(g)|^{p}dg\bigg) \\
&\qquad+\bigg(\int_{D_{1}}|u_{2}(g)|^{p}dg+\int_{D_{2}}|u_{1}(g)|^{p}dg\bigg) \\
&\leq\|u_{1}\|^{p}_{L^{p}(\mathbb{G}^{n}_{\alpha})}+\|u_{2}\|^{p}_{L^{p}(\mathbb{G}^{n}_{\alpha})}.
\end{align*}
Thus, we only need to prove
\begin{equation}\label{semi}
(N^{\mathcal{L},\beta}_{p,p}(G))^{p}+(N^{\mathcal{L},\beta}_{p,p}(H))^{p}\leq(N^{\mathcal{L},\beta}_{p,p}(u_{1}))^{p}
+(N^{\mathcal{L},\beta}_{p,p}(u_{2}))^{p}.
\end{equation}
By calculation, we have
\begin{align*}
(N^{\mathcal{L},\beta}_{p,p}(G))^{p}&=\int_{0}^{\infty}\Big(\int_{\mathbb{G}_{\alpha}^{n}}
e^{-t\mathcal{L}}(|G-G(g')|^{p})(g')dg'\Big)\dfrac{dt}{t^{{\beta
p}/{2}+1}}\\
&=\int_{0}^{\infty}\Big(\int_{\mathbb{G}_{\alpha}^{n}}\int_{\mathbb{G}_{\alpha}^{n}}K_{t}(g,g')|G(g)-G(g')|^{p}dgdg'
\Big) \dfrac{dt}{t^{{\beta p}/{2}+1}}\\
&\leq\int_{0}^{\infty}\Big(\int_{D_{1}}\int_{D_{1}}K_{t}(g,g')|u_{2}(g)-u_{2}(g')|^{p}dgdg'\Big)\dfrac{dt}{t^{{\beta p}/{2}+1}}\\
&\quad+\int_{0}^{\infty}\Big(\int_{D_{1}}\int_{D_{2}}K_{t}(g,g')|u_{1}(g)-u_{2}(g')|^{p}dgdg'\Big)\dfrac{dt}{t^{{\beta p}/{2}+1}}\\
&\quad+\int_{0}^{\infty}\Big(\int_{D_{2}}\int_{D_{1}}K_{t}(g,g')|u_{2}(g)-u_{1}(g')|^{p}dgdg'\Big)\dfrac{dt}{t^{{\beta p}/{2}+1}}\\
&\quad+\int_{0}^{\infty}\Big(\int_{D_{2}}\int_{D_{2}}K_{t}(g,g')|u_{1}(g)-u_{1}(g')|^{p}dgdg'\Big)\dfrac{dt}{t^{{\beta
p}/{2}+1}}
\end{align*}
and
\begin{align*}
(N^{\mathcal{L},\beta}_{p,p}(H))^{p}
&\leq\int_{0}^{\infty}\Big(\int_{D_{1}}\int_{D_{1}}K_{t}(g,g')|u_{1}(g)-u_{1}(g')|^{p}dgdg'\Big)  \dfrac{dt}{t^{{\beta p}/{2}+1}}\\
&\quad+\int_{0}^{\infty}\Big(\int_{D_{1}}\int_{D_{2}}K_{t}(g,g')|u_{2}(g)-u_{1}(g')|^{p}dgdg'\Big)\dfrac{dt}{t^{{\beta p}/{2}+1}}\\
&\quad+\int_{0}^{\infty}\Big(\int_{D_{2}}\int_{D_{1}}K_{t}(g,g')|u_{1}(g)-u_{2}(g')|^{p}dgdg'\Big) \dfrac{dt}{t^{{\beta p}/{2}+1}}\\
&\quad+\int_{0}^{\infty}\Big(\int_{D_{2}}\int_{D_{2}}K_{t}(g,g')|u_{2}(g)-u_{2}(g')|^{p}dgdg'\Big)\dfrac{dt}{t^{{\beta
p}/{2}+1}}.
\end{align*}
Then
\begin{align*}
(N^{\mathcal{L},\beta}_{p,p}(G))^{p}+(N^{\mathcal{L},\beta}_{p,p}(H))^{p}\leq I_{1}+I_{2}+I_{3}+I_{4},
\end{align*}
where
$$\left\{\begin{aligned}
I_{1}&:=\int_{0}^{\infty}\int_{D_{1}}\int_{D_{1}}K_{t}(g,g')|u_{1}(g)-u_{1}(g')|^{p}dgdg'\dfrac{dt}{t^{{\beta p}/{2}+1}}\\
&\quad +\int_{0}^{\infty}\int_{D_{2}}\int_{D_{2}}K_{t}(g,g')|u_{1}(g)-u_{1}(g')|^{p}dgdg'\dfrac{dt}{t^{{\beta p}/{2}+1}};\\
I_{2}&:=\int_{0}^{\infty}\Big(\int_{D_{2}}\int_{D_{1}}K_{t}(g,g')
(|u_{1}(g)-u_{2}(g')|^{p}+|u_{2}(g)-u_{1}(g')|^{p})dgdg'\Big) \dfrac{dt}{t^{{\beta p}/{2}+1}};\\
I_{3}&:=\int_{0}^{\infty}\int_{D_{1}}\int_{D_{1}}K_{t}(g,g')|u_{2}(g)-u_{2}(g')|^{p}dgdg'\dfrac{dt}{t^{{\beta p}/{2}+1}}\\
&\quad +\int_{0}^{\infty}\int_{D_{2}}\int_{D_{2}}K_{t}(g,g')|u_{2}(g)-u_{2}(g')|^{p}dgdg'\dfrac{dt}{t^{{\beta p}/{2}+1}};\\
I_{4}&:=\int_{0}^{\infty}\Big(\int_{D_{1}}\int_{D_{2}}K_{t}(g,g')(|u_{2}(g)-u_{1}(g')|^{p}+
          |u_{1}(g)-u_{2}(g')|^{p})dgdg'\Big)\dfrac{dt}{t^{ {\beta  p}/{2}+1}}.
\end{aligned}\right.$$

Next we deal with $I_{2}$ and $I_{4}$   using \cite[Lemma
3.3]{SW2010} (see also \cite[Lemma 2.2]{W2015}), which says that for
any convex and lower-semi-continuous function
$F:\mathbb{R}^{2}\rightarrow(-\infty,\infty),$ we have
$$F(r_{0},s_{1})+F(s_{0},r_{1})\leq F(r_{0},r_{1})+F(s_{0},s_{1})~\mbox{as}~(r_{0}-s_{0})(r_{1}-s_{1})\leq0.$$
If we take $F(r,s):=|r-s|^{p} $ for any $(r,s)\in \mathbb{R}^{2}$,
then for any $g\in D_{1}$ and $g'\in D_{2},$ we have
\begin{align}\label{ineq10}
|u_{2}(g)-u_{1}(g')|^{p}+|u_{1}(g)-u_{2}(g')|^{p}
&\leq F(u_{2}(g),u_{2}(g'))+F(u_{1}(g),u_{1}(g'))\nonumber\\
&=|u_{2}(g)-u_{2}(g')|^{p}+|u_{1}(g)-u_{1}(g')|^{p}.
\end{align}
Likewise, for any $g\in D_{2}$ and $g'\in D_{1},$ we obtain
\begin{align}\label{ineq11}
|u_{1}(g)-u_{2}(g')|^{p}+|u_{2}(g)-u_{1}(g')|^{p}\leq|u_{1}(g)-u_{1}(g')|^{p}+|u_{2}(g)-u_{2}(g')|^{p}.
\end{align}
Combining (\ref{ineq10}) with (\ref{ineq11}) gives
\begin{align*}
I_{2}+I_{4}
&\leq\int_{0}^{\infty}\int_{D_{1}}\int_{D_{2}}K_{t}(g,g')(|u_{1}(g)-u_{1}(g')|^{p}+|u_{2}(g)-u_{2}(g')|^{p})
dgdg'\dfrac{dt}{t^{{\beta
p}/{2}+1}}\\
&\ \ \ +\int_{0}^{\infty}\int_{D_{2}}\int_{D_{1}}K_{t}(g,g')(|u_{1}(g)-u_{1}(g')|^{p}+|u_{2}(g)-u_{2}(g')|^{p})dgdg'\dfrac{dt}{t^{{\beta
p}/{2}+1}}.
\end{align*}
This,   together  with the expressions of $I_{1}$ and $I_{3}$,
implies (\ref{semi}) immediately.

(ii)  Based on the proof process of (i), it is easy to verify that
the desired result is also valid for
$(p,q,\beta)\in[1,\infty)\times\{\infty\}\times(0,\infty).$
\end{proof}

Furthermore, we have the following Besov type embeddings.
\begin{proposition}\label{emb}
If $0<\beta\leq\beta',$ then for any $(p,q)\in[1,\infty)\times[1,\infty],$
$$B_{p,q}^{\mathcal{L},\beta'}(\mathbb{G}^{n}_{\alpha})
\hookrightarrow B_{p,q}^{\mathcal{L},\beta}(\mathbb{G}^{n}_{\alpha})$$
\end{proposition}
\begin{proof}
Assume that $(p,q)\in[1,\infty)\times[1,\infty).$
By using the hypothesis and the definition of $N_{p,q}^{\mathcal{L},\beta}(\cdot),$ we have
\begin{align*}
N_{p,q}^{\mathcal{L},\beta}(u)^{q}&=\int_{0}^{1}
\bigg(\int_{\mathbb{G}^{n}_{\alpha}}e^{-t\mathcal{L}}(|u-u(g)|^{p})(g)dg\bigg)^{q/p}\frac{dt}{t^{{\beta q}/{2}+1}}\\
&\quad+\int_{1}^{\infty}
\bigg(\int_{\mathbb{G}^{n}_{\alpha}}e^{-t\mathcal{L}}(|u-u(g)|^{p})(g)dg\bigg)^{q/p}\frac{dt}{t^{{\beta q}/{2}+1}}\\
&\leq\int_{0}^{1}
\bigg(\int_{\mathbb{G}^{n}_{\alpha}}e^{-t\mathcal{L}}(|u-u(g)|^{p})(g)dg\bigg)^{q/p}\frac{dt}{t^{{\beta' q}/{2}+1}}\\
&\quad +2^{p-1}\int_{1}^{\infty}
\bigg(\int_{\mathbb{G}^{n}_{\alpha}}\int_{\mathbb{G}^{n}_{\alpha}}
K_{t}(g,g')(|u(g')|^{p}+|u(g)|^{p})dg'dg\bigg)^{q/p}\frac{dt}{t^{{\beta q}/{2}+1}}\\
&\leq N_{p,q}^{\mathcal{L},\beta'}(u)^{q}+\frac{2^{p+{q}/{p}}}{\beta q}\|u\|^{q}_{L^{p}(\mathbb{G}^{n}_{\alpha})},
\end{align*}
where we have used the fact that (\ref{SC}) and Proposition \ref{G-Pro} (iv).
Therefore, this proves that
$$B_{p,q}^{\mathcal{L},\beta'}(\mathbb{G}^{n}_{\alpha})
\hookrightarrow B_{p,q}^{\mathcal{L},\beta}(\mathbb{G}^{n}_{\alpha}).$$
The case $(p,q)\in[1,\infty)\times\{\infty\}$ can be directly derived from Remark \ref{rem2} (ii).
\end{proof}

The following results imply the boundedness properties of $\mathcal{L}^{s}$ from Besov type spaces
$B_{p,q}^{\mathcal{L},\beta}(\mathbb{G}^{n}_{\alpha})$ into $L^{p}(\mathbb{G}^{n}_{\alpha})$ when $p=q.$
\begin{proposition}\label{boundedness}
Let $s\in(0,1).$
For $p\in(1,\infty)$ and $\beta>2s,$ we have
$$\mathcal{L}^{s}:B_{p,p}^{\mathcal{L},\beta}(\mathbb{G}^{n}_{\alpha})\hookrightarrow L^{p}(\mathbb{G}^{n}_{\alpha}).$$
For $p=1,$ the above results are also valid for $\beta\geq 2s.$
\end{proposition}
\begin{proof}
For $p\in[1,\infty)$ and $u\in B_{p,p}^{\mathcal{L},\beta}(\mathbb{G}^{n}_{\alpha}),$ we deduce from (\ref{-G}) that
\begin{align*}
\|\mathcal{L}^{s}u\|_{L^{p}(\mathbb{G}^{n}_{\alpha})}&\leq\frac{s}{\Gamma(1-s)}\int_{0}^{\infty}
t^{-(1+s)}\|e^{-t\mathcal{L}}u-u\|_{L^{p}(\mathbb{G}^{n}_{\alpha})}dt\\
&= I+II,
\end{align*}
where
\begin{align*}
  \left\{\begin{aligned}
  I&:=\frac{s}{\Gamma(1-s)}\int_{0}^{1}t^{-(1+s)}\|e^{-t\mathcal{L}}u-u\|_{L^{p}(\mathbb{G}^{n}_{\alpha})}dt;\\
  II&:=\frac{s}{\Gamma(1-s)}\int_{1}^{\infty}t^{-(1+s)}\|e^{-t\mathcal{L}}u-u\|_{L^{p}(\mathbb{G}^{n}_{\alpha})}dt.
  \end{aligned}\right.
\end{align*}
Using Proposition \ref{G-Pro} (iv), we obtain
$$II\lesssim \|u\|_{L^{p}(\mathbb{G}^{n}_{\alpha})}.$$
Next, we will consider the term $I.$
It follows from H\"older's inequality and (\ref{SC}) that
\begin{align}\label{ineq20}
\|e^{-t\mathcal{L}}u-u\|_{L^{p}(\mathbb{G}^{n}_{\alpha})}
&= \bigg(\int_{\mathbb{G}^{n}_{\alpha}}|e^{-t\mathcal{L}}u(g)-u(g)|^{p}dg\bigg)^{1/p}\nonumber\\
&=\bigg(\int_{\mathbb{G}^{n}_{\alpha}}\bigg|\int_{\mathbb{G}^{n}_{\alpha}}
K_{t}(g,g')(u(g')-u(g))dg'\bigg|^{p}dg\bigg)^{1/p}\nonumber\\
&\leq
\bigg(\int_{\mathbb{G}^{n}_{\alpha}}e^{-t\mathcal{L}}(|u-u(g)|^{p})(g)dg\bigg)^{1/p}.
\end{align}
If $p=1$ and $\beta\geq2s,$ we can obtain from (\ref{ineq20}) that
\begin{align*}
&\int_{0}^{1}t^{-(1+s)}\|e^{-t\mathcal{L}}u-u\|_{L^{1}(\mathbb{G}^{n}_{\alpha})}dt\\
&\leq \int_{0}^{1}\frac{t^{{\beta}/{2}-s}}{t^{{\beta}/{2}+1}}\Bigg(
\int_{\mathbb{G}^{n}_{\alpha}}e^{-t\mathcal{L}}(|u-u(g)|)(g)dg\Bigg)dt\\
&\leq \int_{0}^{1}\frac{1}{t^{{\beta}/{2}+1}}
\Bigg(\int_{\mathbb{G}^{n}_{\alpha}}e^{-t\mathcal{L}}(|u-u(g)|)(g)dg\Bigg)dt\\
&=\widetilde{N}_{1,1}^{\mathcal{L},\beta}(u)<\infty,
\end{align*}
where we have used  Remark \ref{rem2} (i) in the last step.

If $p\in(1,\infty)$ and $\beta>2s,$ H\"older's inequality and (\ref{ineq20}) gives
\begin{align*}
&\int_{0}^{1}t^{-(1+s)}\|e^{-t\mathcal{L}}u-u\|_{L^{p}(\mathbb{G}^{n}_{\alpha})}dt\\
&=\int_{0}^{1}t^{-s+{\beta}/{2}-{\beta}/{2}}\|e^{-t\mathcal{L}}u-u\|_{L^{p}(\mathbb{G}^{n}_{\alpha})}\frac{dt}{t}\\
&\leq\bigg(\int_{0}^{1}t^{-(s-{\beta}/{2})p'}\frac{dt}{t}\bigg)^{1/p'}
\bigg(\int_{0}^{1}t^{-{\beta p}/{2}}\|e^{-t\mathcal{L}}u-u\|^{p}_{L^{p}(\mathbb{G}^{n}_{\alpha})}\frac{dt}{t}\bigg)^{1/p}\\
&\leq\bigg(\int_{0}^{1}t^{-(s-{\beta}/{2})p'-1}dt\bigg)^{1/p'}\bigg(\int_{0}^{1}\frac{1}{t^{{\beta p}/{2}+1}}
\int_{\mathbb{G}^{n}_{\alpha}}e^{-t\mathcal{L}}(|u-u(g)|^p{})(g)dg\bigg)^{1/p}\\
&\lesssim \widetilde{N}_{p,p}^{\mathcal{L},\beta}(u)<\infty.
\end{align*}
In view of Remark \ref{rem2} (i), we obtain the desired conclusion.
\end{proof}

Moreover, if
$(p,q,\beta)\in[1,\infty)\times\{\infty\}\times(2s,\infty)$, we can
conclude that $\mathcal{L}^{s}$ is bounded from
$B_{p,\infty}^{\mathcal{L},\beta}(\mathbb{G}^{n}_{\alpha})$ into
$L^{p}(\mathbb{G}^{n}_{\alpha}).$
\begin{proposition}\label{boundedness1}
Let $s\in(0,1).$
For $(p,q,\beta)\in[1,\infty)\times\{\infty\}\times(2s,\infty)$ we have
$$\mathcal{L}^{s}:B_{p,\infty}^{\mathcal{L},\beta}(\mathbb{G}^{n}_{\alpha})\hookrightarrow L^{p}(\mathbb{G}^{n}_{\alpha}).$$
\end{proposition}
\begin{proof}
For any $u\in B_{p,\infty}^{\mathcal{L},\beta}(\mathbb{G}^{n}_{\alpha}),$
it follows from (\ref{ineq20}) that
$$\|e^{-t\mathcal{L}}u-u\|_{L^{p}(\mathbb{G}^{n}_{\alpha})}\leq t^{\beta/2}N_{p,\infty}^{\mathcal{L},\beta}(u),$$
which, together with (\ref{-G}) and Proposition \ref{G-Pro} (iv), yields
\begin{align*}
\|\mathcal{L}^{s}u\|_{L^{p}(\mathbb{G}^{n}_{\alpha})}
&\leq\frac{s}{\Gamma(1-s)}\int_{0}^{\infty}
t^{-1-s}\|e^{-t\mathcal{L}}u-u\|_{L^{p}(\mathbb{G}^{n}_{\alpha})}dt\\
&\lesssim \int_{0}^{1}t^{-1-s}\|e^{-t\mathcal{L}}u-u\|_{L^{p}(\mathbb{G}^{n}_{\alpha})}dt
+\int_{1}^{\infty}t^{-1-s}\|e^{-t\mathcal{L}}u-u\|_{L^{p}(\mathbb{G}^{n}_{\alpha})}dt\\
&\lesssim N^{\mathcal{L},\beta}_{p,\infty}(u)\int_{0}^{1}t^{-1-s+\beta/2}dt
+\|u\|_{L^{p}(\mathbb{G}^{n}_{\alpha})}\int_{1}^{\infty}t^{-1-s}dt\\
&\lesssim N^{\mathcal{L},\beta}_{p,\infty}(u)+\|u\|_{L^{p}(\mathbb{G}^{n}_{\alpha})}<\infty.
\end{align*}
This completes the proof.
\end{proof}

Combining Proposition \ref{boundedness}, Proposition
\ref{boundedness1} with Definition \ref{STS}, we can directly obtain
the following corollary.
\begin{corollary}\label{BW}
Let $s\in(0,1).$ The following statements   hold:
\begin{itemize}
  \item [(i)]  If $p=q\in(1,\infty)$ and $\beta>2s$ then $$B_{p,p}^{\mathcal{L},\beta}(\mathbb{G}^{n}_{\alpha})\hookrightarrow \mathcal{W}^{2s,p}(\mathbb{G}_{\alpha}^{n}).$$
                  In particular, when $p=q=1$ and $\beta\geq2s,$ we have
                  $$B_{1,1}^{\mathcal{L},\beta}(\mathbb{G}_{\alpha}^{n})\hookrightarrow \mathcal{W}^{2s,1}(\mathbb{G}_{\alpha}^{n}).$$
  \item [(ii)] If $p\in[1,\infty),q=\infty$ and $\beta>2s$ than
  $$B_{p,\infty}^{\mathcal{L},\beta}(\mathbb{G}^{n}_{\alpha})\hookrightarrow \mathcal{W}^{2s,p}(\mathbb{G}_{\alpha}^{n}).$$
\end{itemize}
\end{corollary}

In closing, we provide the connection between
$B_{p,q}^{\mathcal{L},\beta}(\mathbb{G}^{n}_{\alpha})$ and Besov
spaces  $B_{p,q}^{\beta}(\mathbb{G}^{n}_{\alpha})$ defined via the
difference. To be precise, for any $p,q\in[1,\infty)$ and $\beta>0,$
the space $B_{p,q}^{\beta}(\mathbb{G}^{n}_{\alpha})$ is defined to
be the collection of all functions $u\in
L^{p}(\mathbb{G}_{\alpha}^{n})$ such that
\begin{align}\label{Dbesov}
N_{p,q}^{\beta}(u):=\Big(\int_{0}^{\infty}\Big(\int_{\mathbb{G}_{\alpha}^{n}}\int_{B(g,r)}\frac{|u(g)-u(g')|^{p}}{r^{2\beta p}|B(g,r)|}dg'dg\Big)^{q/p}\frac{dr}{r}\Big)^{1/q}<\infty,
\end{align}
When $q=\infty$, then $B_{p,\infty}^{\beta}(\mathbb{G}^{n}_{\alpha})$ is defined to be the collection of all
$u\in L^{p}(\mathbb{G}_{\alpha}^{n})$ such that
$$N_{p,\infty}^{\beta}(u):=\sup_{r>0}\Big(\int_{\mathbb{G}_{\alpha}^{n}}\int_{B(g,r)}\frac{|u(g)-u(g')|^{p}}{r^{2\beta p}|B(g,r)|}dg'dg\Big)^{1/p}<\infty.$$
Notice that this definition does not depend on the Grushin-Laplace
operator $\mathcal{L}.$

The space $B_{p,q}^{\beta}(\mathbb{G}^{n}_{\alpha})$ can be endowed
with the norm
$$\|u\|_{B_{p,q}^{\beta}(\mathbb{G}^{n}_{\alpha})}
:=\|u\|_{L^{p}(\mathbb{G}^{n}_{\alpha})}+N_{p,q}^{\beta}(u).$$
\begin{theorem}\label{com-1}
Let $p\in[1,\infty),q\in[1,\infty]$ and $\beta>0$. Then we have
$$B_{p,q}^{\mathcal{L},2\beta}(\mathbb{G}^{n}_{\alpha})=B_{p,q}^{\beta}(\mathbb{G}^{n}_{\alpha}).$$
\end{theorem}
\begin{proof}
We only consider the case of $p,q\in[1,\infty)$ here, since the
proof of $p\in[1,\infty),q=\infty$ is similar. We first prove
$$B_{p,q}^{\mathcal{L},2\beta}(\mathbb{G}^{n}_{\alpha})\subseteq
B_{p,q}^{\beta}(\mathbb{G}^{n}_{\alpha}).$$
 Using the definition of
the spaces $B_{p,q}^{\mathcal{L},\beta}(\mathbb{G}^{n}_{\alpha})$
and the Gaussian lower bound (\ref{Kt}),
\begin{align*}
   &\int_{0}^{\infty}\frac{1}{t^{\beta q}}\Big(\int_{\mathbb{G}_{\alpha}^{n}}\int_{\mathbb{G}_{\alpha}^{n}}|u(g)-u(g')|^{p}K_{t}(g,g')dg'dg\Big)^{q/p}\frac{dt}{t}  \\
   &\geq \int_{0}^{\infty}\frac{1}{t^{\beta q}}\Big(\int_{\mathbb{G}_{\alpha}^{n}}\int_{B(g,\sqrt{t})}|u(g)-u(g')|^{p}K_{t}(g,g')dg'dg\Big)^{q/p}\frac{dt}{t}  \\
   & \gtrsim\int_{0}^{\infty}\Big(\int_{\mathbb{G}_{\alpha}^{n}}\int_{B(g,\sqrt{t})}\frac{|u(g)-u(g')|^{p}}{t^{\beta p}|B(g,\sqrt{t})|}dg'dg\Big)^{q/p}\frac{dt}{t}.
\end{align*}
By the change of variables $r=\sqrt{t}$,  we have
\begin{align*}
&\int_{0}^{\infty}\frac{1}{t^{\beta q}}\Big(\int_{\mathbb{G}_{\alpha}^{n}}\int_{\mathbb{G}_{\alpha}^{n}}|u(g)-u(g')|^{p}K_{t}(g,g')dg'dg\Big)^{q/p}\frac{dt}{t}  \\
&\gtrsim\int_{0}^{\infty}\Big(\int_{\mathbb{G}_{\alpha}^{n}}\int_{B(g,r)}\frac{|u(g)-u(g')|^{p}}{r^{2\beta p}|B(g,r)|}dg'dg\Big)^{q/p}\frac{dr}{r}\\
  & =(N_{p,q}^{2\beta}(u))^{q}.
\end{align*}
Therefore,
$$B_{p,q}^{\mathcal{L},2\beta}(\mathbb{G}^{n}_{\alpha})\subseteq B_{p,q}^{\beta}(\mathbb{G}^{n}_{\alpha}).$$

Conversely,
it follows from the Gaussian upper bound
(\ref{Kt}) and (\ref{Re}) that
\begin{align*}
&\int_{0}^{\infty}\frac{1}{t^{\beta q}}\Big(\int_{\mathbb{G}_{\alpha}^{n}}\int_{\mathbb{G}_{\alpha}^{n}}|u(g)-u(g')|^{p}K_{t}(g,g')dg'dg\Big)^{q/p}\frac{dt}{t}\\
&\lesssim\int_{0}^{\infty}\frac{1}{t^{\beta q}}\Big(\int_{\mathbb{G}_{\alpha}^{n}}\int_{d_{\alpha}(g,g')<\sqrt{t}}\frac{|u(g)-u(g')|^{p}}{|B(g,\sqrt{t})|}dg'dg\Big)^{q/p}\frac{dt}{t}\\
&\quad+ \int_{0}^{\infty}\frac{1}{t^{\beta q}}\Big(\int_{\mathbb{G}_{\alpha}^{n}}\sum_{i=0}^{\infty}\int_{2^{i}\sqrt{t}
\leq d(g,g')< 2^{i+1}\sqrt{t}}\frac{|u(g)-u(g')|^{p}e^{-C_{2}4^{i}}}{|B(g,\sqrt{t})|}dg'dg\Big)^{q/p}\frac{dt}{t}\\
&\lesssim \sum_{i=0}^{\infty}e^{-C_{2}4^{i}}2^{iQq/p}\int_{0}^{\infty}\frac{1}{t^{\beta q}}\Big(\int_{\mathbb{G}_{\alpha}^{n}e}\int_{B(g, 2^{i}\sqrt{t})}\frac{|u(g)-u(g')|^{p}}{|B(g,2^{i}\sqrt{t})|}dg'dg\Big)^{q/p}\frac{dt}{t}+(N_{p,q}^{2\beta}(u))^{q}.
\end{align*}
Using the change of variables $r=\sqrt{t}$ again, we obtain
\begin{align*}
&\int_{0}^{\infty}\frac{1}{t^{\beta q}}\Big(\int_{\mathbb{G}_{\alpha}^{n}}\int_{\mathbb{G}_{\alpha}^{n}}|u(g)-u(g')|^{p}K_{t}(g,g')dg'dg\Big)^{q/p}\frac{dt}{t}\\
&\lesssim\sum_{i=0}^{\infty}e^{-C_{2}4^{i}}2^{i(Qq/p+2\beta q)}\int_{0}^{\infty}\Big(\int_{\mathbb{G}_{\alpha}^{n}}\int_{B(g,r)}\frac{|u(g)-u(g')|^{p}}{r^{2\beta p}|B(g,r)|}dg'dg\Big)^{q/p}\frac{dr}{r}+(N_{p,q}^{2\beta}(u))^{q}\\
&\lesssim (N_{p,q}^{2\beta}(u))^{q},
\end{align*}
which implies $B_{p,q}^{\beta}(\mathbb{G}^{n}_{\alpha})\subseteq
B_{p,q}^{\mathcal{L},2\beta}(\mathbb{G}^{n}_{\alpha})$. Therefore,
we obtain the desired result.
\end{proof}

\subsection{Besov type spaces associated with $\{e^{-t\mathcal{L}^{s}}\}_{t>0}$}\label{sec-3.2}
Below we investigate the fractional semigroup associated with the
Grushin-Laplace operator $\mathcal{L}.$ In the case of
$\alpha=0,\mathcal{L}=-\Delta,$ the fractional heat semigroup
$\{e^{-t(-\Delta)^{s}}\}_{t>0}$  plays a significant role in many
fields of mathematics, such as  harmonic analysis and partial
differential equations. Due to the  significance and backgrounds in
mathematical physics, the fractional heat semigroup
$\{e^{-t(-\Delta)^{s}}\}_{t>0}$
 has also been applied to  study a wide class of
physical systems and engineering problems,  including  L\'{e}vy
flights,  stochastic  interfaces  and  anomalous  diffusion
problems. In \cite{miao}, Miao, Yuan and Zhang
investigated the regularity of the fractional heat kernel
$e^{-t(-\Delta)^{s}}(\cdot)$ related with $(-\Delta)$ which is
defined via the Fourier transform
  as follows: for $x\in \mathbb{R}^{n}$,
\begin{equation*}
  e^{-t(-\Delta)^{s}}(x)=(2\pi)^{-n/2}\int_{\mathbb{R}^{n}}e^{-t|\xi|^{2s}}e^{ix\xi}d\xi,\  s\in (0,1).
\end{equation*}
and has the following regularity estimate:
$$e^{-t(-\Delta)^{s}}(x)\leq \frac{Ct}{(t^{1/2s}+|x|)^{n+2s}}\ \ \ \  \forall(x,t)\in \mathbb{R}^{n+1}_{+}.$$

  Unlike the case of the Laplace operator,  the fractional heat semigroup
related with $\mathcal{L}$ can be introduced via the following
formulation according to \cite[Section 5.4]{gri} (see also \cite{GL2015}). For $s\in (0,1)$, the
integral kernel $K^{s}_{t}(\cdot,\cdot)$ of $e^{-t\mathcal{L}^{s}}$
can be expressed as
\begin{equation}\label{k-1}
  K^{s}_{t}(g,g')=\int^{\infty}_{0}\eta^{s}_{t}(\sigma)K_{\sigma}(g,g')d\sigma\ \ \forall\ g,g'\in\mathbb{G}_{\alpha}^{n},
\end{equation}
where $\eta_{t}^{s}(\cdot)$ is a non-negative continuous
function on $(0,\infty)$ satisfying
\begin{equation}\label{eq-1.1}
\left\{ \begin{aligned}
 &\int^{\infty}_{0}\eta^{s}_{t}(\sigma)d\sigma=1;\\
 &\eta^{s}_{t}(\sigma)=\frac{1}{t^{1/s}}\eta_{1}^{s}(\frac{\sigma}{t^{1/s}});\\
 &\eta^{s}_{t}(\sigma)\lesssim \frac{t}{\sigma^{1+s}}\ \forall \sigma,t>0;\\
 &\int^{\infty}_{0}\sigma^{-r}\eta^{s}_{1}(\sigma)d\sigma<\infty  \ \ \forall r>0;\\
 &\eta^{s}_{t}(\sigma)\thicksim \frac{t}{\sigma^{1+s}}\ \ \forall \sigma\geq t^{1/s}>0.
\end{aligned} \right.
\end{equation}
With the help of (\ref{k-1}), we define the fractional semigroup
associated with the Grushin-Laplace  operator $\mathcal{L}$ as
follows
\begin{align}\label{sub}
e^{-t\mathcal{L}^{s}}u(g)=\int_{\mathbb{G}_{\alpha}^{n}}K_{t}^{s}(g,g')u(g')dg',
\ \ u\in\mathscr{S}(\mathbb{G}_{\alpha}^{n}).
\end{align}
 Using (\ref{k-1}) and exchanging
the order of integration in (\ref{sub}), we obtain the following
useful representation of the semigroup
$\{e^{-t\mathcal{L}^{s}}\}_{t>0}$ in terms of the Grushin
semigroup
\begin{align}\label{possion}
e^{-t\mathcal{L}^{s}}u(g)=\int_{0}^{\infty}\eta_{t}^{s}(\sigma)
e^{-\sigma\mathcal{L}}u(g)d\sigma.
\end{align}

\begin{remark}\label{1/2}
In Section \ref{sec-4.2}, we are primarily interested in $s=1/2.$
In this case, we have
$$\eta^{1/2}_{t}(\sigma)=\frac{t}{2\sigma^{3/2}}e^{-{t^{2}}/{(4\sigma)}}\ \ \forall\ t,\sigma>0.$$
It is easy to verify that such $\eta_{t}^{1/2}(\cdot)$ satisfies
 (\ref{eq-1.1}) (see \cite[Section 5.4]{gri}).
\end{remark}

Some properties on Grushin semigroups are also valid for the semigroup $\{e^{-t\mathcal{L}^{s}}\}_{t>0}.$
\begin{proposition}\label{Poi}
Let $s\in(0,1).$ The following properties hold:
\begin{itemize}
  \item [(i)] For any $g\in\mathbb{G}_{\alpha}^{n}$ and $t>0,$ we have $e^{-t\mathcal{L}^{s}}1(g)=1;$
  \item [(ii)] Let $p\in[1,\infty].$ Then $e^{-t\mathcal{L}^{s}}$ is  a contraction semigroup on $L^{p}(\mathbb{G}_{\alpha}^{n})$ for any $t>0;$
  \item [(iii)] Let $p\in[1,\infty)$ and $u\in L^{p}(\mathbb{G}^{n}_{\alpha}).$
                 Then there exists a positive constant $C(p,Q),$ depending on $p$ and $Q,$ such that
                 for any $g\in \mathbb{G}^{n}_{\alpha}$ and $t>0,$
                 \begin{align*}
                |e^{-t\mathcal{L}^{s}}u(g)|\leq C(p,Q)\|u\|_{L^{p}(\mathbb{G}^{n}_{\alpha})}
                \int_{0}^{\infty}\eta_{t}^{s}(\sigma)t^{-Q/2p}d\sigma.
                 \end{align*}
          Specifically, if $s=1/2,$ then
                \begin{align*}
                |e^{-t\sqrt{\mathcal{L}}}u(g)|\leq C(p,Q) t^{-Q/p}\|u\|_{L^{p}(\mathbb{G}^{n}_{\alpha})}.
                 \end{align*}
\end{itemize}
\end{proposition}
\begin{proof}
 These results can be obtained immediately from (\ref{SC}), Proposition \ref{G-Pro} (iv) and
Proposition \ref{lem4} by using (\ref{possion}) and (\ref{eq-1.1}).
\end{proof}

We now introduce the Besov classes based on the fractional semigroup
$\{e^{-t\mathcal{L}^{s}}\}_{t>0}$ associated with the
Grushin-Laplace operator $\mathcal{L}.$
\begin{definition}\label{FBesov}
Let $p,q\in[1,\infty), \beta>0$ and $s\in(0,1).$ The Besov type space
$B_{p,q}^{\mathcal{L}^{s},\beta}(\mathbb{G}^{n}_{\alpha})$ is defined to be the collection
of all $u\in L^{p}(\mathbb{G}_{\alpha}^{n})$ such that
$$N_{p,q}^{\mathcal{L}^{s},\beta}(u):=\Big(\int_{0}^{\infty}t^{-{\beta q}/{2}}\Big(\int_{\mathbb{G}_{\alpha}^{n}}e^{-t\mathcal{L}^{s}}(|u-u(g)|^{p})(g)dg\Big)^{q/p}\frac{dt}{t}\Big)^{1/q}<\infty,$$
where $\{e^{-t\mathcal{L}^{s}}\}_{t>0}$ is the fractional
semigroup associated with $\mathcal{L}.$

If $q=\infty,$ then
$B_{p,\infty}^{\mathcal{L}^{s},\beta}(\mathbb{G}^{n}_{\alpha})$ is
defined to be the collection of all $u\in
L^{p}(\mathbb{G}_{\alpha}^{n})$ such that
$$N_{p,\infty}^{\mathcal{L}^{s},\beta}(u):=\sup_{t>0}t^{-{\beta}/{2}}\Big(\int_{\mathbb{G}_{\alpha}^{n}}e^{-t\mathcal{L}^{s}}(|u-u(g)|^{p})(g)dg\Big)^{1/p}<\infty.$$
The space $B_{p,q}^{\mathcal{L}^{s},\beta}(\mathbb{G}^{n}_{\alpha})$
is endowed with the norm
$$\|u\|_{B_{p,q}^{\mathcal{L}^{s},\beta}(\mathbb{G}^{n}_{\alpha})}=\|u\|_{L^{p}(\mathbb{G}_{\alpha}^{n})}+N_{p,q}^{\mathcal{L}^{s},\beta}(u).$$
\end{definition}
\begin{remark}
Via Proposition \ref{possion} (i) \& (ii),
we can similarly prove general properties of Besov type spaces associated with $\{e^{-t\mathcal{L}^{s}}\}_{t>0},$ such as completeness (see Proposition \ref{RB}), the min-max property (see Lemma \ref{max}) and embedding properties (see Proposition \ref{emb}).
\end{remark}

In what follows, we can use the formula (\ref{k-1}) to get the regularity
estimate of the fractional semigroup
$\{e^{-t\mathcal{L}^{s}}\}_{t>0}$ related to the Grushin
operator $\mathcal{L}.$

\begin{lemma}\label{wang-1}
Let $s\in(0,1)$.  For $g,g'\in\mathbb{G}_{\alpha}^{n}$ and $t>0$, then we have
$$K^{s}_{t}(g,g') \thicksim \frac{1}{|B(g,t^{1/(2s)}+d_{\alpha}(g,g'))|}\frac{t}{(t^{1/(2s)}+d_{\alpha}(g,g'))^{2s}}.$$
\end{lemma}
\begin{proof}
It follows from the definition (\ref{k-1}) of the integral kernel
$K^{s}_{t}(\cdot,\cdot)$, the Gaussian upper bound (\ref{Kt}) and (\ref{eq-1.1}) that
\begin{align*}
K^{s}_{t}(g,g')&\lesssim\int^{\infty}_{0}\frac{t}{\sigma^{1+s}}K_{\sigma}(g,g')d\sigma \\
&\lesssim\int^{\infty}_{0}\frac{t}{\sigma^{1+s}}\frac{1}{|B(g,\sqrt{\sigma})|}e^{{-C_{2}d_{\alpha}(g,g')^{2}/\sigma}}d\sigma.
\end{align*}
By the change of variables $\sigma=t^{1/s}u$, we get
\begin{align*}
K^{s}_{t}(g,g')&\lesssim\int^{\infty}_{0}\frac{t^{1+1/s}}{(t^{1/s}u)^{1+s}}
\frac{e^{{-C_{2}d_{\alpha}(g,g')^{2}/(t^{1/s}u)}}}{|B(g, {t^{1/(2s)}\sqrt{u}})|}du\\
&\lesssim
\int^{\infty}_{0}\frac{e^{{-C_{2}d_{\alpha}(g,g')^{2}/(t^{1/s}u)}}}{|B(g, {t^{1/(2s)}\sqrt{u}})|}u^{-(1+s)}du.
\end{align*}
Letting $d_{\alpha}(g,g')^{2}/t^{1/s}u=r^{2}$ reaches to
\begin{align*}
K^{s}_{t}(g,g')&\lesssim
\int^{\infty}_{0}\frac{e^{-C_{2}r^{2}}}{|B(g,
d_{\alpha}(g,g')/r)|}\frac{tr^{2s-1}}{d_{\alpha}(g,g')^{2s}}
dr \lesssim I_{1}+I_{2},
\end{align*}
where
\begin{align*}
  \left\{\begin{aligned}
  I_{1}&:=\frac{t}{d_{\alpha}(g,g')^{2s}}
         \int^{1}_{0}\frac{r^{2s-1}e^{-C_{2}r^{2}}}{|B(g,
           d_{\alpha}(g,g')/r)|}dr;\\
  I_{2}&:=\frac{t}{d_{\alpha}(g,g')^{2s}}\int^{\infty}_{1}\frac{r^{2s-1}e^{-C_{2}r^{2}}}{|B(g,
           d_{\alpha}(g,g')/r)|}dr.
  \end{aligned}\right.
\end{align*}
For $I_{2}$, since $r\geq1$,   (\ref{Re}) can be utilized to derive
\begin{align*}
\frac{1}{|B(g, {d_{\alpha}(g,g')/r})|}
&=\frac{1}{|B(g, {d_{\alpha}(g,g')})|}\frac{|B(g, {rd_{\alpha}(g,g')/r})|}{|B(g,{d_{\alpha}(g,g')/r})|}\\
&\lesssim\frac{r^{Q}}{|B(g, {d_{\alpha}(g,g')})|},
\end{align*}
which further yields
\begin{eqnarray*}
I_{2}
\lesssim\frac{{t}{d_{\alpha}(g,g')^{-2s}}}{|B(g, {d_{\alpha}(g,g')})|}\int^{\infty}_{1}r^{2s+Q-1}e^{-C_{2}r^{2}}dr
\lesssim\frac{td_{\alpha}(g,g')^{-2s}}{|B(g, {d_{\alpha}(g,g')})|}.
\end{eqnarray*}
 It remains to prove $I_{1}$. Since $0<r<1$, using (\ref{Re}), we also have
 $$|B(g, {d_{\alpha}(g,g')/r})|\gtrsim r^{-n}|B(g,{d_{\alpha}(g,g')})|,$$
which implies
\begin{eqnarray*}
I_{1}\lesssim\frac{{t}{d_{\alpha}(g,g')^{-2s}}}{|B(g, {d_{\alpha}(g,g')})|}\int^{1}_{0}r^{2s-1+n}e^{-C_{2}r^{2}}dr
\lesssim\frac{{t}{d_{\alpha}(g,g')^{-2s}}}{|B(g, {d_{\alpha}(g,g')})|}.
\end{eqnarray*}
Therefore, we obtain
\begin{equation}\label{3.4}
K^{s}_{t}(g,g')\lesssim\frac{1}{|B(g, {d_{\alpha}(g,g')})|}\frac{t}{d_{\alpha}(g,g')^{2s}}.
\end{equation}
On the other hand, noting that
$$K^{s}_{t}(g,g')\lesssim\int^{\infty}_{0}\frac{1}{|B(g,\sqrt{\sigma})|}\frac{1}{t^{1/s}}
\eta^{s}_{1}(\sigma/t^{1/s})d\sigma,$$
we can apply the change of variables $\tau=\sigma/t^{1/s}$ to  get
\begin{align*}
K^{s}_{t}(g,g')&\lesssim
\int^{\infty}_{0}\frac{1}{|B(g,\sqrt{\tau t^{1/s}})|}\frac{1}{t^{1/s}}
\eta^{s}_{1}(\tau)t^{1/s}d\tau\\
&\lesssim\int^{\infty}_{0}\frac{1}{|B(g,\sqrt{\tau t^{1/s}})|}
\eta^{s}_{1}(\tau)d\tau\\
&\lesssim I_{3}+I_{4},
\end{align*}
where
\begin{align*}
  \left\{\begin{aligned}
  I_{3}&:=\int^{1}_{0}\frac{1}{|B(g,\sqrt{\tau t^{1/s}})|}\eta^{s}_{1}(\tau)d\tau;\\
  I_{4}&:=\int^{\infty}_{1} \frac{1}{|B(g,\sqrt{\tau t^{1/s}})|}\eta^{s}_{1}(\tau)d\tau.
  \end{aligned}\right.
\end{align*}
Now we are in a position to deal with  $I_{3}$. By the fact that
$0<\tau<1$ and (\ref{Re}), we obtain
$$I_{3}\lesssim\frac{1}{|B(g, {t^{1/(2s)}})|}\int^{1}_{0}\tau^{-n/2}\eta^{s}_{1}(\tau)d\tau
\lesssim\frac{1}{|B(g, {t^{1/(2s)}})|}.$$ For $I_{4}$,
since $\tau\geq1$, we also conclude  that
$$I_{4}\lesssim\frac{1}{|B(g,{t^{1/(2s)}})|}\int^{\infty}_{1}\tau^{-Q/2}\eta^{s}_{1}(\tau)d\tau
\lesssim\frac{1}{|B(g,{t^{1/(2s)}})|}.$$
Therefore, the above estimates imply that
\begin{equation}\label{3.5}
K^{s}_{t}(g,g')\lesssim\frac{1}{|B(g, {t^{1/(2s)}})|}.
\end{equation}
Below we divide the range of $t^{1/(2s)}$ into two cases.

Case 1: If $t^{1/(2s)}\leq d_{\alpha}(g,g'),$ then $t^{1/(2s)}+d_{\alpha}(g,g')\leq 2d_{\alpha}(g,g').$ It follows from (\ref{3.4})
and (\ref{Re}) that
\begin{align*}
K^{s}_{t}(g,g')
&\lesssim\frac{1}{|B(g, {d_{\alpha}(g,g')})|}\frac{t}{d_{\alpha}(g,g')^{2s}}\\
&\lesssim\frac{1}{|B(g,t^{1/(2s)}+d_{\alpha}(g,g'))|}\frac{t}{({t^{1/(2s)}+d_{\alpha}(g,g')})^{2s}}.
\end{align*}

Case 2: If $t^{1/(2s)}>d_{\alpha}(g,g'),$ then $2t^{1/(2s)}>d_{\alpha}(g,g')+t^{1/(2s)}.$
By (\ref{3.5}) and (\ref{Re}), we get
\begin{align*}
K^{s}_{t}(g,g')
&\lesssim\frac{1}{|B(g, {t^{1/(2s)}})|}\frac{t}{t}\\
&\lesssim\frac{1}{|B(g,t^{1/(2s)}+d_{\alpha}(g,g'))|}\frac{t}{({t^{1/(2s)}+d_{\alpha}(g,g')})^{2s}}.
\end{align*}

Hence, in any case, we have
$$K^{s}_{t}(g,g')\lesssim\frac{1}{|B(g,t^{1/(2s)}+d_{\alpha}(g,g'))|}\frac{t}{( {t^{1/(2s)}+d_{\alpha}(g,g')})^{2s}}.$$

For lower bounds, we deduce from (\ref{k-1}), (\ref{eq-1.1}) and (\ref{Kt}) that
\begin{align*}
K^{s}_{t}(g,g')&\gtrsim\int^{\infty}_{\max\{t^{1/s},\, d_{\alpha}(g,g')^{2}\}}\frac{t}{\sigma^{1+s}}|B(g,\sqrt{\sigma})|^{-1}
e^{-C_{1}d_{\alpha}(g,g')^{2}/\sigma}d\sigma\\
&\gtrsim\int^{\infty}_{\max\{t^{1/s},\, d_{\alpha}(g,g')^{2}\}}\frac{t}{\sigma^{1+s}}|B(g,\sqrt{\sigma})|^{-1}d\sigma\\
&\gtrsim t |B(g,\sqrt{\max\{t^{1/s},d_{\alpha}(g,g')^{2}\}})|^{-1}\int^{2\max\{t^{1/s},\, d_{\alpha}(g,g')^{2}\}}_{\max\{t^{1/s},\, d_{\alpha}(g,g')^{2}\}}\frac{1}{\sigma^{1+s}}d\sigma\\
&\thicksim\frac{t}{|B(g,t^{1/(2s)}+d_{\alpha}(g,g'))|}\min\Big\{\frac{1}{t},\ \frac{1}{d_{\alpha}(g,g')^{2s}}\Big\}\\
&\gtrsim\frac{1}{|B(g,t^{1/(2s)}+d_{\alpha}(g,g'))|}\frac{t}{(t^{1/(2s)}+d_{\alpha}(g,g'))^{2s}},
\end{align*}
which gives the desired result.
\end{proof}

Similarly to the case of
$B_{p,q}^{\mathcal{L},\beta}(\mathbb{G}^{n}_{\alpha})$, we also
establish the relationship between Besov type spaces
$B_{p,q}^{\mathcal{L}^{s},\beta}(\mathbb{G}^{n}_{\alpha})$ in
Definition \ref{FBesov} and Besov classes
$B_{p,q}^{\beta}(\mathbb{G}^{n}_{\alpha})$ defined via the
difference (see (\ref{Dbesov})).

\begin{theorem}\label{com-2}
Let $(p,q,\beta)\in[1,\infty)\times[1,\infty]\times(0,1/p)$. Then for $s\in(0,1),$ we have
$$B_{p,q}^{\mathcal{L}^{s},2\beta}(\mathbb{G}^{n}_{\alpha})=B_{p,q}^{s\beta}(\mathbb{G}^{n}_{\alpha}).$$
\end{theorem}
\begin{proof}
We only need to prove the case $(p,q,\beta)\in[1,\infty)\times[1,\infty)\times(0,1/p)$, since the proof of $(p,q,\beta)\in[1,\infty)\times\{\infty\}\times(0,1/p)$ is similar.
We first prove
$$B_{p,q}^{\mathcal{L}^{s},2\beta}(\mathbb{G}^{n}_{\alpha})\subseteq
B_{p,q}^{s\beta}(\mathbb{G}^{n}_{\alpha}).$$
Let us write
\begin{align*}
  \left\{\begin{aligned}
  \Phi_{1}(\sigma)&=(1+\sigma)^{-2s};\\
  \Phi_{2}(\sigma)&=(1+\sigma)^{-2s-n}.
  \end{aligned}\right.
\end{align*}
By Lemma \ref{wang-1}, we know from $d_{\alpha}(g,g')\leq t^{1/(2s)}$ that
$$K_{t}^{s}(g,g')\geq \Phi_{1}(1)\frac{1}{|B(g,t^{1/(2s)})|}.$$
 Therefore, we have
\begin{align*}
   &\int_{0}^{\infty}\frac{1}{t^{\beta q}}\Big(\int_{\mathbb{G}_{\alpha}^{n}}\int_{\mathbb{G}_{\alpha}^{n}}|u(g)-u(g')|^{p}K^{s}_{t}(g,g')dg'dg\Big)^{q/p}\frac{dt}{t}  \\
   &\geq \int_{0}^{\infty}\frac{1}{t^{\beta q}}\Big(\int_{\mathbb{G}_{\alpha}^{n}}\int_{B(g,t^{1/(2s)})}|u(g)-u(g')|^{p}K^{s}_{t}(g,g')dg'dg\Big)^{q/p}\frac{dt}{t}  \\
   & \gtrsim\int_{0}^{\infty}\Big(\int_{\mathbb{G}_{\alpha}^{n}}\int_{B(g,t^{1/(2s)})}\frac{|u(g)-u(g')|^{p}}{t^{\beta p}|B(g,t^{1/(2s)})|}dg'dg\Big)^{q/p}\frac{dt}{t}\\
   &\thicksim \int_{0}^{\infty}\Big(\int_{\mathbb{G}_{\alpha}^{n}}\int_{B(g,r)}\frac{|u(g)-u(g')|^{p}}{r^{2s\beta
   p}|B(g,r)|}dg'dg\Big)^{q/p}\frac{dr}{r},
\end{align*}
which deduces that
$$B_{p,q}^{\mathcal{L}^{s},2\beta}(\mathbb{G}^{n}_{\alpha})\subseteq
B_{p,q}^{s\beta}(\mathbb{G}^{n}_{\alpha}).$$

Conversely, it follows
from  Lemma \ref{wang-1} and (\ref{Re}) that
\begin{align*}
&\int_{0}^{\infty}\Big(\int_{\mathbb{G}_{\alpha}^{n}}\int_{\mathbb{G}_{\alpha}^{n}}|u(g)-u(g')|^{p}K^{s}_{t}(g,g')dg'dg\Big)^{q/p}\frac{dt}{t^{1+\beta q}}\\
&\lesssim\int_{0}^{\infty}\Big(\int_{\mathbb{G}_{\alpha}^{n}}\int_{d_{\alpha}(g,g')<t^{1/(2s)}}\frac{|u(g)-u(g')|^{p}}{|B(g,t^{1/(2s)})|}dg'dg\Big)^{q/p}\frac{dt}{t^{1+\beta q}}\\
&\quad+ \int_{0}^{\infty}\Big(\int_{\mathbb{G}_{\alpha}^{n}}\sum_{i=0}^{\infty}\int_{2^{i}t^{1/(2s)}
\leq d_{\alpha}(g,g')< 2^{i+1}t^{1/(2s)}}\frac{|u(g)-u(g')|^{p}\Phi_{1}(2^{i})}{|B(g,t^{1/(2s)}+2^{i}t^{1/2s})|}dg'dg\Big)^{q/p}\frac{dt}{t^{1+\beta q}}\\
&\lesssim \sum_{i=0}^{\infty}\Phi_{2}(2^{i})^{q/p}2^{i Qq/p}\int_{0}^{\infty}\Big(\int_{\mathbb{G}_{\alpha}^{n}}\int_{B(g, 2^{i}t^{1/(2s)})}\frac{|u(g)-u(g')|^{p}}{|B(g,2^{i}t^{1/(2s)})|}dg'dg\Big)^{q/p}\frac{dt}{t^{1+\beta q}}+(N_{p,q}^{s\beta}(u))^{q}\\
&\thicksim\sum_{i=0}^{\infty}\Phi_{2}(2^{i})^{q/p}2^{i Qq/p+2is\beta q}\int_{0}^{\infty}\Big(\int_{\mathbb{G}_{\alpha}^{n}}\int_{B(g,r)}\frac{|u(g)-u(g')|^{p}}{r^{2s\beta p}|B(g,r)|}dg'dg\Big)^{q/p}\frac{dr}{r}+(N_{p,q}^{s\beta}(u))^{q}\\
&\lesssim (N_{p,q}^{s\beta}(u))^{q},
\end{align*}
where we have used the fact that $\beta\in (0,1/p)$. This implies
that $B_{p,q}^{s\beta}(\mathbb{G}^{n}_{\alpha})\subseteq
B_{p,q}^{\mathcal{L}^{s},2\beta}(\mathbb{G}^{n}_{\alpha})$.
Therefore we obtain the desired result.
\end{proof}

According to  Theorems \ref{com-1} \&\ \ref{com-2}, we have the
following conclusion.
\begin{corollary}
Let $(p,q,\beta)\in[1,\infty)\times[1,\infty]\times(0,s/p).$  Then for any $s\in(0,1),$
$$B_{p,q}^{\mathcal{L},2\beta}(\mathbb{G}^{n}_{\alpha})=B_{p,q}^{\mathcal{L}^{s},2\beta/s}(\mathbb{G}^{n}_{\alpha})
=B_{p,q}^{\beta}(\mathbb{G}^{n}_{\alpha}).$$
\end{corollary}

\subsection{Some limiting behaviour of Besov semi-norms}\label{sec-3.3}
In a celebrated work by Bourgain, Brezis and Mironescu \cite{BBM},
they study the asymptotic behavior of the
Aronszajn-Gagliardo-Slobedetzky semi-norm (\ref{AGS}) when the order
of differentiability approaches one. To be precise, for
$p\in[1,\infty)$ and $u\in W^{1,p}(\mathbb{R}^{n}),$ it holds
$$\lim_{s\rightarrow1^{-}}(1-s)[u]_{s,p}^{p}=C(p,n)\|\nabla u\|^{p}_{L^{p}(\mathbb{R}^{n})},$$
where $C(p,n)$ is a constant that depends only on $p$ and $n$, and can be computed explicitly.

Later, Maz'ya and Shaposhnikova \cite{MS} and \cite{MS1} used a
different normalization to deal with the case where the parameter
$s$ tends to zero, and proved that
$$\lim_{s\rightarrow0^{+}}s[u]_{s,p}^{p}=c(p,n)\|u\|^{p}_{L^{p}(\mathbb{R}^{n})}.$$
These results have sparked strong research interest and have been applied and promoted
in many different directions, see
\cite{ABC2020,ABCRST2020,ABCRST2021,BGT2022,AGT2020,AGT2022,MPPP2007} and the references
therein.

In this section, we study the behaviour as $\beta\rightarrow0^{+}$
and $\beta\rightarrow1^{-}$ of some semigroup based Besov semi-norms
associated with Grushin-Laplace operators. Our results generalize a
previous one of Maz'ya-Shaposhnikova and Bourgain-Brezis-Mironescu
for the classical fractional Sobolev spaces on the Grushin setting.
Note that the constants present in our results are independent of
the spatial dimension.

Throughout Section \ref{sec-3.3}, we consider Besov spaces
$\mathfrak{B}^{\mathcal{L},\beta}_{p,q}(\mathbb{G}^{n}_{\alpha})$
with $p=q.$ The following lemma plays an important role in proving
the main theorem in this subsection.
\begin{lemma}\label{MS}
Let $q\in[1,\infty)$ and $u\in C_{c}^{\infty}(\mathbb{G}_{\alpha}^{n}).$ Then we have
$$\lim_{\beta\rightarrow0^{+}}\beta\int_{1}^{\infty}\frac{1}{t^{{\beta q}/{2}+1}}
\int_{\mathbb{G}^{n}_{\alpha}}e^{-t\mathcal{L}}(|u-u(g)|^{q})(g)dgdt=\dfrac{4}{q}\|u\|^{q}_{L^{q}(\mathbb{G}^{n}_{\alpha})}.$$
\end{lemma}
\begin{proof}
Assume first that $q>1.$ We claim that for any $u\in C_{c}^{\infty}(\mathbb{G}_{\alpha}^{n})\subset L^{p}(\mathbb{G}_{\alpha}^{n}),$
\begin{align}\label{lim}
\lim_{\beta\rightarrow0^{+}}\beta\int_{1}^{\infty}\int_{\mathbb{G}_{\alpha}^{n}}\int_{\mathbb{G}_{\alpha}^{n}}
K_{t}(g,g')(|u(g)|^{q-1}|u(g')|+|u(g)||u(g')|^{q-1})dg'dg\frac{dt}{t^{{\beta q}/{2}+1}}=0.
\end{align}
In fact, by exploiting H\"older's inequality and
the boundedness properties of
$$e^{-t\mathcal{L}}:
L^{p}(\mathbb{G}^{n}_{\alpha})\rightarrow L^{q}(\mathbb{G}^{n}_{\alpha})
\ \ \mbox{and}\ \ L^{p}(\mathbb{G}^{n}_{\alpha})\rightarrow L^{q'}(\mathbb{G}^{n}_{\alpha})$$
  with $p=1$
(see Proposition \ref{ulc}), we obtain
\begin{align*}
&\beta\int_{1}^{\infty}\int_{\mathbb{G}_{\alpha}^{n}}\int_{\mathbb{G}_{\alpha}^{n}}
K_{t}(g,g')\big(|u(g)|^{q-1}|u(g')|+|u(g)||u(g')|^{q-1}\big)dg'dg\frac{dt}{t^{{\beta q}/{2}+1}}\\
&=\beta\int_{1}^{\infty}\int_{\mathbb{G}_{\alpha}^{n}}
\bigg\{|u(g)|^{q-1}e^{-t\mathcal{L}}(|u|)(g)+|u(g)|e^{-t\mathcal{L}}(|u|^{q-1})(g)\bigg\}dg\frac{dt}{t^{{\beta q}/{2}+1}}\\
&\leq\beta\int_{1}^{\infty}
\bigg\{\|u\|_{L^{q}(\mathbb{G}_{\alpha}^{n})}^{q-1}\|e^{-t\mathcal{L}}(|u|)\|_{L^{q}(\mathbb{G}_{\alpha}^{n})}
+\|u\|_{L^{q}(\mathbb{G}_{\alpha}^{n})}\|e^{-t\mathcal{L}}(|u|^{q-1})\|_{L^{q'}(\mathbb{G}_{\alpha}^{n})}\bigg\}\frac{dt}{t^{{\beta q}/{2}+1}}\\
&\lesssim\beta\|u\|_{L^{q}(\mathbb{G}_{\alpha}^{n})}^{q-1}\|u\|_{L^{1}(\mathbb{G}_{\alpha}^{n})}
\int_{1}^{\infty}\frac{dt}{t^{{\beta q}/{2}+1+{Q(q-1)}/{(2q)}}}\\
&\quad
+\beta\|u\|_{L^{q}(\mathbb{G}_{\alpha}^{n})}\|u^{q-1}\|_{L^{1}(\mathbb{G}_{\alpha}^{n})}
\int_{1}^{\infty}\frac{dt}{t^{{\beta q}/{2}+1+{Q(q'-1)}/{(2q')}}}\\
&\lesssim\frac{2\beta q}{\beta q^{2}+Q(q-1)}\|u\|_{L^{q}(\mathbb{G}_{\alpha}^{n})}^{q-1}\|u\|_{L^{1}(\mathbb{G}_{\alpha}^{n})}
+\frac{2\beta q}{\beta q^{2}+Q}\|u\|_{L^{q}(\mathbb{G}_{\alpha}^{n})}\|u^{q-1}\|_{L^{1}(\mathbb{G}_{\alpha}^{n})}.
\end{align*}
Therefore, the last term tends to 0 as $\beta\rightarrow0^{+}.$ This shows (\ref{lim}).

In addition,
observe that there exists a positive constant $C_{q}$ such that for all $a,b\in\mathbb{R}$,
$$\big||a-b|^{q}-|a|^{q}-|b|^{q}\big|\leq C_{q}(|a|^{q-1}|b|+|a||b|^{q-1}).$$
We use the above inequality to deduce that
\begin{align*}
&\beta\bigg|\int_{1}^{\infty}\bigg\{\int_{\mathbb{G}_{\alpha}^{n}}\frac{e^{-t\mathcal{L}}(|u-u(g)|^{q})(g)}{t^{{\beta q}/{2}+1}}dg
-\int_{\mathbb{G}_{\alpha}^{n}}\int_{\mathbb{G}_{\alpha}^{n}}\frac{K_{t}(g,g')}{t^{{\beta q}/{2}+1}}
(|u(g)|^{q}+|u(g')|^{q})dg'dg\bigg\}dt\bigg|\\
&=\bigg|\beta\int_{1}^{\infty}\bigg\{\int_{\mathbb{G}_{\alpha}^{n}}\int_{\mathbb{G}_{\alpha}^{n}}K_{t}(g,g')
(|u(g)-u(g')|^{q}-|u(g)|^{q}-|u(g')|^{q})dg'dg\bigg\}\frac{dt}{t^{{\beta q}/{2}+1}}\bigg|\\
&\lesssim \beta\int_{1}^{\infty}\bigg\{\int_{\mathbb{G}_{\alpha}^{n}}\int_{\mathbb{G}_{\alpha}^{n}}
K_{t}(g,g')\big(|u(g)|^{q-1}|u(g')|+|u(g)||u(g')|^{q-1}\big)dg'dg\bigg\}\frac{dt}{t^{{\beta q}/{2}+1}}.
\end{align*}
Combining this estimate with (\ref{lim}), we conclude that
\begin{align*}
&\lim_{\beta\rightarrow0^{+}}\beta\int_{1}^{\infty}
\int_{\mathbb{G}^{n}_{\alpha}}e^{-t\mathcal{L}}(|u-u(g)|^{q})(g)dg\frac{dt}{t^{{\beta q}/{2}+1}}\\
&=\lim_{\beta\rightarrow0^{+}}\beta\int_{1}^{\infty}\bigg\{\int_{\mathbb{G}_{\alpha}^{n}}\int_{\mathbb{G}_{\alpha}^{n}}K_{t}(g,g')
(|u(g)|^{q}+|u(g')|^{q})dg'dg\bigg\}\frac{dt}{t^{{\beta q}/{2}+1}}.
\end{align*}
Hence, we obtain from (\ref{SC}) and Proposition
\ref{G-Pro} (iv) that
\begin{align*}
&\beta\int_{1}^{\infty}\int_{\mathbb{G}_{\alpha}^{n}}\int_{\mathbb{G}_{\alpha}^{n}}K_{t}(g,g')
(|u(g)|^{q}+|u(g')|^{q})dg'dg\frac{dt}{t^{{\beta q}/{2}+1}}\\
&=2\beta\|u\|^{q}_{L^{q}(\mathbb{G}_{\alpha}^{n})}\int_{1}^{\infty}t^{-{\beta q}/{2}-1}dt\\
&=\dfrac{4}{q}\|u\|^{q}_{L^{q}(\mathbb{G}^{n}_{\alpha})}.
\end{align*}
At this point, the desired conclusion follows by passing to the limit for the above equality in the case $q>1.$

We thus turn the attention to the case $q=1.$ For $\beta>0,$ by using (\ref{SC}) and Proposition \ref{G-Pro} (iv) again, we have
\begin{align*}
&\beta\int_{1}^{\infty}\frac{1}{t^{{\beta}/{2}+1}}
\int_{\mathbb{G}^{n}_{\alpha}}e^{-t\mathcal{L}}(|u-u(g)|)(g)dgdt\\
&\leq\beta\int_{1}^{\infty}\frac{1}{t^{{\beta}/{2}+1}}
\int_{\mathbb{G}^{n}_{\alpha}}\int_{\mathbb{G}^{n}_{\alpha}}K_{t}(g,g')(|u(g)|+|u(g')|)dg'dgdt\\
&=2\beta\|u\|_{L^{1}(\mathbb{G}^{n}_{\alpha})}\int_{1}^{\infty}t^{-{\beta}/{2}-1}dt=4\|u\|_{L^{1}(\mathbb{G}^{n}_{\alpha})},
\end{align*}
which implies
$$\limsup_{\beta\rightarrow0^{+}}\beta\int_{1}^{\infty}\frac{1}{t^{{\beta }/{2}+1}}
\int_{\mathbb{G}^{n}_{\alpha}}e^{-t\mathcal{L}}(|u-u(g)|)(g)dgdt\leq
4\|u\|_{L^{1}(\mathbb{G}^{n}_{\alpha})}.$$ Thus, to finish the proof
of Lemma \ref{MS}, we only need to prove
\begin{align}\label{liminf}
\liminf_{\beta\rightarrow0^{+}}\beta\int_{1}^{\infty}\frac{1}{t^{{\beta}/{2}+1}}
\int_{\mathbb{G}^{n}_{\alpha}}e^{-t\mathcal{L}}(|u-u(g)|)(g)dgdt\geq 4\|u\|_{L^{1}(\mathbb{G}^{n}_{\alpha})}.
\end{align}
We first notice that for a fixed
 $\varepsilon>0.$ Since $u\in L^{1}(\mathbb{G}^{n}_{\alpha}),$ we can find a compact set $F_{\varepsilon}
\subset\mathbb{G}^{n}_{\alpha}$ such that
\begin{align}\label{AC}
\int_{\mathbb{G}^{n}_{\alpha}\setminus F_{\varepsilon}}|u(\xi)|d\xi\leq\varepsilon.
\end{align}
Moreover, it follows from the symmetry of the kernel $K_{t}(\cdot,\cdot)$ and (\ref{SC}) that
\begin{align}\label{liminf1}
&\beta\int_{1}^{\infty}\frac{1}{t^{{\beta}/{2}+1}}
\int_{\mathbb{G}^{n}_{\alpha}}e^{-t\mathcal{L}}(|u-u(g)|)(g)dgdt\nonumber\\
&\geq \beta\int_{1}^{\infty}\frac{1}{t^{{\beta }/{2}+1}}
\int_{F_{\varepsilon}}\int_{\mathbb{G}^{n}_{\alpha}\setminus F_{\varepsilon}}K_{t}(g,g')(|u(g)|-|u(g')|)dg'dgdt\nonumber\\
&\quad+\beta\int_{1}^{\infty}\frac{1}{t^{{\beta }/{2}+1}}
\int_{\mathbb{G}^{n}_{\alpha}\setminus F_{\varepsilon}}\int_{F_{\varepsilon}}K_{t}(g,g')(|u(g')|-|u(g)|)dg'dgdt\nonumber\\
&=\beta\int_{1}^{\infty}\frac{1}{t^{{\beta }/{2}+1}}
\int_{F_{\varepsilon}}|u(g)|\bigg(1-\int_{F_{\varepsilon}}K_{t}(g,g')dg'\bigg)dgdt\nonumber\\
&\quad+\beta\int_{1}^{\infty}\frac{1}{t^{{\beta }/{2}+1}}
\int_{F_{\varepsilon}}|u(g')|\bigg(1-\int_{F_{\varepsilon}}K_{t}(g,g')dg\bigg)dg'dt\nonumber\\
&\quad-\beta\int_{1}^{\infty}\frac{1}{t^{{\beta }/{2}+1}}
\int_{F_{\varepsilon}}\int_{\mathbb{G}^{n}_{\alpha}\setminus F_{\varepsilon}}K_{t}(g,g')|u(g')|dg'dgdt\nonumber\\
&\quad-\beta\int_{1}^{\infty}\frac{1}{t^{{\beta }/{2}+1}}
\int_{\mathbb{G}^{n}_{\alpha}\setminus F_{\varepsilon}}\int_{F_{\varepsilon}}K_{t}(g,g')|u(g)|dg'dgdt\\
&= I_{1}-I_{2}-I_{3},\nonumber
\end{align}
where
\begin{align*}
  \left\{\begin{aligned}
  I_{1}&:=\beta\int_{1}^{\infty}\frac{1}{t^{{\beta }/{2}+1}}
  \int_{F_{\varepsilon}}|u(g)|dgdt+\beta\int_{1}^{\infty}\frac{1}{t^{{\beta}/{2}+1}} \int_{F_{\varepsilon}}|u(g')|dg'dt;\\
  I_{2}&:=\beta\int_{1}^{\infty}\frac{1}{t^{{\beta }/{2}+1}}
       \int_{F_{\varepsilon}}|u(g)|\int_{F_{\varepsilon}}K_{t}(g,g')dg'dgdt\\
     &\ \ \ +\beta\int_{1}^{\infty}\frac{1}{t^{{\beta }/{2}+1}}
       \int_{F_{\varepsilon}}|u(g')|\int_{F_{\varepsilon}}K_{t}(g,g')dgdg'dt;\\
  I_{3}&:=\beta\int_{1}^{\infty}\frac{1}{t^{{\beta }/{2}+1}}
          \int_{\mathbb{G}^{n}_{\alpha}\setminus F_{\varepsilon}}|u(g')|\int_{F_{\varepsilon}}K_{t}(g,g')dgdg'dt\\
         &\ \ \ +\beta\int_{1}^{\infty}\frac{1}{t^{{\beta }/{2}+1}}
         \int_{\mathbb{G}^{n}_{\alpha}\setminus F_{\varepsilon}}|u(g)|\int_{F_{\varepsilon}}K_{t}(g,g')dg'dgdt.
  \end{aligned}\right.
\end{align*}
By using (\ref{AC}), it is easy to verify
\begin{align*}
I_{1}&\geq2\beta\int_{1}^{\infty}\frac{1}{t^{{\beta }/{2}+1}}
\int_{F_{\varepsilon}}|u(g)|dgdt\\
&\geq 2\beta\int_{1}^{\infty}\frac{1}{t^{{\beta }/{2}+1}}
\Big(\int_{\mathbb{G}^{n}_{\alpha}}|u(g)|dg-\int_{\mathbb{G}^{n}_{\alpha}\setminus F_{\varepsilon}}|u(g)|dg\Big)dt\\
&\geq 4(\|u\|_{L^{1}(\mathbb{G}^{n}_{\alpha})}-\varepsilon).\\
\end{align*}
For $I_{2},$ we deduce from the Gaussian upper bound
(\ref{Kt}) and (\ref{B}) that
\begin{align*}
I_{2}&\lesssim \beta|F_{\varepsilon}|\int_{1}^{\infty}\frac{1}{t^{{\beta }/{2}+1+{Q}/{2}}}
\int_{F_{\varepsilon}}|u(g)|dgdt+
\beta |F_{\varepsilon}|\int_{1}^{\infty}\frac{1}{t^{{\beta }/{2}+1+{Q}/{2}}}
\int_{F_{\varepsilon}}|u(g')|dg'dt\\
&\lesssim 2\beta|F_{\varepsilon}|\|u\|_{L^{1}(\mathbb{G}^{n}_{\alpha})}\int_{1}^{\infty}
t^{-{\beta}/{2}-1-{Q}/{2}}dt\\
&=\frac{4\beta}{\beta+Q}|F_{\varepsilon}|\|u\|_{L^{1}(\mathbb{G}^{n}_{\alpha})}.
\end{align*}
On the other hand, the symmetry of $K_{t}(\cdot,\cdot)$
and (\ref{SC}) together with (\ref{AC}) ensure that
\begin{align*}
I_{3}&\leq\beta\int_{1}^{\infty}\frac{1}{t^{{\beta }/{2}+1}}
\int_{\mathbb{G}^{n}_{\alpha}\setminus F_{\varepsilon}}|u(g')|dg'dt
+\beta\int_{1}^{\infty}\frac{1}{t^{{\beta }/{2}+1}}
\int_{\mathbb{G}^{n}_{\alpha}\setminus F_{\varepsilon}}|u(g)|dgdt\\
&\leq 2\beta\varepsilon\int_{1}^{\infty}t^{-{\beta}/{2}-1}dt=4\varepsilon.
\end{align*}
Combining the estimates of the terms $I_{1}$, $I_{2}$ and $I_{3}$ with the inequality (\ref{liminf1})
derives that
\begin{align*}
\beta\int_{1}^{\infty}\frac{1}{t^{{\beta }/{2}+1}}
\int_{\mathbb{G}^{n}_{\alpha}}e^{-t\mathcal{L}}(|u-u(g)|)(g)dgdt
\geq4\|u\|_{L^{1}(\mathbb{G}^{n}_{\alpha})}-8\varepsilon-\frac{4\beta}{\beta+Q}|F_{\varepsilon}|\|u\|_{L^{1}(\mathbb{G}^{n}_{\alpha})},
\end{align*}
which implies (\ref{liminf}) by letting $\beta\rightarrow0^{+}$ and by using the arbitrariness of $\varepsilon.$
This concludes the proof of Lemma \ref{MS}.
\end{proof}

According to Lemma \ref{MS}, we can immediately prove the main
theorem of this section.
\begin{theorem}\label{MS1}
Let $p\in[1,\infty).$ If $u\in \mathfrak{B}^{\mathcal{L},\beta}_{p,p}(\mathbb{G}^{n}_{\alpha})$ for some $\beta\in(0,1),$ then
$$\lim_{\beta\rightarrow0^{+}}\beta N_{p,p}^{\mathcal{L},\beta}(u)^{p}=\frac{4}{p}\|u\|^{p}_{L^{p}(\mathbb{G}^{n}_{\alpha})}.$$
\end{theorem}
\begin{proof}
Let $p\in[1,\infty).$ Assume that $u\in
\mathfrak{B}^{\mathcal{L},\beta'}_{p,p}(\mathbb{G}^{n}_{\alpha}) $
with  $\beta'\in(0,1)$. For any $\beta\in(0,\beta']$ and
$k\in\mathbb{N},$ we have
\begin{align}\label{apro}
\bigg|\beta
N_{p,p}^{\mathcal{L},\beta}(u)^{p}-\dfrac{4}{p}\|u\|^{p}_{L^{p}(\mathbb{G}^{n}_{\alpha})}\bigg|
&\leq
\beta\bigg|N_{p,p}^{\mathcal{L},\beta}(u)^{p}-N_{p,p}^{\mathcal{L},\beta}(u_k)^{p}\bigg|
+\bigg|\beta N_{p,p}^{\mathcal{L},\beta}(u_k)^{p}-\frac{4}{p}\|u_k\|^{p}_{L^{p}(\mathbb{G}^{n}_{\alpha})}\bigg|\nonumber\\
&\quad+\frac{4}{p}\bigg|\|u_k\|^{p}_{L^{p}(\mathbb{G}^{n}_{\alpha})}-\|u\|^{p}_{L^{p}(\mathbb{G}^{n}_{\alpha})}\bigg|
\end{align} with $ u_k \in C_{c}^{\infty}(\mathbb{G}_{\alpha}^{n})$.
Now we estimate the above three terms, respectively. We first note
that for any $a,b>0$ such that $a\neq b,$ the elementary inequality
$$|a^{p}-b^{p}|\leq \max\{a,b\}^{p-1}|a-b| $$ holds true.
Therefore, it follows from the proof of Proposition \ref{emb} that
\begin{align*}
&\beta\big|N_{p,p}^{\mathcal{L},\beta}(u)^{p}-N_{p,p}^{\mathcal{L},\beta}(u_k)^{p}\big|\\
&\leq\beta
\max\Big\{N_{p,p}^{\mathcal{L},\beta}(u),\ N_{p,p}^{\mathcal{L},\beta}(u_k)\Big\}^{p-1}
\big|N_{p,p}^{\mathcal{L},\beta}(u)-N_{p,p}^{\mathcal{L},\beta}(u_k)\big|\\
&\leq
\max\Big\{\beta^{1/p}N_{p,p}^{\mathcal{L},\beta}(u),\ \beta^{1/p}N_{p,p}^{\mathcal{L},\beta}(u_k)\Big\}^{p-1}
\beta^{1/p}N_{p,p}^{\mathcal{L},\beta}(u-u_k)\\
&\leq \max\Big\{\beta'N_{p,p}^{\mathcal{L},\beta'}(u)^{p}+\frac{2^{p+1}}{p}\|u\|^{p}_{L^{p}(\mathbb{G}^{n}_{\alpha})},\
\beta'N_{p,p}^{\mathcal{L},\beta'}(u_k)^{p}+\frac{2^{p+1}}{p}\|u_k\|^{p}_{L^{p}(\mathbb{G}^{n}_{\alpha})}\Big\}^{1-{1}/{p}}\\
&\quad\times
\big(\beta'N_{p,p}^{\mathcal{L},\beta'}(u-u_k)^{p}+\frac{2^{p+1}}{p}\|u-u_k\|^{p}_{L^{p}(\mathbb{G}^{n}_{\alpha})}\big)^{1/p}.
\end{align*}
On the other hand, we know from Definition \ref{Besov2} that there exist a sequence
$\{u_k\}_{k\in\mathbb{N}}\subset C_{c}^{\infty}(\mathbb{G}_{\alpha}^{n})$
such that
$$\|u_k-u\|_{L^{p}(\mathbb{G}^{n}_{\alpha})},\
N_{p,p}^{\mathcal{L},\beta'}(u_k-u)\rightarrow0\ \mbox{as}\
k\rightarrow\infty.$$ Specially, for any given $\varepsilon>0$ there
exists $N(\varepsilon)\in\mathbb{N}$ such that
\begin{align}\label{apro1}
\frac{4}{p}\big|\|u_k\|^{p}_{L^{p}(\mathbb{G}^{n}_{\alpha})}-\|u\|^{p}_{L^{p}(\mathbb{G}^{n}_{\alpha})}
\big|\leq\frac{\varepsilon}{3},\ \mbox{whenever}\ k\geq
N(\varepsilon).
\end{align}
Moreover, there exists $N'(\varepsilon)\in\mathbb{N}$ such that for any $\beta\in(0,\beta'],$
\begin{align}\label{apro2}
\beta\big|N_{p,p}^{\mathcal{L},\beta}(u)^{p}-N_{p,p}^{\mathcal{L},\beta}(u_k)^{p}\big|\leq\frac{\varepsilon}{3},\
\mbox{whenever}\ k\geq N'(\varepsilon).
\end{align}
Let $N''(\varepsilon)=\max\{N(\varepsilon),N'(\varepsilon)\}$ and
fix $k'\geq N''(\varepsilon).$ For any $\beta\in(0,\beta'],$ we
deduce from (\ref{apro}), (\ref{apro1}) and (\ref{apro2}) that
$$\big|\beta N_{p,p}^{\mathcal{L},\beta}(u)^{p}-\frac{4}{p}\|u\|^{p}_{L^{p}(\mathbb{G}^{n}_{\alpha})}\big|
\leq \frac{2\varepsilon}{3}+\big|\beta
N_{p,p}^{\mathcal{L},\beta}(u_{k'})^{p}-\frac{4}{p}\|u_{k'}\|^{p}_{L^{p}(\mathbb{G}^{n}_{\alpha})}\big|.$$
In what follows, we need to prove that
\begin{align}\label{ms}
\lim_{\beta\rightarrow0^{+}}\beta
N_{p,p}^{\mathcal{L},\beta}(u_{k'})^{p}=\frac{4}{p}\|u_{k'}\|^{p}_{L^{p}(\mathbb{G}^{n}_{\alpha})}.
\end{align}
In fact, from the definition of $N_{p,p}^{\mathcal{L},\beta}(\cdot),$ we can write
\begin{align*}
\beta
N_{p,p}^{\mathcal{L},\beta}(u_{k'})^{p}&=\beta\int_{0}^{1}\frac{1}{t^{{\beta
p}/{2}+1}}
\int_{\mathbb{G}^{n}_{\alpha}}e^{-t\mathcal{L}}(|u_{k'}-u_{k'}(g)|^{p})(g)dgdt\\
&\quad+\beta\int_{1}^{\infty}\frac{1}{t^{{\beta p}/{2}+1}}
\int_{\mathbb{G}^{n}_{\alpha}}e^{-t\mathcal{L}}(|u_{k'}-u_{k'}(g)|^{p})(g)dgdt.
\end{align*}
Since $u_{k'}\in
\mathfrak{B}_{p,p}^{\mathcal{L},\beta'}(\mathbb{G}^{n}_{\alpha})$
for some $\beta'\in(0,1),$ then for any $\beta\in(0,\beta'],$ we
have
\begin{align*}
&\beta\int_{0}^{1}\frac{1}{t^{{\beta p}/{2}+1}}
\int_{\mathbb{G}^{n}_{\alpha}}e^{-t\mathcal{L}}(|u_{k'}-u_{k'}(g)|^{p})(g)dgdt\\
&\leq \beta\int_{0}^{1}\frac{1}{t^{{\beta' p}/{2}+1}}
\int_{\mathbb{G}^{n}_{\alpha}}e^{-t\mathcal{L}}(|u_{k'}-u_{k'}(g)|^{p})(g)dgdt\leq
\beta N_{p,p}^{\mathcal{L},\beta'}(u_{k'})^{p}<\infty.
\end{align*}
This implies that
$$\lim_{\beta\rightarrow0^{+}}\beta\int_{0}^{1}\frac{1}{t^{{\beta p}/{2}+1}}
\int_{\mathbb{G}^{n}_{\alpha}}e^{-t\mathcal{L}}(|u_{k'}-u_{k'}(g)|^{p})(g)dgdt=0.$$
Combining this with Lemma \ref{MS}, we conclude that (\ref{ms}) is valid.

 Therefore, there exists $\beta''<\beta'$ such that for any
$\beta\in(0,\beta''],$
$$\big|\beta N_{p,p}^{\mathcal{L},\beta}(u_{k'})^{p}-\frac{4}{p}\|u_{k'}\|^{p}_{L^{p}(\mathbb{G}^{n}_{\alpha})}\big|\leq \frac{\varepsilon}{3},$$
which derives
$$\big|\beta N_{p,p}^{\mathcal{L},\beta}(u)^{p}-\frac{4}{p}\|u\|^{p}_{L^{p}(\mathbb{G}^{n}_{\alpha})}\big|
\leq \frac{2\varepsilon}{3}+\frac{\varepsilon}{3}=\varepsilon.$$
This completes the proof of the theorem.
\end{proof}

A complementary and natural question is what happens when we
consider  the limit as $\beta\rightarrow1^{-}$. The answer is
contained in the following proposition.
\begin{proposition}\label{BBM}
Let $p\in[1,\infty)$ and $u\in L^{p}(\mathbb{G}^{n}_{\alpha}).$ Then we have
\begin{align*}
&\frac{2}{p}\liminf_{t\rightarrow0^{+}}t^{-{p}/{2}}\int_{\mathbb{G}^{n}_{\alpha}}
e^{-t\mathcal{L}}(|u-u(g)|^{p})(g)dg\\
&\leq \liminf_{\beta\rightarrow1^{-}}(1-\beta)N_{p,p}^{\mathcal{L},\beta}(u)^{p}\\
&\leq \limsup_{\beta\rightarrow1^{-}}(1-\beta)N_{p,p}^{\mathcal{L},\beta}(u)^{p}\\
&\leq
\frac{2}{p}\limsup_{t\rightarrow0^{+}}t^{-{p}/{2}}\int_{\mathbb{G}^{n}_{\alpha}}
e^{-t\mathcal{L}}(|u-u(g)|^{p})(g)dg.
\end{align*}
In particular, if $\lim_{\beta\rightarrow1^{-}}(1-\beta)N_{p,p}^{\mathcal{L},\beta}(u)^{p}$
exists, then
\begin{align*}
&\frac{2}{p}\liminf_{t\rightarrow0^{+}}t^{-{p}/{2}}\int_{\mathbb{G}^{n}_{\alpha}}
e^{-t\mathcal{L}}(|u-u(g)|^{p})(g)dg\\
&\leq \lim_{\beta\rightarrow1^{-}}(1-\beta)N_{p,p}^{\mathcal{L},\beta}(u)^{p}\\
&\leq\frac{2}{p}\limsup_{t\rightarrow0^{+}}t^{-{p}/{2}}\int_{\mathbb{G}^{n}_{\alpha}}
e^{-t\mathcal{L}}(|u-u(g)|^{p})(g)dg.
\end{align*}
\end{proposition}
\begin{proof}
We first prove that
\begin{align}\label{belim}
\limsup_{\beta\rightarrow1^{-}}(1-\beta)N_{p,p}^{\mathcal{L},\beta}(u)^{p}\leq
\frac{2}{p}\limsup_{t\rightarrow0^{+}}t^{-{p}/{2}}\int_{\mathbb{G}^{n}_{\alpha}}
e^{-t\mathcal{L}}(|u-u(g)|^{p})(g)dg.
\end{align}
If $$\limsup_{t\rightarrow0^{+}}t^{-{p}/{2}}\int_{\mathbb{G}^{n}_{\alpha}}
e^{-t\mathcal{L}}(|u-u(g)|^{p})(g)dg=\infty,$$ then the desired conclusion (\ref{belim}) is
trivial. Thus we assume that $$\limsup_{t\rightarrow0^{+}}t^{-{p}/{2}}\int_{\mathbb{G}^{n}_{\alpha}}
e^{-t\mathcal{L}}(|u-u(g)|^{p})(g)dg<\infty,$$
then  there exists a positive constant $\varepsilon_{0}$ such that
$$\sup_{\tau\in(0,\varepsilon_{0})}\tau^{-{p}/{2}}\int_{\mathbb{G}^{n}_{\alpha}}
e^{-\tau\mathcal{L}}(|u-u(g)|^{p})(g)dg<\infty.$$
This implies
\begin{align}\label{belim1}
&\int_{0}^{\varepsilon}\frac{1}{t^{{\beta p}/{2}+1}}
\int_{\mathbb{G}^{n}_{\alpha}}e^{-t\mathcal{L}}(|u-u(g)|^{p})(g)dgdt\nonumber\\
&\leq\bigg(\sup_{\tau\in(0,\varepsilon)}\tau^{-{p}/{2}}\int_{\mathbb{G}^{n}_{\alpha}}
e^{-\tau\mathcal{L}}(|u-u(g)|^{p})(g)dg\bigg)\int_{0}^{\varepsilon}\frac{dt}{t^{{p}(\beta-1)/{2}+1}}\nonumber\\
&=\bigg(\sup_{\tau\in(0,\varepsilon)}\tau^{-{p}/{2}}\int_{\mathbb{G}^{n}_{\alpha}}
e^{-\tau\mathcal{L}}(|u-u(g)|^{p})(g)dg\bigg)\frac{2\varepsilon^{{p}(1-\beta)/{2}}}{p(1-\beta)}
\end{align}
for any $\varepsilon\in(0,\varepsilon_{0}).$ On the other hand, by
using (\ref{SC}) and Proposition \ref{G-Pro} (iv) together with the
elementary inequality $|a-b|^{p}\leq 2^{p-1}(a^{p}+b^{p})$, we have
\begin{align*}
\int_{\varepsilon}^{\infty}\frac{1}{t^{{\beta p}/{2}+1}}
\int_{\mathbb{G}^{n}_{\alpha}}e^{-t\mathcal{L}}(|u-u(g)|^{p})(g)dgdt
&\leq
2^{p}\|u\|^{p}_{L^{p}(\mathbb{G}^{n}_{\alpha})}\int_{\varepsilon}^{\infty}\frac{dt}{t^{{\beta
p}/{2}+1}}\\
& =\frac{2^{p+1}}{\beta p}\varepsilon^{-{\beta
p}/{2}}\|u\|^{p}_{L^{p}(\mathbb{G}^{n}_{\alpha})}.
\end{align*} Combining the
previous estimates with (\ref{belim1}) derives
\begin{align}\label{belim2}
N_{\beta,p}^{\mathcal{L}}(u)^{p}&=\int_{0}^{\varepsilon}\frac{1}{t^{{\beta p}/{2}+1}}
\int_{\mathbb{G}^{n}_{\alpha}}e^{-t\mathcal{L}}(|u-u(g)|^{p})(g)dgdt\\
&\ \ \ +\int_{\varepsilon}^{\infty}\frac{1}{t^{{\beta p}/{2}+1}}
\int_{\mathbb{G}^{n}_{\alpha}}e^{-t\mathcal{L}}(|u-u(g)|^{p})(g)dgdt\nonumber\\
&\leq\bigg(\sup_{\tau\in(0,\varepsilon)}\tau^{-{p}/{2}}\int_{\mathbb{G}^{n}_{\alpha}}
e^{-\tau\mathcal{L}}(|u-u(g)|^{p})(g)dg\bigg)\frac{2\varepsilon^{{p}(1-\beta)/{2}}}{p(1-\beta)}\nonumber\\
&\quad+
\frac{2^{p+1}}{\beta p}\varepsilon^{-{\beta p}/{2}}\|u\|^{p}_{L^{p}(\mathbb{G}^{n}_{\alpha})}.\nonumber
\end{align}
Multiplying by $(1-\beta)$ in (\ref{belim2}) and taking the $\limsup_{\beta\rightarrow1^{-}},$ we find for any $\varepsilon\in(0,\varepsilon_{0}),$
$$\limsup_{\beta\rightarrow1^{-}}(1-\beta)N_{p,p}^{\mathcal{L},\beta}(u)^{p}\leq\frac{2}{p}
\sup_{\tau\in(0,\varepsilon)}\tau^{-{p}/{2}}\int_{\mathbb{G}^{n}_{\alpha}}
e^{-\tau\mathcal{L}}(|u-u(g)|^{p})(g)dg,$$
which shows the validity of (\ref{belim}) by letting $\varepsilon\rightarrow0^{+}.$

Finally,  we prove that
\begin{align}\label{belim3}
\frac{2}{p}\liminf_{t\rightarrow0^{+}}t^{-{p}/{2}}\int_{\mathbb{G}^{n}_{\alpha}}
e^{-t\mathcal{L}}(|u-u(g)|^{p})(g)dg\leq \liminf_{\beta\rightarrow1^{-}}(1-\beta)N_{p,p}^{\mathcal{L},\beta}(u)^{p}.
\end{align}
For any $\beta\in(0,1)$ and any $\varepsilon>0,$ we have
\begin{align*}
(1-\beta)N_{p,p}^{\mathcal{L},\beta}(u)^{p}&\geq (1-\beta)\int_{0}^{\varepsilon}\frac{1}{t^{{\beta p}/{2}+1}}
\int_{\mathbb{G}^{n}_{\alpha}}
e^{-t\mathcal{L}}(|u-u(g)|^{p})(g)dgdt\\
&\geq(1-\beta)\bigg(\inf_{\tau\in(0,\varepsilon)}\tau^{-{p}/{2}}\int_{\mathbb{G}^{n}_{\alpha}}
e^{-\tau\mathcal{L}}(|u-u(g)|^{p})(g)dg\bigg)\int_{0}^{\varepsilon}t^{{p}(1-\beta)/{2}-1}dt\\
&=\frac{2}{p}\varepsilon^{{p}(1-\beta)/{2}}\inf_{\tau\in(0,\varepsilon)}\tau^{-{p}/{2}}\int_{\mathbb{G}^{n}_{\alpha}}
e^{-\tau\mathcal{L}}(|u-u(g)|^{p})(g)dg.
\end{align*}
Taking the $\liminf_{\beta\rightarrow1^{-}}$ in the latter
inequality  yields
$$\liminf_{\beta\rightarrow1^{-}}(1-\beta)N_{p,p}^{\mathcal{L},\beta}(u)^{p}\geq \frac{2}{p}
\inf_{\tau\in(0,\varepsilon)}\tau^{-{p}/{2}}\int_{\mathbb{G}^{n}_{\alpha}}
e^{-\tau\mathcal{L}}(|u-u(g)|^{p})(g)dg.$$ Letting
$\varepsilon\rightarrow0^{+},$ we then get (\ref{belim3}) and hence
this  completes the proof of this proposition.
\end{proof}

\section{Fractional perimeters on Grushin spaces}\label{sec-4}
The bounded variation and the perimeter measure on the ambient space
with different measures have been widely studied. See
\cite{A2002,ABC2020,ABCRST2020,ABCRST2021,F2003,GN1996,GT2020,M2003,Monti}
and the references therein. Inspired by \cite{ABC2020},
\cite{ABCRST2020}, \cite{GT2020} and \cite{M2003}, we introduce
fractional bounded variation functions and fractional perimeters on
Grushin spaces.
\begin{definition}\label{FBV}
Let $s\in(0,1/2).$ We say that $u\in L^{1}(\mathbb{G}^{n}_{\alpha})$
has bounded $s$-variation if there is a sequence of functions
$\{u_{k}\}_{k\in\mathbb{N}}\subset
\mathscr{S}(\mathbb{G}^{n}_{\alpha})$ such that $u_{k}\rightarrow u$
in $L^{1}(\mathbb{G}^{n}_{\alpha})$ and
$$\liminf_{k\rightarrow\infty}\|\mathcal{L}^{s}u_{k}\|_{L^{1}(\mathbb{G}^{n}_{\alpha})}<\infty.$$
We denote by
$\mathcal{BV}_{s}^{\mathcal{L}}(\mathbb{G}_{\alpha}^{n})$ the
bounded $s$-variation space consisting of  all functions with
bounded $s$-variation on $\mathbb{G}^{n}_{\alpha}$.
\end{definition}
\begin{remark}
For any $u\in\mathscr{S}(\mathbb{G}_{\alpha}^{n}),$ we deduce from Remark \ref{rem1} (i) that $\mathcal{L}^{s}u\in L^{1}(\mathbb{G}_{\alpha}^{n}).$
This ensures the validity of the above definition.
\end{remark}

For $u\in \mathcal{BV}_{s}^{\mathcal{L}}(\mathbb{G}_{\alpha}^{n}),$
the $s$-total variation of $u$ is defined by
$$|V_{s}^{\mathcal{L}}u|(\mathbb{G}^{n}_{\alpha})=\inf\Big\{\liminf_{k\rightarrow\infty}\|\mathcal{L}^{s}u_{k}\|_{L^{1}
(\mathbb{G}^{n}_{\alpha})}:\{u_{k}\}_{k\in\mathbb{N}}\subset \mathscr{S}(\mathbb{G}^{n}_{\alpha}),
u_{k}\rightarrow u\ \mbox{in}\ L^{1}(\mathbb{G}^{n}_{\alpha})\Big\}.$$

\begin{definition}\label{perimeter}
Let $E\subset \mathbb{G}^{n}_{\alpha}$ be a measurable set with finite measure. We will say that $E$ has a
finite $s$-perimeter if $\mathbf{1}_{E}\in\mathcal{BV}_{s}^{\mathcal{L}}(\mathbb{G}_{\alpha}^{n}).$
In that case, we will denote
\begin{align*}
&P_{s}^{\mathcal{L}}(E):=|V_{s}^{\mathcal{L}}\mathbf{1}_{E}|(\mathbb{G}^{n}_{\alpha})\\
&=\inf\Big\{\liminf_{k\rightarrow\infty}\|\mathcal{L}^{s}u_{k}\|_{L^{1}
(\mathbb{G}^{n}_{\alpha})}:\{u_{k}\}_{k\in\mathbb{N}}\subset \mathscr{S}(\mathbb{G}^{n}_{\alpha}),
u_{k}\rightarrow \mathbf{1}_{E}\ \mbox{in}\ L^{1}(\mathbb{G}^{n}_{\alpha})\Big\}.
\end{align*}
\end{definition}

As we have said in the introduction, when $\alpha=0,\mathcal{L}=-\Delta,$
 up to a renormalizing factor the semi-norm $N_{s,p}^{-\Delta}$  coincides
with the classical Aronszajn-Gagliardo-Slobedetzky semi-norm
(\ref{AGS}): in such framework it is nowadays a common practice to
call {\it nonlocal perimeter}, defined in (\ref{Ps}),  the
Aronszajn-Gagliardo-Slobedetzky semi-norm of the characteristic
function (see \ref{AGSP}), and we refer the reader to the
influential work \cite{CRS2010} where the structure of the critical
points of nonlocal perimeters was first analyzed.

Inspired by the above description, we introduce another fractional perimeters based on norm (\ref{Besov}) in this section.
\begin{definition}\label{perimeter1}
Let $s\in (0,1/2).$ We say that a measurable set $E\subset\mathbb{G}_{\alpha}^{n}$
has finite $s$-perimeter$^{*}$ if $\mathbf{1}_{E}\in \mathfrak{B}_{1,1}^{\mathcal{L},2s}(\mathbb{G}_{\alpha}^{n})$
and we define the $s$-perimeter$^{*}$ associated with $\mathcal{L}$ as
$$P_{s}^{\mathcal{L},*}(E):=N_{1,1}^{\mathcal{L},2s}(\mathbf{1}_{E}).$$
\end{definition}
By comparing the two fractional perimeters in Definitions \ref{perimeter} \& \ref{perimeter1},
we immediately have the following results.
\begin{proposition}\label{com}
Let $s\in(0,1/2).$ If $\mathbf{1}_{E}\in \mathfrak{B}_{1,1}^{\mathcal{L},2s}(\mathbb{G}_{\alpha}^{n})$,
then  $$P_{s}^{\mathcal{L}}(E)\leq \frac{s}{\Gamma(1-s)}P_{s}^{\mathcal{L},*}(E).$$
\end{proposition}
\begin{proof}
We first prove that for any measurable $E\subset \mathbb{G}^{n}_{\alpha},$
\begin{align}\label{B-W-eq}
\mathbf{1}_{E}\in \mathcal{W}^{2s,1}(\mathbb{G}^{n}_{\alpha})\Longleftrightarrow\mathbf{1}_{E}\in B_{1,1}^{\mathcal{L},2s}(\mathbb{G}^{n}_{\alpha}).
\end{align}
In fact, the above statements can be directly obtained from the following equality:
\begin{align}\label{equ-sta}
\|\mathcal{L}^{s}\mathbf{1}_{E}\|_{L^{1}(\mathbb{G}^{n}_{\alpha})}
&=\frac{s}{\Gamma(1-s)}\int_{\mathbb{G}^{n}_{\alpha}}\bigg|\int_{0}^{\infty}\frac{1}{t^{1+s}}
\big(e^{-t\mathcal{L}}\mathbf{1}_{E}(g)-\mathbf{1}_{E}(g)\big)dt\bigg|dg\nonumber\\
&=\frac{s}{\Gamma(1-s)}\int_{0}^{\infty}\frac{1}{t^{1+s}}\|e^{-t\mathcal{L}}
\mathbf{1}_{E}-\mathbf{1}_{E}\|_{L^{1}(\mathbb{G}^{n}_{\alpha})}dt\nonumber\\
&=\frac{s}{\Gamma(1-s)}\int_{0}^{\infty}\frac{1}{t^{1+s}}\int_{\mathbb{G}^{n}_{\alpha}}
e^{-t\mathcal{L}}(|\mathbf{1}_{E}-\mathbf{1}_{E}(g)|)(g)dgdt\nonumber\\
&=\frac{s}{\Gamma(1-s)}N_{1,1}^{\mathcal{L},2s}(\mathbf{1}_{E}),
\end{align}
where we have used (\ref{-G}) and the equality:
\begin{align}\label{fact}
\|e^{-t\mathcal{L}}\mathbf{1}_{E}-\mathbf{1}_{E}\|_{L^{1}(\mathbb{G}^{n}_{\alpha})}
&=\int_{E}\big(1-e^{-t\mathcal{L}}\mathbf{1}_{E}(g)\big)dg+\int_{E^{c}}
e^{-t\mathcal{L}}\mathbf{1}_{E}(g)dg\nonumber\\
&=\int_{\mathbb{G}^{n}_{\alpha}}\mathbf{1}_{E}(g)e^{-t\mathcal{L}}(1-\mathbf{1}_{E})(g)dg
+\int_{\mathbb{G}^{n}_{\alpha}}\mathbf{1}_{E^{c}}(g)e^{-t\mathcal{L}}\mathbf{1}_{E}(g)dg\nonumber\\
&=\int_{\mathbb{G}^{n}_{\alpha}}\int_{\mathbb{G}^{n}_{\alpha}}K_{t}(g,g')\bigg(
\mathbf{1}_{E^{c}}(g)\mathbf{1}_{E}(g')+\mathbf{1}_{E}(g)
\mathbf{1}_{E^{c}}(g')\bigg)dg'dg\nonumber\\
&=\int_{\mathbb{G}^{n}_{\alpha}}\int_{\mathbb{G}^{n}_{\alpha}}K_{t}(g,g')|\mathbf{1}_{E}(g')-
\mathbf{1}_{E}(g)|dg'dg.
\end{align}
Note that $\mathfrak{B}_{1,1}^{\mathcal{L},2s}(\mathbb{G}_{\alpha}^{n})\subseteq B_{1,1}^{\mathcal{L},2s}(\mathbb{G}_{\alpha}^{n})$
(see Proposition \ref{desi}).
Thus, if $\mathbf{1}_{E}\in \mathfrak{B}_{1,1}^{\mathcal{L},2s}(\mathbb{G}_{\alpha}^{n}),$
we know from (\ref{equ-sta}) and Definition \ref{perimeter1} that
\begin{align}\label{eq}
\|\mathcal{L}^{s}\mathbf{1}_{E}\|_{L^{1}(\mathbb{G}^{n}_{\alpha})}
=\frac{s}{\Gamma(1-s)}P_{s}^{\mathcal{L},*}(E).
\end{align}
On the other hand,
if $\mathbf{1}_{E}\in \mathfrak{B}_{1,1}^{\mathcal{L},2s}(\mathbb{G}_{\alpha}^{n}),$
then we deduce from (\ref{B-W-eq}) and Proposition \ref{W-de} that there exists $\{u_{k}\}_{k\in\mathbb{N}}\subset\mathscr{S}(\mathbb{G}_{\alpha}^{n})$
such that
$$\|u_{k}-\mathbf{1}_{E}\|_{\mathcal{W}^{2s,1}(\mathbb{G}_{\alpha}^{n})}\rightarrow0,\ \mbox{as}\ k\rightarrow\infty,$$
which, together with (\ref{eq}), yields,
$$\lim_{k\rightarrow\infty}\|\mathcal{L}^{s}u_{k}\|_{L^{1}(\mathbb{G}^{n}_{\alpha})}
=\|\mathcal{L}^{s}\mathbf{1}_{E}\|_{L^{1}(\mathbb{G}^{n}_{\alpha})}
=\frac{s}{\Gamma(1-s)}
P_{s}^{\mathcal{L},*}(E).$$
Combining this with Definition \ref{perimeter} we reach the desired conclusion.
\end{proof}

Another important tool for our purposes is the validity of a coarea
formula which yields a further link between the semi-norm
$N_{1,1}^{\mathcal{L},2s}(\cdot)$ and the fractional perimeter.
\begin{lemma}\label{co-area}
Let $s\in(0,1/2).$ For any $u\in
\mathfrak{B}_{1,1}^{\mathcal{L},2s}(\mathbb{G}^{n}_{\alpha}),$
\begin{align}\label{coarea}
N_{1,1}^{\mathcal{L},2s}(u)\geq \frac{\Gamma(1-s)}{s}\int_{\mathbb{R}}P_{s}^{\mathcal{L}}(\{u>\sigma\})d\sigma.
\end{align}
For the fractional perimeter $P_{s}^{\mathcal{L},*}(\cdot)$ in
Definition \ref{perimeter1}, we also have
\begin{align}\label{coarea1}
N_{1,1}^{\mathcal{L},2s}(u)=\int_{\mathbb{R}}P_{s}^{\mathcal{L},*}(\{u>\sigma\})d\sigma.
\end{align}
\end{lemma}
\begin{proof}
For any measurable function $u$ and for any $\sigma\in\mathbb{R},$
we denote
$$E_{\sigma}=\{g\in\mathbb{G}^{n}_{\alpha}:u(g)>\sigma\}.$$
It follows from (\ref{fact}) that
\begin{align*}
\|e^{-t\mathcal{L}}\mathbf{1}_{E_{\sigma}}-\mathbf{1}_{E_{\sigma}}\|_{L^{1}(\mathbb{G}^{n}_{\alpha})}
=\int_{\mathbb{G}^{n}_{\alpha}}\int_{\mathbb{G}^{n}_{\alpha}}K_{t}(g,g')|\mathbf{1}_{E_{\sigma}}(g')-
\mathbf{1}_{E_{\sigma}}(g)|dg'dg.
\end{align*}
Note that $$\int_{\mathbb{R}}|\mathbf{1}_{E_{\sigma}}(g')-
\mathbf{1}_{E_{\sigma}}(g)|d\sigma=|u(g')-u(g)|.$$ Via these facts
and Tonelli's theorem, we obtain
\begin{align*}
&\int_{\mathbb{R}}\int_{0}^{\infty}\frac{1}{t^{1+s}}\|e^{-t\mathcal{L}}\mathbf{1}_{E_{\sigma}}-
\mathbf{1}_{E_{\sigma}}\|_{L^{1}(\mathbb{G}^{n}_{\alpha})}dtd\sigma\\
&=\int_{0}^{\infty}\frac{1}{t^{1+s}}\int_{\mathbb{G}^{n}_{\alpha}}\int_{\mathbb{G}^{n}_{\alpha}}K_{t}(g,g')|u(g')-u(g)|dg'dgdt\\
&=\int_{0}^{\infty}\frac{1}{t^{1+s}}\int_{\mathbb{G}^{n}_{\alpha}}e^{-t\mathcal{L}}(|u-u(g)|)(g)dgdt\\
&=N_{1,1}^{\mathcal{L},2s}(u).
\end{align*}
If $u\in \mathfrak{B}_{1,1}^{\mathcal{L},2s}(\mathbb{G}^{n}_{\alpha}),$ then
for almost every $\sigma\in\mathbb{R},$
$$\int_{0}^{\infty}\frac{1}{t^{1+s}}\|e^{-t\mathcal{L}}\mathbf{1}_{E_{\sigma}}-
\mathbf{1}_{E_{\sigma}}\|_{L^{1}(\mathbb{G}^{n}_{\alpha})}dt<\infty.$$
Thus, for such $\sigma,$
we deduce from the proof of Lemma
\ref{co}, (\ref{B-W-eq}) and (\ref{equ-sta}) that
$$\mathbf{1}_{E_{\sigma}}\in
\mathcal{W}^{2s,1}(\mathbb{G}^{n}_{\alpha}),\ N_{1,1}^{\mathcal{L},2s}(u)=\frac{\Gamma(1-s)}{s}\int_{\mathbb{R}}\|\mathcal{L}^{s}
\mathbf{1}_{E_{\sigma}}\|_{L^{1}(\mathbb{G}^{n}_{\alpha})}d\sigma,$$
where we note that the fact $\mathfrak{B}_{1,1}^{\mathcal{L},2s}(\mathbb{G}^{n}_{\alpha})\subseteq B_{1,1}^{\mathcal{L},2s}(\mathbb{G}^{n}_{\alpha}).$

To reach the desired conclusion we only need to prove that
\begin{align}\label{Per1}
\|\mathcal{L}^{s}\mathbf{1}_{E_{\sigma}}\|_{L^{1}(\mathbb{G}^{n}_{\alpha})}\geq P_{s}^{\mathcal{L}}(E_{\sigma}).
\end{align}
Indeed, we know from Proposition \ref{W-de} that there exists $\{u_{k}\}_{k\in\mathbb{N}}\subset\mathscr{S}(\mathbb{G}_{\alpha}^{n})$
such that
$$\|u_{k}-\mathbf{1}_{E_{\sigma}}\|_{\mathcal{W}^{2s,1}(\mathbb{G}_{\alpha}^{n})}\rightarrow0\ \mbox{as}\ k\rightarrow\infty.$$
 Hence, we obtain from Definition \ref{perimeter} that
$$P_{s}^{\mathcal{L}}(E_{\sigma})\leq \lim_{k\rightarrow\infty}
\|\mathcal{L}^{s}u_{k}\|_{L^{1}(\mathbb{G}^{n}_{\alpha})}=\|\mathcal{L}^{s}\mathbf{1}_{E_{\sigma}}\|_{L^{1}(\mathbb{G}^{n}_{\alpha})}.$$
This completes the proof of (\ref{coarea}).

To prove the last statement (\ref{coarea1}), by using (\ref{eq}), we conclude that
$$N_{1,1}^{\mathcal{L},2s}(u)=\frac{\Gamma(1-s)}{s}\int_{\mathbb{R}}\|\mathcal{L}^{s}\mathbf{1}_{E_{\sigma}}
\|_{L^{1}(\mathbb{G}^{n}_{\alpha})}d\sigma=\int_{\mathbb{R}}P_{s}^{\mathcal{L},*}(\{u>\sigma\})d\sigma.$$
\end{proof}

In closing we mention that, as a special case of their celebrated works, in \cite[Corollary 5]{BBM},
 or in \cite{MS}, Bourgain, Brezis and Mironescu obtained a new characterisation
of BV functions and  of De Giorgi's perimeters, based on their two
sided bounds,
$$C_{1}P(E)\leq\liminf_{s\nearrow1/2}(1/2-s)P_{s}(E)\leq\limsup_{s\nearrow1/2}(1/2-s)P_{s}(E)\leq C_{2}P(E).$$

The next result provides   two sided bounds for the limiting case
$s=1/2$ similar to the Bourgain, Brezis and Mironescu's bounds.
\begin{proposition}\label{BBM1}
For any measurable set $E\subset \mathbb{G}_{\alpha}^{n},$ we have
$$\liminf_{s\nearrow1/2}(1-2s)P_{s}^{\mathcal{L},*}(E)\geq \liminf_{t\rightarrow0^{+}}\frac{2}{\sqrt{t}}
\|e^{-t\mathcal{L}}\mathbf{1}_{E}-\mathbf{1}_{E}\|_{L^{1}(\mathbb{G}_{\alpha}^{n})}.$$
Moreover, if $|E|<\infty,$ then we have
$$\limsup_{s\nearrow1/2}(1-2s)P_{s}^{\mathcal{L},*}(E)\leq \limsup_{t\rightarrow0^{+}}\frac{2}{\sqrt{t}}
\|e^{-t\mathcal{L}}\mathbf{1}_{E}-\mathbf{1}_{E}\|_{L^{1}(\mathbb{G}_{\alpha}^{n})}.$$
\end{proposition}
\begin{proof}
For any measurable set $E\subset \mathbb{G}_{\alpha}^{n}$ with finite $s$-perimeter$^{*}$, via choosing $u=\mathbf{1}_{E}$ in Proposition \ref{BBM}
and by using Definition \ref{perimeter1}, the desired result can be deduced from Proposition \ref{BBM}.
\end{proof}

The following theorem is readily derived from Theorem \ref{MS1}.
\begin{theorem}
Assume that $E$ has finite $s$-perimeter$^{*}$ for some $s\in(0,1/2),$ then
$$\lim_{s\rightarrow0^{+}}s P_{s}^{\mathcal{L},*}(E)=2|E|.$$
\end{theorem}

At the end of this section, we consider the semi-norm
$N_{p,q}^{\mathcal{L},\beta}(\cdot)$ with $p=1,q=\infty$ and
$\beta\in(0,1)$ (see (\ref{NP})). The definition of fractional
perimeters $P_{s,\infty}^{\mathcal{L}}$ is given as follows:
\begin{definition}\label{FP3}
Let $s\in (0,1/2).$ We say that a measurable set
$E\subset\mathbb{G}_{\alpha}^{n}$ has finite
$s$-perimeter$^{\infty}$ if $\mathbf{1}_{E}\in
B_{1,\infty}^{\mathcal{L},2s}(\mathbb{G}_{\alpha}^{n})$ and we
define the $s$-perimeter$^{\infty}$ associated to $\mathcal{L}$ as
$$P_{s,\infty}^{\mathcal{L}}(E):=N_{1,\infty}^{\mathcal{L},2s}(\mathbf{1}_{E}).$$
\end{definition}
For the case of the Dirichlet space satisfying the weak
Bakry-\'{E}mery condition,  the above fractional perimeter had been
investigated in \cite{ABCRST2020}. In Theorem \ref{main1}, we obtain
the isoperimetric inequality for the $s$-perimeter$^{\infty}$. It
should be noted that we   don't know if the Grushin space satisfies
the weak Bakry-\'{E}mery condition so far.

\subsection{Proofs of Theorems \ref{main1} and \ref{main2}}\label{sec-4.1}

In order to prove Theorem \ref{main1}, we need the following lemmas.
\begin{lemma}\label{non-incre}
Let $s\in(0,1)$ and $p\in[1,\infty).$
For any $u\in\mathscr{S}(\mathbb{G}^{n}_{\alpha}),$ the function
$t\mapsto\|\mathcal{L}^{s}e^{-t\mathcal{L}}u\|_{L^{p}(\mathbb{G}^{n}_{\alpha})}$
is non-increasing on $(0,\infty).$ Moreover,
$$\lim_{t\rightarrow0^{+}}\|\mathcal{L}^{s}e^{-t\mathcal{L}}u\|_{L^{p}(\mathbb{G}^{n}_{\alpha})}
=\sup_{t>0}\|\mathcal{L}^{s}e^{-t\mathcal{L}}u\|_{L^{p}(\mathbb{G}^{n}_{\alpha})}<\infty.$$
\end{lemma}
\begin{proof}
We set $t=\sigma+\tau,$ where $\sigma,\tau>0.$
It follows from Proposition \ref{G-Pro} (i) and Lemma \ref{co} that
$$\mathcal{L}^{s}e^{-t\mathcal{L}}u=e^{-\sigma\mathcal{L}}\mathcal{L}^{s}
e^{-\tau\mathcal{L}}u,$$
which, together with Proposition \ref{G-Pro} (iv) and Remark \ref{rem1} (i), implies
\begin{align*}
&\|\mathcal{L}^{s}e^{-t\mathcal{L}}u\|_{L^{p}(\mathbb{G}^{n}_{\alpha})}
=\|e^{-\sigma\mathcal{L}}\mathcal{L}^{s}
e^{-\tau\mathcal{L}}u\|_{L^{p}(\mathbb{G}^{n}_{\alpha})}\\
&\leq \|\mathcal{L}^{s}
e^{-\tau\mathcal{L}}u\|_{L^{p}(\mathbb{G}^{n}_{\alpha})}
=\|e^{-\tau\mathcal{L}}\mathcal{L}^{s}
u\|_{L^{p}(\mathbb{G}^{n}_{\alpha})}
\leq \|\mathcal{L}^{s}
u\|_{L^{p}(\mathbb{G}^{n}_{\alpha})}<\infty.
\end{align*}
This completes the proof.
\end{proof}

\begin{lemma}\label{Led}
Let $s\in(0,1).$ For any $u\in \mathscr{S}(\mathbb{G}^{n}_{\alpha})$
and for all $t>0,$
$$\|e^{-t\mathcal{L}}u-u\|_{L^{1}(\mathbb{G}^{n}_{\alpha})}\leq\frac{2t^{s}}{\Gamma(1+s)}
\sup_{\sigma>0}\|\mathcal{L}^{s}e^{-\sigma\mathcal{L}}u\|_{L^{1}(\mathbb{G}^{n}_{\alpha})}.$$
\end{lemma}
\begin{proof}
For any $u\in \mathscr{S}(\mathbb{G}^{n}_{\alpha})$ and for all
$t>0,$ we first notice  that
\begin{align}\label{P}
e^{-t\mathcal{L}}u(g)-u(g)=\int_{0}^{\infty}h_{s}(\sigma,t)\mathcal{L}^{s}e^{-\sigma\mathcal{L}}
u(g)d\sigma,
\end{align}
where the function $h_{s}(\sigma,t)$ has the property
\begin{align}\label{P1}
\int_{0}^{\infty}|h_{s}(\sigma,t)|d\sigma=\frac{2t^{s}}{\Gamma(1+s)}.
\end{align}
In fact,   using Lemma \ref{FF} and (\ref{R}), we obtain
\begin{align*}
e^{-t\mathcal{L}}u-u&=e^{-t\mathcal{L}}(\mathscr{I}_{2s}(\mathcal{L}^{s}u))-\mathscr{I}_{2s}(\mathcal{L}^{s}u)\\
&=\frac{\mathcal{L}^{s}}{\Gamma(s)}\int_{0}^{\infty}\sigma^{s-1}(e^{-(t+\sigma)\mathcal{L}}u
-e^{-\sigma\mathcal{L}}u)d\sigma\\
&=\frac{\mathcal{L}^{s}}{\Gamma(s)}\bigg(\int_{t}^{\infty}(\sigma-t)^{s-1}
e^{-\sigma\mathcal{L}}ud\sigma-\int_{0}^{\infty}\sigma^{s-1}e^{-\sigma\mathcal{L}}ud\sigma\bigg)\\
&=\int_{0}^{\infty}h_{s}(\sigma,t)\mathcal{L}^{s}e^{-\sigma\mathcal{L}}
u(g)d\sigma,
\end{align*}
where
$$h_{s}(\sigma,t)=\frac{1}{\Gamma(s)}\big(\mathbf{1}_{(t,\infty)}(\sigma)(\sigma-t)^{s-1}
-\mathbf{1}_{(0,\infty)}(\sigma)\sigma^{s-1}\big).$$ Note that the
validity of the above chain of identities is guaranteed by Proposition \ref{G-Pro} (v) and
Proposition \ref{lem4}. On the other hand,
via a simple calculation, we can obtain
$$\int_{0}^{\infty}|h_{s}(\sigma,t)|d\sigma=\frac{1}{\Gamma(s)}\int_{0}^{\infty}
\big|\mathbf{1}_{(t,\infty)}(\sigma)(\sigma-t)^{s-1}
-\mathbf{1}_{(0,\infty)}(\sigma)\sigma^{s-1}\big|d\sigma=\frac{2t^{s}}{s\Gamma(s)},$$
which proves (\ref{P1}), and therefore (\ref{P}).
Finally, we deduce from Minkowski's inequality that
\begin{align*}
\|e^{-t\mathcal{L}}u-u\|_{L^{1}(\mathbb{G}^{n}_{\alpha})}&\leq \int_{0}^{\infty}
|h_{s}(\sigma,t)|\|\mathcal{L}^{s}e^{-\sigma\mathcal{L}}u\|_{L^{1}(\mathbb{G}^{n}_{\alpha})}d\sigma\\
&\leq \sup_{\sigma>0}\|\mathcal{L}^{s}e^{-\sigma\mathcal{L}}u\|_{L^{1}(\mathbb{G}^{n}_{\alpha})}
\int_{0}^{\infty}|h_{s}(\sigma,t)|d\sigma\\
&=\frac{2t^{s}}{\Gamma(1+s)}
\sup_{\sigma>0}\|\mathcal{L}^{s}e^{-\sigma\mathcal{L}}u\|_{L^{1}(\mathbb{G}^{n}_{\alpha})}.
\end{align*}
This completes the proof of the lemma.
\end{proof}

We are now in a position to prove Theorem \ref{main1}.
\begin{proof}[Proof of Theorem \ref{main1}]
We divide the proof into three cases.

{\bf Case 1:} When $\mathscr{P}^{\mathcal{L}}_{s}=P^{\mathcal{L}}_{s}.$
If $P_{s}^{\mathcal{L}}(E)=\infty,$ then the desired conclusion is
obviously valid. Thus we assume that $$P_{s}^{\mathcal{L}}(E)<\infty.$$
By the hypothesis, there exists a sequence $\{u_{k}\}_{k\in\mathbb{N}}\subset \mathscr{S}(\mathbb{G}^{n}_{\alpha})$
such that
$$\|u_{k}-\mathbf{1}_{E}\|_{L^{1}(\mathbb{G}^{n}_{\alpha})}\rightarrow 0\ \ as\ \ k\rightarrow\infty,$$
which, together with Proposition \ref{G-Pro} (iv), derives $$\|e^{-t\mathcal{L}}u_{k}-e^{-t\mathcal{L}}
\mathbf{1}_{E}\|_{L^{1}(\mathbb{G}^{n}_{\alpha})}\rightarrow0.$$
Further, we can also obtain
$$\|e^{-t\mathcal{L}}u_{k}-u_{k}\|_{L^{1}(\mathbb{G}^{n}_{\alpha})}
\rightarrow\|e^{-t\mathcal{L}}\mathbf{1}_{E}-
\mathbf{1}_{E}\|_{L^{1}(\mathbb{G}^{n}_{\alpha})}\ \
\mbox{\textrm{as}}\ \ k\rightarrow\infty.$$ Combining Lemma
\ref{Led} with Lemma \ref{non-incre}, for all $k\in\mathbb{N}$,  we
conclude that
\begin{align*}
\|e^{-t\mathcal{L}}u_{k}-u_{k}\|_{L^{1}(\mathbb{G}^{n}_{\alpha})}
&\lesssim \frac{t^{s}}{\Gamma(1+s)}
\sup_{\sigma>0}\|\mathcal{L}^{s}e^{-\sigma\mathcal{L}}u_{k}\|_{L^{1}(\mathbb{G}^{n}_{\alpha})}\\
&=\frac{t^{s}}{\Gamma(1+s)}\lim_{\sigma\rightarrow0^{+}}\|\mathcal{L}^{s}e^{-\sigma\mathcal{L}}u_{k}\|_{L^{1}(\mathbb{G}^{n}_{\alpha})}.
\end{align*}
From Remark \ref{rem1} (i) and Proposition \ref{G-Pro} (v), we
conclue that
$$\|e^{-\sigma\mathcal{L}}\mathcal{L}^{s}u_{k}-\mathcal{L}^{s}u_{k}\|_{L^{1}(\mathbb{G}^{n}_{\alpha})}
\rightarrow0 \ \ as \ \ \sigma\rightarrow0^{+}.$$  Therefore,
\begin{align*}
\|e^{-t\mathcal{L}}u_{k}-u_{k}\|_{L^{1}(\mathbb{G}^{n}_{\alpha})}
&\lesssim \frac{t^{s}}{\Gamma(1+s)}\lim_{\sigma\rightarrow0^{+}}\|\mathcal{L}^{s}e^{-\sigma\mathcal{L}}u_{k}\|_{L^{1}(\mathbb{G}^{n}_{\alpha})}\\
&=\frac{t^{s}}{\Gamma(1+s)}\lim_{\sigma\rightarrow0^{+}}\|e^{-\sigma\mathcal{L}}\mathcal{L}^{s}u_{k}\|_{L^{1}(\mathbb{G}^{n}_{\alpha})}\\
&=\frac{t^{s}}{\Gamma(1+s)}\|\mathcal{L}^{s}u_{k}\|_{L^{1}(\mathbb{G}^{n}_{\alpha})},
\end{align*}
where we have used Lemma \ref{co} in the first equality.
 Taking the $\liminf_{k\rightarrow\infty}$ in the above
inequality, we have
$$\|e^{-t\mathcal{L}}\mathbf{1}_{E}-\mathbf{1}_{E}\|_{L^{1}(\mathbb{G}^{n}_{\alpha})}
\lesssim
\frac{t^{s}}{\Gamma(1+s)}\liminf_{k\rightarrow\infty}\|\mathcal{L}^{s}u_{k}\|_{L^{1}
(\mathbb{G}^{n}_{\alpha})},$$ which, together with Definition \ref{perimeter}, implies,
\begin{align}\label{equa1}
\|e^{-t\mathcal{L}}\mathbf{1}_{E}-\mathbf{1}_{E}\|_{L^{1}(\mathbb{G}^{n}_{\alpha})}
\lesssim \frac{t^{s}}{\Gamma(1+s)}P_{s}^{\mathcal{L}}(E).
\end{align}
On the other hand, by using the symmetry of $K_{t}(\cdot,\cdot)$
and (\ref{SC}), we have
\begin{align*}
\|e^{-t\mathcal{L}}\mathbf{1}_{E}-\mathbf{1}_{E}\|_{L^{1}(\mathbb{G}^{n}_{\alpha})}
&=\int_{E}\big(1-e^{-t\mathcal{L}}\mathbf{1}_{E}(g)\big)dg+\int_{E^{c}}e^{-t\mathcal{L}}\mathbf{1}_{E}(g)dg\\
&=\int_{E}\big(1-e^{-t\mathcal{L}}\mathbf{1}_{E}(g)\big)dg+\int_{E}e^{-t\mathcal{L}}\mathbf{1}_{E^{c}}(g)dg\\
&=2\Big(|E|-\int_{E}e^{-t\mathcal{L}}\mathbf{1}_{E}(g)dg\Big).
\end{align*}
It follows from (\ref{B}) and Lemma \ref{upper} that
$$\int_{E}e^{-t\mathcal{L}}\mathbf{1}_{E}(g)dg=\int_{E}\int_{E}K_{t}(g,g')dgdg'
\lesssim t^{-Q/2}|E|^{2},$$
which leads to
\begin{align}\label{equa2}
\|e^{-t\mathcal{L}}\mathbf{1}_{E}-\mathbf{1}_{E}\|_{L^{1}(\mathbb{G}^{n}_{\alpha})}
\gtrsim|E|-t^{-Q/2}|E|^{2}.
\end{align}
Combining (\ref{equa1}) with (\ref{equa2}),  we conclude that, for
all $t>0$,
$$|E|\lesssim\frac{t^{s}}{\Gamma(1+s)}P_{s}^{\mathcal{L}}(E)+t^{-Q/2}|E|^{2}.$$
If we minimize the right-hand side of the above inequality by
choosing $t$ satisfying
$$t=\Big(\frac{Q|E|^{2}\Gamma(1+s)}{2sP_{s}^{\mathcal{L}}(E)}\Big)^{2/(Q+2s)},$$
then the desired conclusion (\ref{Per}) can be proved for some
positive constant $C(Q,s)$ depending exclusively on $Q$ and $s.$

{\bf Case 2:} When $\mathscr{P}^{\mathcal{L}}_{s}=P^{\mathcal{L},*}_{s}.$
The desired conclusion (\ref{Per}) can be directly obtained from Proposition \ref{com} and {\bf Case 1}.

{\bf Case 3:} When $\mathscr{P}^{\mathcal{L}}_{s}=P^{\mathcal{L}}_{s,\infty}.$
From Definition \ref{Besov} (ii) and (\ref{ineq20}), we deduce that
$$\|e^{-t\mathcal{L}}\textbf{1}_{E}-\textbf{1}_{E}\|_{L^{1}(\mathbb{G}^{n}_{\alpha})}\leq t^{s}N_{1,\infty}^{\mathcal{L},2s}(\textbf{1}_{E})=t^{s}P_{s,\infty}^{\mathcal{L}}(E),$$
which, together with (\ref{equa2}), implies
$$|E|\lesssim t^{s}P_{s,\infty}^{\mathcal{L}}(E)+t^{-Q/2}|E|^{2}.$$
Minimising the function in the above inequality with respect to $t>0$, we easily
infer the desired conclusion (\ref{Per}) for some constant $C(Q,s)$ depending exclusively on
$Q$ and $s.$
\end{proof}

As an application of Theorem \ref{main1}, we can directly derive
Theorem \ref{main2} by exploiting a coarea formula (\ref{co-area})
and Theorem \ref{main1}. We also need the following auxiliary lemma.
\begin{lemma}\cite[p.50, Lemma 1.86]{MZ1997}\label{inequality}
Let $0<p\leq q\leq\infty$ and $\gamma\in\mathbb{R}.$ If $U\geq0$ is a non-increasing function
on $[0,\infty),$ then there exists a positive constant $C$ depending on $p,q$ and $\gamma$ such that
$$\bigg(\int_{0}^{\infty}\big(t^{\gamma}U(t)\big)^{q}\frac{dt}{t}\bigg)^{1/q}
\leq C\bigg(\int_{0}^{\infty}\big(t^{\gamma}U(t)\big)^{p}\frac{dt}{t}\bigg)^{1/p}.$$
\end{lemma}

\begin{proof}[Proof of Theorem \ref{main2}]
We first prove that (\ref{BEG}) holds when $\mathscr{P}^{\mathcal{L}}_{s}=P^{\mathcal{L}}_{s}.$
Assume that $u\in
\mathfrak{B}_{1,1}^{\mathcal{L},2s}(\mathbb{G}_{\alpha}^{n}).$ For
$\sigma\geq0,$ let
$$E_{\sigma}=\{g\in \mathbb{G}_{\alpha}^{n}:|u(g)|>\sigma\},$$ and let $U(\sigma)=|E_{\sigma}|.$
Thus, $U$ is non-increasing on $[0,\infty).$
In Lemma \ref{inequality}, we set
$$p=(Q-2s)/Q, \ q=1\ \ \mbox{and}\ \ \gamma=Q/(Q-2s).$$
Clearly, we have $$0<p=(Q-2s)/Q<1=q.$$
Therefore, there exists a positive constant $C$ depending on $Q$ and $s$ such that
$$\int_{0}^{\infty}\sigma^{{2s}/{(Q-2s)}}U(\sigma)d\sigma \leq C\bigg(\int_{0}^{\infty}U(\sigma)^{{(Q-2s)}/{Q}}d\sigma\bigg)^{Q/(Q-2s)}.$$
Using the above inequality and Theorem \ref{main1}, we deduce that
\begin{align*}
\|u\|_{L^{\frac{Q}{Q-2s}}(\mathbb{G}_{\alpha}^{n})}&=\bigg(\int_{\mathbb{G}_{\alpha}^{n}}
|u(g)|^{{Q}/{(Q-2s)}}dg\bigg)^{(Q-2s)/Q}\\
&=\bigg(\frac{Q}{Q-2s}\int_{0}^{\infty}\sigma^{{2s}/{(Q-2s)}}U(\sigma)d\sigma\bigg)^{(Q-2s)/Q}\\
&\leq \int_{0}^{\infty}U(\sigma)^{{(Q-2s)}/{Q}}d\sigma=\int_{0}^{\infty}|E_{\sigma}|^{{(Q-2s)}/{Q}}d\sigma\\
&\leq C(Q,s)^{-1}\int_{0}^{\infty}P_{s}^{\mathcal{L}}(E_{\sigma}).
\end{align*}
Lemma \ref{co-area} implies that
$$\int_{0}^{\infty}P_{s}^{\mathcal{L}}(E_{\sigma})\leq \frac{s}{\Gamma(1-s)}N_{1,1}^{\mathcal{L},2s}(|u|)
\leq \frac{s}{\Gamma(1-s)}N_{1,1}^{\mathcal{L},2s}(u),$$
which completes the proof of (\ref{BEG}).

To prove the remaining part of Theorem \ref{main2}.
Assume that $\mathscr{P}^{\mathcal{L}}_{s}=P^{\mathcal{L},*}_{s}.$
Firstly, let $u=\mathbf{1}_{E}$ in (\ref{BEG1}). It follows
from Definition \ref{perimeter1} that (\ref{Per}) holds.

Conversely, for any $0\leq
u\in C_{c}^{\infty}(\mathbb{G}_{\alpha}^{n}),$ by the coarea formula
(\ref{coarea1})
and (\ref{Per}), we have
$$N_{1,1}^{\mathcal{L},2s}(u)\gtrsim\int_{0}^{\infty}P^{\mathcal{L},*}_{s}(E_{\sigma})
d\sigma\gtrsim\int_{0}^{\infty}|E_{\sigma}|^{(Q-2s)/Q}d\sigma,$$
where $E_{\sigma}=\{g\in\mathbb{G}_{\alpha}^{n}:u(g)>\sigma\}.$ For any
$\sigma\in\mathbb{R},$ let
$$u_{\sigma}=\min\{\sigma,u\},\psi(\sigma)=\Big(\int_{\mathbb{G}_{\alpha}^{n}}\big(u_{\sigma}(g)\big)^{Q/(Q-2s)}dg\Big)^{(Q-2s)/Q}.$$
Obviously,
$$\lim_{\sigma\rightarrow\infty}\psi(\sigma)=\Big(\int_{\mathbb{G}_{\alpha}^{n}}|u(g)|^{Q/(Q-2s)}dg\Big)^{(Q-2s)/Q}.$$
We can check that $\psi$ is nondecreasing on $(0,\infty)$ and for
$t>0,$
$$0\leq\psi(\sigma+t)-\psi(\sigma)\leq\Big(\int_{\mathbb{G}_{\alpha}^{n}}|u_{\sigma+t}(g)-u_{\sigma}(g)|^{Q/(Q-2s)}dg\Big)^{(Q-2s)/Q}\leq t|E_{\sigma}|^{(Q-2s)/Q}.$$ Then $\psi$ is locally Lipschitz and $$\psi'(\sigma)\leq|E_{\sigma}|^{(Q-2s)/Q}\ \ \mbox{for a.e.} \ \ \sigma>0.$$
Therefore,
$$\Big(\int_{\mathbb{G}_{\alpha}^{n}}|u(g)|^{Q/(Q-2s)}dg\Big)^{(Q-2s)/Q}
=\int_{0}^{\infty}\psi'(\sigma)d\sigma\leq\int_{0}^{\infty}|E_{\sigma}|^{(Q-2s)/Q}d\sigma\lesssim N_{1,1}^{\mathcal{L},2s}(u).$$
Based on the above arguments, we deduce that
(\ref{Per})$\Leftrightarrow$(\ref{BEG1}) when $\mathscr{P}^{\mathcal{L}}_{s}=P^{\mathcal{L},*}_{s}.$
\end{proof}

\subsection{Proofs of Theorem \ref{Sobolev} and Corollary \ref{propo2.6}}\label{sec-4.2}
In this section, we first introduce the radial maximal
function related to the semigroup $\{e^{-t\mathcal{L}^{s}}\}_{t>0}$ when $s=1/2,$ and then use the
boundedness of this maximal function (see Lemma \ref{stro-weak}) and
a very important pointwise estimation (\ref{Frac-ineq}) to obtain
Hardy-Littlewood-Sobolev embeddings (see Proposition \ref{HLS}).
Finally, via Lemma \ref{FF}, we complete the proof of Theorem
\ref{Sobolev}.

For $u\in \mathscr{S}(\mathbb{G}^{n}_{\alpha}),$ the
radial maximal function related to the semigroup $\{e^{-t\sqrt{\mathcal{L}}}\}_{t>0}$ is defined by
\begin{align}\label{maximal}
\mathcal{M}u(g)=\sup_{t>0}|e^{-t\sqrt{\mathcal{L}}}u(g)|,\ \ g\in\mathbb{G}^{n}_{\alpha}.
\end{align}
Moreover, we denote $$u^{*}(g)=\sup_{z>0}\bigg|\frac{1}{z}\int_{0}^{z}e^{-s\mathcal{L}}u(g)ds\bigg|.$$

In the following results, we establish strong and weak estimates for
the radial maximal operator.
\begin{lemma}\label{stro-weak}
Let $p\in[1,\infty].$ The following statements hold:
\begin{itemize}
  \item [(i)] If $p=1,$ then for $u\in L^{1}(\mathbb{G}^{n}_{\alpha})$ and for any $\lambda>0,$
              $$|\{g\in\mathbb{G}^{n}_{\alpha}:\mathcal{M}u(g)>\lambda\}|\lesssim\lambda^{-1}\|u\|_{L^{1}(\mathbb{G}^{n}_{\alpha})}.$$
  \item [(ii)] If $p\in(1,\infty],$ then for any $u\in L^{p}(\mathbb{G}^{n}_{\alpha}),$
             $$\|\mathcal{M}u\|_{L^{p}(\mathbb{G}^{n}_{\alpha})}\lesssim\|u\|_{L^{p}(\mathbb{G}^{n}_{\alpha})}.$$
\end{itemize}
\end{lemma}
\begin{proof}
(i): Observe that the semigroup $\{e^{-t\mathcal{L}}\}_{t>0}$
is a contraction
semigroup on $L^{p}(\mathbb{G}^{n}_{\alpha}), p\in[1,\infty],$ strongly
continuous when $p\in[1,\infty)$ (see Proposition \ref{G-Pro} (iv) \& (v)).
Hence, we use \cite[Lemma 6]{D1956}
to deduce that
\begin{align}\label{claim}
|\{g\in\mathbb{G}^{n}_{\alpha}:u^{*}(g)>\lambda\}|\lesssim\frac{1}{\lambda}
\int_{\{g\in\mathbb{G}^{n}_{\alpha}:|u(g)|>\lambda/2\}}|u(g)|dg\lesssim\lambda^{-1}\|u\|_{L^{1}(\mathbb{G}^{n}_{\alpha})}
\end{align}
for any $\lambda>0.$

In what follows we claim that for any $g\in\mathbb{G}^{n}_{\alpha},$
\begin{align}\label{claim1}
\mathcal{M}u(g)\lesssim u^{*}(g).
\end{align}
Taking the claim for granted, from it and (\ref{claim}) we obtain
$$|\{g\in\mathbb{G}^{n}_{\alpha}:\mathcal{M}u(g)>\lambda\}|\lesssim
|\{g\in\mathbb{G}^{n}_{\alpha}:u^{*}(g)>\lambda\}|\lesssim\lambda^{-1}\|u\|_{L^{1}(\mathbb{G}^{n}_{\alpha})},$$
which implies the desired conclusion.

We are thus left with proving (\ref{claim1}).
Let $$\rho(t,z):=\frac{tz^{-3/2}}{\sqrt{4\pi}}e^{-{t^{2}}/{(4z)}},\ f(z):=\frac{1}{z}\int_{0}^{z}
e^{-s\mathcal{L}}u(g)ds.$$
We can write (\ref{possion}) (applied with $s=1/2$) as
$$e^{-t\sqrt{\mathcal{L}}}u(g)=\int_{0}^{\infty}\rho(t,z)\frac{d}{dz}(zf(z))dz.$$
From Proposition \ref{G-Pro} (iv), it is easy to see
$$|f(z)|\leq \|u\|_{L^{\infty}(\mathbb{G}^{n}_{\alpha})}.$$
Note that $z\rho(t,z)\rightarrow0$ as $z\rightarrow\infty.$
Therefore, we have
$$|e^{-t\sqrt{\mathcal{L}}}u(g)|=\bigg|\int_{0}^{\infty}z\frac{\partial\rho}{\partial z}(t,z)
f(z)dz\bigg|\leq u^{*}(g)\int_{0}^{\infty}z\bigg|\frac{\partial
\rho}{\partial z}(t,z)\bigg|dz.$$ Via a simple verification, it can
be deduced that
$$z\frac{\partial\rho}{\partial z}(t,z)=\frac{t^{2}}{z}\rho(t,z)-\frac{3}{2}\rho(t,z)$$ and
$$\int_{0}^{\infty}\rho(t,z)dz=1,\ \int_{0}^{\infty}\frac{t^{2}}{z}\rho(t,z)dz=2.$$
Furthermore, we conclude that
$$\int_{0}^{\infty}z\big|\frac{\partial \rho}{\partial z}(t,z)\big|dz\leq \int_{0}^{\infty}\frac{t^{2}}{z}\rho(t,z)dz+
\frac{3}{2}\int_{0}^{\infty}\rho(t,z)dz=\frac{7}{2},$$ which proves
$$|e^{-t\sqrt{\mathcal{L}}}u(g)|\leq\frac{7}{2}u^{*}(g),$$ and
therefore,  (\ref{claim1}) is valid.

(ii)  By using Proposition \ref{Poi} (ii) with $s=1/2$, we have
$$e^{-t\sqrt{\mathcal{L}}}:L^{\infty}(\mathbb{G}^{n}_{\alpha})\rightarrow L^{\infty}(\mathbb{G}^{n}_{\alpha})\ \mbox{with} \ \|e^{-t\sqrt{\mathcal{L}}}\|_{L^{\infty}(\mathbb{G}^{n}_{\alpha})\rightarrow L^{\infty}(\mathbb{G}^{n}_{\alpha})}\le1.$$
This implies that
$$\mathcal{M}:L^{\infty}(\mathbb{G}^{n}_{\alpha})\rightarrow L^{\infty}(\mathbb{G}^{n}_{\alpha})\ \mbox{with} \ \|\mathcal{M}\|_{L^{\infty}(\mathbb{G}^{n}_{\alpha})\rightarrow L^{\infty}(\mathbb{G}^{n}_{\alpha})}\le1.$$
The proof is complete  by combining (i) of Lemma \ref{stro-weak}
with the Marcinckiewicz  interpolation  theorem.
\end{proof}

According to (\ref{possion}) and Remark \ref{1/2}, we can derive the potential operator based on the semigroup $\{e^{-t\sqrt{\mathcal{L}}}\}_{t>0}$,
which is consistent with the definition in (\ref{R}). More precisely,
for any $u\in \mathscr{S}(\mathbb{G}^{n}_{\alpha})$ we have
\begin{align}\label{Frac}
\mathscr{I}_{\tilde{\alpha}}u(g)=\frac{1}{\Gamma(\tilde{\alpha})}
\int_{0}^{\infty}t^{\tilde{\alpha}-1}e^{-t\sqrt{\mathcal{L}}}u(g)dt.
\end{align}
Indeed, we deduce from (\ref{possion}) and Remark \ref{1/2} that
\begin{align*}
\int_{0}^{\infty}t^{\tilde{\alpha}-1}e^{-t\sqrt{\mathcal{L}}}u(g)dt
&=\frac{1}{\sqrt{4\pi}}\int_{0}^{\infty}t^{\tilde{\alpha}-1}\int_{0}^{\infty}
\frac{t}{z^{3/2}}e^{-{t^{2}}/{(4z)}}e^{-z\mathcal{L}}u(g)dzdt\\
&=\frac{1}{\sqrt{4\pi}}\int_{0}^{\infty}\frac{1}{z^{3/2}}\bigg(\int_{0}^{\infty}
t^{\tilde{\alpha}+1}e^{-{t^{2}}/{(4z)}}\frac{dt}{t}\bigg)e^{-z\mathcal{L}}u(g)dz\\
&=\frac{2^{\tilde{\alpha}-1}\Gamma({(\tilde{\alpha}+1)}/{2})}{\sqrt{\pi}}
\int_{0}^{\infty}z^{\tilde{\alpha}/2-1}e^{-z\mathcal{L}}u(g)dz\\
&=\frac{2^{\tilde{\alpha}-1}\Gamma({(\tilde{\alpha}+1})/{2})\Gamma({\tilde{\alpha}}/{2})}{\sqrt{\pi}}\mathscr{I}_{\tilde{\alpha}}u(g)\\
&=\Gamma(\tilde{\alpha})\mathscr{I}_{\tilde{\alpha}}u(g),
\end{align*}
where  we have used the fact
$$2^{2x-1}\Gamma(x)\Gamma(x+1/2)=\sqrt{\pi}\Gamma(2x) $$ in the last
equality.

Via Lemma \ref{stro-weak}, we also establish the following Hardy-Littlewood-Sobolev
 inequality for potential operators $\mathscr{I}_{\tilde{\alpha}}$,
which is the key step in the proof of Theorem \ref{Sobolev}.
\begin{proposition}\label{HLS}
Let $\tilde{\alpha}\in(0,Q),p\in[1,Q/\tilde{\alpha})$ and $1/p-1/q=\tilde{\alpha}/Q.$
Then the following statements hold:
\begin{itemize}
  \item [(i)] If $p=1,$ then there exists a positive constant $C,$ depending on $Q$ and $\tilde{\alpha},$ such that for any $u\in L^{1}(\mathbb{G}^{n}_{\alpha}),$
  $$\sup_{\lambda>0}\lambda|\{g\in\mathbb{G}^{n}_{\alpha}:|\mathscr{I}_{\tilde{\alpha}}u(g)|>\lambda\}|^{{(Q-\tilde{\alpha})}/{Q}}
  \leq C\|u\|_{L^{1}(\mathbb{G}^{n}_{\alpha})}.$$
  \item [(ii)] If $p\in(1,Q/\tilde{\alpha})$ then there exists a positive constant $C,$ depending on $Q,p$ and $\tilde{\alpha},$ such that for any $u\in L^{p}(\mathbb{G}^{n}_{\alpha}),$
  $$\|\mathscr{I}_{\tilde{\alpha}}u\|_{L^{q}(\mathbb{G}^{n}_{\alpha})}\leq C\|u\|_{L^{p}(\mathbb{G}^{n}_{\alpha})}.$$
\end{itemize}
\end{proposition}
\begin{proof}
We first prove that for any $u\in\mathscr{S}(\mathbb{G}^{n}_{\alpha})$ and $\varepsilon>0,$ there exist positive constants
$C(\tilde{\alpha})$ and $C(\tilde{\alpha},p,Q)$ such that
\begin{align}\label{Frac-ineq}
|\mathscr{I}_{\tilde{\alpha}}u(g)|\leq C(\tilde{\alpha})\mathcal{M}u(g)\varepsilon^{\tilde{\alpha}}+
C(\tilde{\alpha},p,Q)\|u\|_{L^{p}(\mathbb{G}^{n}_{\alpha})}
\varepsilon^{\tilde{\alpha}-{Q}/{p}},
\end{align}
where $\mathcal{M}$ is as in (\ref{maximal}) and $p\in[1,Q/\tilde{\alpha}).$

Indeed, for any $\varepsilon>0,$ we deduce from (\ref{Frac}) that
\begin{align*}
|\mathscr{I}_{\tilde{\alpha}}u(g)|\leq
I+II.
\end{align*}
where
\begin{align*}
  \left\{\begin{aligned}
  I&:=\frac{1}{\Gamma(\tilde{\alpha})}\int_{0}^{\varepsilon}
       t^{\tilde{\alpha}-1}|e^{-t\sqrt{\mathcal{L}}}u(g)|dt,\\
  II&:=\frac{1}{\Gamma(\tilde{\alpha})}\int_{\varepsilon}^{\infty}
       t^{\tilde{\alpha}-1}|e^{-t\sqrt{\mathcal{L}}}u(g)|dt.
  \end{aligned}\right.
\end{align*}
The first term $I$ is controlled by the estimate
$$I\leq \frac{1}{\Gamma(1+\tilde{\alpha})}\mathcal{M}u(g)\varepsilon^{\tilde{\alpha}}=C(\tilde{\alpha})\mathcal{M}u(g)
\varepsilon^{\tilde{\alpha}}.$$
For the second term $II,$ we know from Proposition \ref{Poi} (iii) (applied with $s=1/2$) that
$$|e^{-t\sqrt{\mathcal{L}}}u(g)|\leq C(p,Q) t^{-Q/p}\|u\|_{L^{p}(\mathbb{G}^{n}_{\alpha})}.$$
Therefore, we obtain
$$II\leq \frac{C(p,Q)}{\Gamma(\tilde{\alpha})}\|u\|_{L^{p}(\mathbb{G}^{n}_{\alpha})}\int_{\varepsilon}^{\infty}
t^{\tilde{\alpha}-1-Q/p}dt=C(\tilde{\alpha},p,Q)\|u\|_{L^{p}(\mathbb{G}^{n}_{\alpha})}
\varepsilon^{\tilde{\alpha}-{Q}/{p}},$$
where we have used the fact that $p\in[1,Q/\tilde{\alpha}).$
This proves (\ref{Frac-ineq}).

In what follows we prove that (i) holds. Without loss of generality,
we assume that $u\in L^{1}(\mathbb{G}^{n}_{\alpha})$ with
$\|u\|_{L^{1}(\mathbb{G}^{n}_{\alpha})}\neq0.$ For any given
$\lambda>0,$ we choose $\varepsilon>0$ such that
$$C(\tilde{\alpha},1,Q)\|u\|_{L^{1}(\mathbb{G}^{n}_{\alpha})}
\varepsilon^{\tilde{\alpha}-Q}=\lambda.$$ Therefore, it follows from
(\ref{Frac-ineq}) with $p=1$ and Lemma \ref{stro-weak} (i) that
\begin{align*}
\Big|\Big\{g\in\mathbb{G}^{n}_{\alpha}:\ |\mathscr{I}_{\tilde{\alpha}}u(g)|>2\lambda\Big\}\Big|&\leq
\Big|\Big\{g\in\mathbb{G}^{n}_{\alpha}:\ C(\tilde{\alpha})\mathcal{M}u(g)\varepsilon^{\tilde{\alpha}}>\lambda\Big\}\Big|\\
&\leq \frac{C(\tilde{\alpha})\varepsilon^{\tilde{\alpha}}}{\lambda}\|u\|_{L^{1}(\mathbb{G}^{n}_{\alpha})}.
\end{align*}
Note that $$\varepsilon^{\tilde{\alpha}}=\Big(\frac{C(\tilde{\alpha},1,Q)\|u\|_{L^{1}(\mathbb{G}^{n}_{\alpha})}}{\lambda}\Big)
^{\tilde{\alpha}/(Q-\tilde{\alpha})}.$$ This proves (i).

To prove (ii), we assume that $p\in(1,Q/\tilde{\alpha}).$
We minimize the right-hand side of the inequality (\ref{Frac-ineq}) by choosing $\varepsilon$ satisfying
$$\varepsilon=\bigg(\frac{(Q-\tilde{\alpha}p)C(\tilde{\alpha},p,Q)\|u\|_{L^{p}(\mathbb{G}^{n}_{\alpha})}}
{\tilde{\alpha}pC(\tilde{\alpha})\mathcal{M}u(g)}\bigg)^{p/Q}$$ and
for this choice of $\varepsilon, $  we obtain
$$|\mathscr{I}_{\tilde{\alpha}}u(g)|\leq C(\tilde{\alpha},p,Q)\mathcal{M}u(g)^{1-{\tilde{\alpha}p}/{Q}}
\|u\|_{L^{p}(\mathbb{G}^{n}_{\alpha})}^{{\tilde{\alpha}p}/{Q}}.$$
Combining this with Lemma \ref{stro-weak} (ii), we conclude that
 Proposition \ref{HLS} (ii) is valid.
\end{proof}

\begin{proof}[Proof of Theorem \ref{Sobolev}]
From Lemma \ref{FF}, we know that for any $g\in\mathbb{G}^{n}_{\alpha},$
$$|u(g)|=|\mathscr{I}_{2s}(\mathcal{L}^{s}u(g))|.$$
We use the above equality and   Proposition \ref{HLS} to obtain the
desired result.
\end{proof}

\begin{proof}[Proof of Corollary \ref{propo2.6}]
The desired result can be directly derived from Corollary \ref{BW} and Theorem \ref{Sobolev}.
\end{proof}

%
%
%
%
%
%


%
%



\end{document}